\documentclass[reqno]{amsart}
\usepackage{amsmath,amssymb,stmaryrd}
\usepackage{amsrefs}
\usepackage{amsfonts}

\newtheorem{step}{Step}

\usepackage[colorlinks=true, pdfborder={0 0 0}]{hyperref}
\usepackage{cleveref}
\usepackage{upref}

\newtheorem{theorem}{Theorem}[section]
\newtheorem{defi}{Definition}[section]
\newtheorem{proposition}{Proposition}[section]

\newtheorem{lemma}{Lemma}[section]

\def \rr {\mathbb{R}}
\def \rn {\mathbb{R}^n}
\def \nn {\mathbb{N}}
\def \rn {\mathbb{R}^n}

\def \snmoinsun {\mathbb{S}^{n-1}}
\def \eps {\epsilon}
\def \crit {2^\star}
\def \Pl {{\mathcal P}_\lambda}
\def \Mmx {M\setminus\{x_0\}}
\def \ua {u_\alpha}
\def \tua {\tilde{u}_\alpha}
\def \bua {\bar{u}_\alpha}
\def \hua {\hat{u}_\alpha}
\def \ga {g_\alpha}
\def \va {v_\alpha}
\def \da {\delta_\alpha}
\def \tga {\tilde{g}_\alpha}
\def \bga {\bar{g}_\alpha}
\def \xa {x_\alpha}
\def \ma {\mu_\alpha}

\def \ra {r_\alpha}
\def \wa {w_\alpha}
\def \ea {\epsilon_\alpha}
\def \tpl {\tilde{\psi}_\lambda}
\def \phia {\varphi_\alpha}
 
\title[Localization of bubbling for high order nonlinear equations]{Localization of bubbling for high order nonlinear equations}

\author{Fr\'ed\'eric Robert}
\address{Fr\'ed\'eric Robert, Institut \'Elie Cartan, Universit\'e de Lorraine, BP 70239, F-54506 Vand{\oe}uvre-l\`es-Nancy, France}
\email{frederic.robert@univ-lorraine.fr}

\date{December 3rd, 2024}


 \addtocontents{toc}{\protect\setcounter{tocdepth}{1}}

\begin{document}
\begin{abstract} We analyze the asymptotic pointwise behavior of families of solutions to the high-order critical equation
$$P_\alpha\ua=\Delta_g^k\ua +\hbox{lot}=|\ua|^{\crit-2-\ea}\ua\hbox{ in }M$$
that behave like
$$\ua=u_0+B_\alpha+o(1)\hbox{ in }H_k^2(M)$$
where $B=(B_\alpha)_\alpha$ is a Bubble, also called a Peak. We give obstructions for such a concentration to occur: depending on the dimension, they involve the mass of the associated Green's function or the difference between $P_\alpha$ and the conformally invariant GJMS operator. The bulk of this analysis is the proof of the pointwise control
\begin{equation*}
|\ua(x)|\leq C\Vert u_0\Vert_\infty^{(\crit-1)^2}+C\left(\frac{\ma^{2}}{\ma^{2 }+d_g(x,\xa)^{2 }}\right)^{\frac{n-2k}{2}}\hbox{ for all }x\in M\hbox{ and }\alpha\in\nn,
\end{equation*}
where $|\ua(\xa)|=\max_M|u_\alpha|\to +\infty$ and $\ma:=|\ua(\xa)|^{-\frac{2}{n-2k}}$. The key to obtain this estimate is a sharp control of the Green's function for elliptic operators involving a Hardy potential.
\end{abstract}
\maketitle
\centerline{\it Pour Mathilde et Daniel}\medskip\noindent

\tableofcontents
\section{Introduction and statement of the main results}
Let $(M,g)$ be a smooth compact Riemannian manifold of dimension $n\geq 3$ without boundary, and let $k\in\nn$ be such that $2\leq 2k<n$. We consider families of function $(\ua)_{\alpha>0}\in C^{2k}(M)$ that are solutions to 
\begin{equation}\label{eq:ua}
P_\alpha u_\alpha=|u_\alpha|^{\crit-2-\ea}u_\alpha\hbox{ in }M,
\end{equation}
where $\crit:=\frac{2n}{n-2k}$, $(\ea)_\alpha\in [0,\crit-2)$ is such that $\lim_{\alpha\to +\infty}\ea=0$ and $(P_\alpha)_{\alpha>0}$ is a family of general elliptic symmetric differential operators of the type
$$P:=\Delta_g^k+\sum_{i=0}^{k-1} (-1)^i\nabla^i(A^{(i)} \nabla^i)\hbox{ for }\alpha>0$$
where $\Delta_g:=-\hbox{div}_g\nabla$ is the Riemannian Laplacian with minus-sign convention and for all $i=0,...,k-1$, $ A^{(i)}\in C^{i}_{\chi}(M,\Lambda_{S}^{(0,2i)}(M))$ is a $(0,2i)-$tensor field of class $C^i$ on $M$ such that for any $i$ , $A^{(i)}(T,S)=A^{(i)}(S,T)$ for any $(i,0)-$tensors $S$ and $T$. In coordinates, we mean that $\nabla^i(A^{(i)} \nabla^i)=\nabla^{q_i\cdots q_1}(A_{p_1\cdots p_i,q_1\cdots q_i}\nabla^{p_1\cdots p_i})$.

\begin{defi}\label{def:scc} We say that $(P_\alpha)_{\alpha>0}\to P_0$ is of type (SCC) if $(P_\alpha)_\alpha$ and $P_0$  are differential operators such that $P_\alpha:=\Delta_g^k+\sum_{i=0}^{k-1} (-1)^i\nabla^i(A^{(i)}_\alpha \nabla^i)$ for all $\alpha>0$ and $P_0:=\Delta_g^k+\sum_{i=0}^{k-1} (-1)^i\nabla^i(A^{(i)}_0 \nabla^i)$ are as above and there exists $\theta\in (0,1)$ such that \par
\noindent$\bullet$ for all $i=0,...,k-1$, $(A^{(i)}_\alpha)_{\alpha>0}, A^{(i)}_0\in C^{i,\theta}$ and $\lim_{\alpha\to 0}A_\alpha^{(i)}=A_0^{(i)}$ in $C^{i, \theta}$. \par
\noindent$\bullet$  $(P_\alpha)_{\alpha}$ is uniformly coercive in the sense that there exists $c>0$ such that
$$\int_M u P_\alpha u\, dv_g=\int_M (\Delta_g^{\frac{k}{2}}u)^2\, dv_g+\sum_{i=0}^{k-1}\int_M A^{(i)}_\alpha(\nabla^iu,\nabla^i u)\, dv_g\geq c\Vert u\Vert_{H_k^2}^2$$
for all $u\in H_k^2(M)$, where $H_k^2(M)$ is the completion of $C^\infty(M)$ for the norm $u\mapsto\Vert u\Vert_{H_k^2}:= \sum_{i=0}^{k}\Vert\nabla^iu\Vert_2$. When $k=2l+1$ is odd, we have set $(\Delta_g^{\frac{k}{2}}u)^2=|\nabla \Delta_g^l u|^2_g$.
\end{defi}

\smallskip\noindent Equation \eqref{eq:ua} arises in conformal geometry: indeed, it relates Branson's $Q-$curvatures of two metrics in the same conformal class, see \cite{branson}. The model for such operators is the conformal Graham-Jenne-Mason-Sparling (GJMS) operator  
\begin{equation}\label{def:gjms}
P_g^k=\Delta_g^k+\sum_{i=0}^{k-1} (-1)^i\nabla^i(A_{g}^{(i)} \nabla^i),
\end{equation}
that enjoys invariance under conformal changes of metrics. Namely, for $\omega\in C^\infty(M)$, $\omega>0$, taking $\tilde{g}:=\omega^{\frac{4}{n-2k}}g$ a metric conformal to $g$, we have that
\begin{equation}\label{invar:gjms}
P_{\tilde{g}}^k\varphi=\omega^{1-\crit}P_g^k(\omega\varphi)\hbox{ for all }\varphi\in C^\infty(M).
\end{equation}
The GJMS operator is a fundamental operator in conformal geometry: we refer to Paneitz \cite{paneitz}, Branson \cite{branson}, Graham-Jenne-Mason-Sparling  \cite{gjms}, Fefferman-Graham \cite{FG} for more considerations on the subject.  Families of solutions $(\ua)_\alpha\in C^{2k}(M)$ to \eqref{eq:ua} might blow-up and develop singularities along bubbles as in Definitions \ref{def:std:bubble} and \ref{def:bubble}. Our first results shows that there is no blow-up in high dimensions when the limiting operator $P_0$ is "below" the GJMS operator, while the nonexistence persists in small dimensions when the weak limit is positive.

\begin{theorem}\label{th:teaser:2}[Standard bubble] Let $(M,g)$ be a compact Riemannian manifold of dimension $n$ without boundary and let $k\in\nn$ be such that $2\leq 2k<n$. Let $u_0\in C^{2k}(M)$ be a function (possibly sign-changing), $x_0\in M$ be a point, and $(P_\alpha)_{\alpha>0}\to P_0$ be of type (SCC). Assume that we are in one of the following situations:
\begin{itemize}
\item $2k<n<2k+4$ and $  u_0(x_0)>0$;
\item $n=2k+4$, $Tr_{x_0}(A^{(k-1)}_0-A_{g}^{(k-1)}) < \kappa_k   u_0(x_0)$ (see \eqref{def:Tr:A} and \eqref{def:kappa});
\item $n> 2k+4$ and $Tr_{x_0}(A_0^{(k-1)}-A_{g}^{(k-1)})<0$,
\end{itemize}
where $A_{g}^{(k-1)}$ is the coefficient of the term of order $2k-2$ in the GJMS operator \eqref{def:gjms}.
Then there is \underline{no} family $(\ua)_{\alpha>0}\in C^{2k}(M)$ such that 
$$P_\alpha \ua=|u_\alpha|^{\crit-2-\ea}u_\alpha\hbox{ in }M,\; (\ea)_\alpha\in [0,\crit-2)$$
and 
\begin{equation}
\ua=u_0+  B_\alpha^++o(1)\hbox{ in }H_k^2(M)\hbox{ as }\alpha\to 0\label{eq:B}
\end{equation}
where $(B_\alpha^+)_\alpha$ is a standard bubble as in \eqref{def:pos:bubble} centered at  $\xa\to x_0\in M$ and $\ea\to 0$ in $\rr$ as $\alpha\to 0$. 
\end{theorem}
Theorem \ref{th:teaser:2} is a consequence of the more general Theorem \ref{th:general} of Section \ref{sec:gene}. In the above theorem, the trace of a $(0,2p)-$tensor $A$ is 
\begin{equation}\label{def:Tr:A}
Tr_{x_0}(A):= g^{i_1i_2}\cdots g^{i_{2p-1}i_{2p}}A^S_{i_1i_2\cdots i_{2p}}\hbox{ where } A^S_{i_1...i_{2p}}:=\frac{1}{|\mathcal{S}_{2p}|}\sum_{\sigma\in \mathcal{S}_{2p}}A_{i_{\sigma(1)}\cdots i_{\sigma(2p)}}.
\end{equation}
When $p=1$, one recovers the usual trace of a $(0,2)-$tensor.

\smallskip\noindent When $k=1$ and $\ua>0$ for all $\alpha>0$, Theorem \ref{th:teaser:2} is a consequence of Druet \cite{druet:jdg}. The obstructions to single blow-up above is almost optimal. When the strict inequalities are reversed, Robert-V\'etois \cite{RV} ($k=1$, $\kappa_1=2$) and Bakri-Casteras \cite{bakri-casteras} ($k=2$, $\kappa_2=2\sqrt{30}$) proved the existence of sign-changing blowing-up solutions (under some technical additional assumptions though). When $\ua>0$, there are several blow-up existence results when the hypotheses of Theorem \ref{th:teaser:2} are not satisfied, see for instance Micheletti-Pistoia-V\'etois \cite{mpv} and Robert-V\'etois \cite{RV.imrn}. Other profiles are in Premoselli \cite{premo-jgea} and Premoselli-V\'etois \cites{premo-vetois-jmpa,premo-vetois-aim,premo-vetois-eigen}. When $k=2$, $n>8$, $u_0\equiv 0$ and $ A_\alpha^{(1)}\equiv A_{g}^{(1)}$, Pistoia-Vaira \cite{pv} proved also the existence of families of solutions to \eqref{eq:ua} that blow-up like \eqref{eq:B} for suitable families $(P_\alpha)$.

\medskip\noindent Our next result concerns general bubbles that can be sign-changing. We also assume that $u_0\equiv 0$, which allows to get a more precise statement involving the mass of the limiting operator $P_0$. For any coercive symmetric operator $P_0$ of order $2k$, let $G_0$ be its Green's function. When $n=2k+1$, see Subsection \ref{subsec:mass}, for all $x_0\in M$, there exists $m_{P_0}(x_0)\in\rr$ (refered to as the mass) such that for $C_{n,k}$ as in \eqref{def:Cnk} 
\begin{equation}\label{def:mass:intro}
G_0(x,x_0)=\frac{C_{n,k}}{d_g(x,x_0)^{n-2k}}+m_{P_0}(x_0)+o(1)\hbox{ as }x\to x_0,\, x\in M-\{x_0\}.
\end{equation}
\begin{theorem}\label{th:teaser:3}[$u_0\equiv 0$] Let $(M,g)$ be a compact Riemannian manifold of dimension $n$ without boundary and let $k\in\nn$ be such that $2\leq 2k<n$. Let $x_0\in M$ be a point,  $(P_\alpha)_{\alpha>0}\to P_0$ be of type (SCC). Assume that we are in one of the following situations:
\begin{itemize}
\item $n=2k+1$ and $m_{P_0}(x_0)>0$;
\item $n\geq 2k+2$, $\hbox{Weyl}_g(x_0)=0$, and $(A_0^{(k-1)}-A_{g}^{(k-1)})(x_0)<0$ in the sense  that  $(A_0^{(k-1)}-A_{g}^{(k-1)})(x_0)(S,S)<0$ for all $(k-1,0)-$tensor $S\neq 0$ such that $S_{i_1...i_{k-1}}$ is independent of the order of the indice,
\end{itemize}
where $A_{g}^{(k-1)}$ is the coefficient of the term of order $2k-2$ in the GJMS operator \eqref{def:gjms}. Then there is \underline{no} family $(\ua)_{\alpha>0}\in C^{2k}(M)$ such that 
$$P_\alpha\ua=|u_\alpha|^{\crit-2-\ea}u_\alpha\hbox{ in }M,\; (\ea)_\alpha\in [0,\crit-2)\to 0$$
and 
$$\ua=B_\alpha+o(1)\hbox{ in }H_k^2(M)\hbox{ as }\alpha\to 0$$
for a Bubble $B=(B_\alpha)_\alpha$ as in Definition \ref{def:bubble} with parameter $\xa\to x_0\in M$ and $\int_{\rn}|U|^{\crit-2}U\, dx\neq 0$.
\end{theorem}
This theorem is a consequence of the more general Theorem \ref{th:no:unull}. When $k=1$, this result is a particular case of Premoselli-Robert \cite{PR}. 
These two theorems rely on a sharp pointwise control of the blow-up:
\begin{theorem}\label{th:estim-co:intro} Let $(M,g)$ be a compact Riemannian manifold of dimension $n$ without boundary and let $k\in\nn$ be such that $2\leq 2k<n$. We consider a family of operators $(P_\alpha)_{\alpha>0}\to P_0$ of type (SCC) and a family of functions $(\ua)_\alpha\in C^{2k}(M)$ such that \eqref{eq:ua} holds for all $\alpha>0$.
We assume that there exists $u_0\in C^{2k}(M)$ (possibly sign-changing) and a Bubble $B=(B_\alpha)_\alpha$ with parameters $(\xa)_\alpha\in M$ and  $(\ma)_\alpha\in\rr_{>0}$ as in Definition \ref{def:bubble} such that
$$\ua=u_0+B_\alpha+o(1)\hbox{ in }H_k^2(M)\hbox{ as }\alpha\to 0.$$
Then there exists $C>0$ such that
\begin{equation}
|\ua(x)|\leq C\Vert u_0\Vert_\infty^{(\crit-1)^2}+C\left(\frac{\ma^{2}}{\ma^{2 }+d_g(x,\xa)^{2 }}\right)^{\frac{n-2k}{2}}\label{est:opti} \end{equation}
for all $x\in M$ and $\alpha>0$.
\end{theorem}
Pointwise controls like \eqref{est:opti} have already been obtained for $k=1$, that is when $P_\alpha:=\Delta_g+h_\alpha$ where $(h_\alpha)_\alpha\in C^0(M)$ is a potential. They have their origins in Atkinson-Peletier \cite{AP}, Z.-C.Han \cite{zchan}, Hebey-Vaugon \cite{HV}. Still for $k=1$, the control for arbitrary high energy is in Druet-Hebey-Robert \cite{DHR}: we refer to Druet-Hebey \cite{DH} (see \cite{hebey.eth} for an exposition in book form). Recently, Premoselli \cite{premoselli} has developed a promising approach by being able to "neglict" the critical nonlinearity in a suitable sense. In the case of unbounded energy, the method of simple blowup points has been initiated by Schoen \cite{schoen1988} to get controls like \eqref{est:opti}. Such methods have turned to be paramount to get compactness and stability results for second-order equations: see Schoen-Zhang \cite{sz}, Li-Zhu \cite{LiZhu}, Li-Zhang \cite{LiZha3}, Druet \cites{druet:jdg,druet:imrn}, Marques \cite{Mar} and Khuri-Marques-Schoen \cite{kms}. 

\smallskip\noindent When $k=1$, the methods used in the references above strongly rely on the specificities of second-order operators, namely the Harnack inequality, Hopf's maximum principle or positivity preserving properties. These results do not hold in full generality for higher-order problems, that is for $k>1$. We refer to Gazzola-Grunau-Sweers \cite{GGS} for a general exposition about the lack of comparison principle for higher-order problems. 

\medskip\noindent For $k= 2$, there has been several attempts to get controls in the spirit of \eqref{est:opti}. Early estimates are in Hebey \cite{hebey:2003} and Felli-Hebey-Robert \cite{FHR}. In \cite{lx}, Li-Xiong  have obtained a beautiful compactness results via a sharp pointwise control: here, the operator $P_\alpha=P^2_{g}$ is the geometric Paneitz operator \cite{paneitz} (that is the GJMS operator for $k=2$), and the solutions are geometric, that is $u_\alpha>0$. In \cites{HRW,HR}, Hebey, Wen and the author have obtained a series of compactness results for specific operators $P_\alpha$ and positive solutions $u_\alpha>0$. These methods extend to the case $k\geq 3$.  Let us note that Gursky-Malchiodi \cite{GM} have proved a remarkable comparison principle for the geometric Paneitz operator $P^2_{g}$ on a compact Riemannian manifold. More recently, Carletti \cite{carletti} has developed a sharp pointwise asymptotic analysis relying on Premoselli's \cite{premoselli} to prove the attainability of best constants in high-order Sobolev inequalities.

\medskip\noindent In the present work, we get a general control that is independent of any geometric hypothesis, that is for general elliptic operators $(P_\alpha)$ and functions $(u_\alpha)$ that might change sign. The main point is to transform the nonlinear equation
$$P_\alpha u_\alpha=|u_\alpha|^{\crit-2}u_\alpha\hbox{ in }M$$
into a linear equation
$$(P_\alpha -V_\alpha)u_\alpha=0\hbox{ in }M,\hbox{ where }V_\alpha:=|u_\alpha|^{\crit-2}.$$
It turns out that $(V_\alpha)$ behaves like a Hardy potential. We use Green's representation formula to express $u_\alpha$ via the Green's function for $P_\alpha -V_\alpha$. In order to get  \eqref{est:opti}, we need pointwise controls for this Green's function that are the object of Theorem \ref{th:Green:pointwise:BIS}. When $k=1$, and with more informations on the Hardy potential, some related estimates are in the joint work \cite{GR.CalcVar} of Ghoussoub and the author: they rely on second-order operator properties. In order to bypass these restrictions, we prove and use the general regularity Lemma \ref{lem:main} that is a pointwise control valid for "small" Hardy potentials and any order $k\geq 1$. The conclusions of Theorem \ref{th:estim-co:intro} is a consequences of these pointwise estimates. Theorem \ref{th:estim-co:intro} allows to get the precise behavior of $\ua$ at any scale (see Propositions \ref{prop:cv:1} and \ref{prop:cv:2}). Theorems  \ref{th:teaser:2} and \ref{th:teaser:3} are consequences of this behavior and the high-order Pohozaev-Pucci-Serrin identity. 

\medskip\noindent This paper is organized as follows. In Section \ref{sec:bubble}, we define the bubbles and we give a few of their properties. Section \ref{sec:gene} is devoted to generalizations or alternative versions of the theorems above: Theorems \ref{th:teaser:2} and \ref{th:teaser:3} are direct consequence of Theorems \ref{th:general} and \ref{th:no:unull}. First properties of solutions to \eqref{eq:ua} that behave like a single bubble are proved in Section \ref{sec:prelim}. The proof of the main Theorem \ref{th:estim-co:intro} is performed in Section \ref{sec:strong} assuming some properties of Green's functions. The core of the paper is the proof of the pointwise estimates of Green's function with Hardy potential that are proved in Sections \ref{sec:green1} to \ref{sec:lemma}. In Sections \ref{sec:11}, we apply Theorem \ref{th:estim-co:intro} to obtain a sharp control of solutions to \eqref{eq:ua}. We prove a Pohozaev-Pucci-Serrin identity in Section \ref{sec:poho} to which we apply Theorem \ref{th:estim-co:intro} to prove Theorems \ref{th:general} and \ref{th:no:unull} in Section \ref{sec:appli}. In the Appendices, we perform an expansion of the powers of the Laplacian, we prove a Hardy inequality on compact manifolds and we state some regularity theorems that are consequences of Agmon-Douglis-Nirenberg \cite{ADN}. We also prove the existence and uniqueness of the Green's function for arbitrary high order and smooth coefficients.

\smallskip{\it Notations:} In the sequel, $C(a,b,...)$ will denote a constant depending on $(M,g)$, $a,b,...$. The value can change from one line to another, and even in the same line. Given two sequence $(a_\alpha)_\alpha$ and $(b_\alpha)_\alpha$, we will write $a_\alpha\asymp b_\alpha$ if there existe $c_1,c_2>0$ such that $c_1 a_\alpha\leq b_\alpha\leq c_2 a_\alpha$ for all $\alpha>0$.

\section{The notion of bubble and the product of the Weyl tensor and a bubble}\label{sec:bubble}
In the sequel, we fix $r_g\in (0, i_g(M))$ where $i_g(M)>0$ is the injectivity radius of $(M,g)$. The basic model of bubble is the standard bubble:
\begin{defi}\label{def:std:bubble} We say that a family $B=(B_\alpha^+)_\alpha\in H_k^2(M)$ is a \underline{standard bubble} centered at  $(\xa)_\alpha\in M$ with radius $(\ma)_\alpha\in \rr_{>0}$ if there exists a cutoff function $\eta$ such that
\begin{equation}\label{def:pos:bubble}
B_\alpha^+(x):=\eta(d_g(x,\xa))\left(\frac{\ma}{\ma^2+a_{n,k} d_g(x,\xa)^2}\right)^{\frac{n-2k}{2}}+R_\alpha(x)\hbox{ for all }x\in M
\end{equation}
where $\eta(t)=1$ if $t\leq r_g/2$ and $\eta(t)=0$ if $t\geq r_g$, and $\lim_{\alpha\to 0}R_\alpha=0$ in $H_k^2(M)$.
\end{defi}
Standard bubbles are modeled on
\begin{equation}\label{def:sol:pos}
X\mapsto U_{\mu, x_0}(X):=\left(\frac{\mu}{\mu^2+a_{n,k}|X-X_0|^2}\right)^{\frac{n-2k}{2}}\hbox{ for }X\in\rn,
\end{equation}
where $\mu>0$ and $X_0\in\rn$ are parameters and $a_{n,k}:=\left(\Pi_{j=-k}^{k-1}(n+2j)\right)^{-\frac{1}{k}}$. They are the only positive solutions $U\in C^{2k}(\rn)$ to $\Delta_\xi^k U=U^{\crit-1}$ on $\rn$ where $\xi$ is the Euclidean metric on $\rn$ (see Wei-Xu \cite{weixu}). In order to model bubbles on sign-changing solutions to $\Delta_\xi^k U=|U|^{\crit-2}U$ on $\rn$, we need another notion that was introduced in Premoselli-Robert \cite{PR}.
\begin{defi}[Exponential chart]\label{def:exp:chart} A smooth exponential chart $\tilde{\hbox{exp}}$ around $p_0\in M$ is a function 
$$\begin{array}{cccc}
\tilde{\hbox{exp}}_p: & \rn & \to &  M\\
&(X^1,...,X^n) &\mapsto &\hbox{exp}_p(\sum_i X^i E_i(p))\end{array}$$
where $\hbox{exp}_p: T_pM\to M$ is the usual exponential map and for all $p\in B_{r_g}(p_0)\subset M$, $(E_i(p))_{i=1,...,n}$ is an orthonormal basis of $T_pM$ and for all $i$, $E_i$ is a continuous vector field. If $\tilde{\tilde{exp}}$ is another smooth exponential chart around $p_0$, then for all $p\in B_{r_g}(p_0)$, there is an isometry $L_p: \rn\to\rn$ such that $\tilde{\tilde{exp}}_p= \tilde{exp}_p\circ L_p$. Moreover, $p\mapsto L_p$ is continuous.\end{defi}
In the sequel, we let $D_k^2(\rn)$ be the completion of $C^\infty_c(\rn)$ for the norm $u\mapsto \Vert u\Vert_{D_k^2}:=\Vert \Delta_\xi^\frac{k}{2}u\Vert_2$.

\begin{defi}\label{def:bubble} We say that a family $B=(B_\alpha)_\alpha\in H_k^2(M)$ is a \underline{Bubble} centered at  $(\xa)_\alpha\in M$ with radius $(\ma)_\alpha\in \rr_{>0}$ if there exists $U\in D_{k}^2(\rn)$, $U\not\equiv 0$, and an exponential chart $\tilde{\hbox{exp}}$ around $x_0:=\lim_{\alpha\to 0}\xa$ and a cutoff function $\eta$ such that
\begin{equation}\label{exp:bubble}
B_\alpha(x)=\eta(d_g(x,\xa))\frac{1}{\ma^{\frac{n-2k}{2}}}U\left(\frac{\tilde{\hbox{exp}}_{\xa}^{-1}(x)}{\ma}\right)+R_\alpha(x)\hbox{ for all }x\in M,
\end{equation}
where $\eta(t)=1$ if $t\leq r_g/2$ and $\eta(t)=0$ if $t\geq r_g$, and $\lim_{\alpha\to 0}R_\alpha=0$ in $H_k^2(M)$.
\end{defi}
Note that the choice of $U$ in Definition \ref{def:bubble} depends on the exponential chart:
\begin{proposition}\label{prop:unique:bubble} Given $(\xa)_\alpha\in M$ and $(\ma)_\alpha\in \rr_{>0}$, let $B=(B_\alpha)_\alpha$ be a bubble as in Definition \eqref{def:bubble} with a chart $\tilde{exp}$ and $U\in D_k^2(\rn)$, $U\not\equiv 0$. Then, for another exponential chart $\tilde{\tilde{\hbox{exp}}}$ and $V\in D_k^2(\rn)$, we have that
$$\left\{B_\alpha(x)=\eta(d_g(x,\xa))\frac{1}{\ma^{\frac{n-2k}{2}}}V\left(\frac{\tilde{\tilde{\hbox{exp}}}_{\xa}^{-1}(x)}{\ma}\right)+o(1)\right\}\Leftrightarrow V=U\circ L_{x_0},$$
where $L_{x_0}\in Isom(\rn)$ is as in Definition \ref{def:exp:chart}.
\end{proposition}

 \subsection{The product $\hbox{Weyl}_g\otimes B$}
 The following definition is a direct generalization of the product defined in Premoselli-Robert \cite{PR} for $k=1$:
\begin{defi}\label{def:tensor:B} Let $B=(B_\alpha)_\alpha$ be a bubble. We define the tensor product of the Weyl tensor and the bubble as, when this makes sense, 
\begin{equation*}
\hbox{Weyl}_g\otimes B:=\frac{k}{3}(Weyl_g(x_0))_{ip jq}\int_{\rn}X^p X^p \partial_{ij}\Delta_{\xi}^{k-1}U \left(\frac{n-2k}{2}U+X^l\partial_l U\right)\, dX
\end{equation*}
where $B$ is written as in \eqref{exp:bubble} and the coordinates of the Weyl tensor are taken with respect to the chart $\tilde{exp}_{x_0}$.  This expression is independent of the decomposition in \ref{exp:bubble}.
\end{defi}
The independence is a direct consequence of Proposition \eqref{prop:unique:bubble}.

\begin{proposition} Let $B$ be a bubble centered at $(\xa)_\alpha\in M$ and with radius $(\ma)_\alpha\in \rr_{>0}$. Assume that $|X|^2|\nabla^{2k} U|\cdot(|U|+|X|\cdot|\nabla U|)\in L^1(\rn)$ where $B$ is written as in \eqref{exp:bubble} so that $\hbox{Weyl}_g\otimes B$ makes sense. Then $\hbox{Weyl}_g\otimes B=0$ in the following (non-exhaustive) situations:
\begin{itemize}
\item $\hbox{Weyl}_g(x_0)=0$ where $x_0:=\lim_{\alpha\to 0}\xa$,
\item The bubble is radial in the sense that $U$ in Definition \ref{def:bubble} is radial,
\end{itemize}
\end{proposition}
\begin{proof} The proof goes as in Premoselli-Robert \cite{PR}. The first point is trivial. If $B$ is radial, then let us choose an exponential chart $\varphi_\alpha:=\tilde{\hbox{exp}}_{\xa}$ and a radial function $U\in D_k^2(\rn)$ as in Definition \ref{def:bubble}. Since $\Delta_{\xi}^{k-1}U$ is also radial, there exists $U_1, U_2$ some radial functions such that
 $$\hbox{Weyl}_g\otimes B=(Weyl_g)_{ip jq}\int_{\rn}X^p X^q (\partial_{ij}U_1) U_2\, dX.$$
Since $U_1$ is radial, we have that
$$\partial_{ij} U_1=\delta_{ij}\frac{U_1^\prime(r)}{r}+\frac{x_i x_j}{r} \left(\frac{U_1^\prime(r)}{r}\right)^\prime$$
so that with a change of variable, we have that
 $$\hbox{Weyl}_g\otimes B=(Weyl_g)_{ip jq}\left(A\delta_{ij}\int_{\snmoinsun}\sigma^p\sigma^q\, d\sigma+B\int_{\snmoinsun}\sigma^i\sigma^j\sigma^p\sigma^q\, d\sigma\right).$$
Since (see formula (2.46) in Mazumdar-V\'etois \cite{mv}) we have that
$$\int_{\snmoinsun}\sigma^p\sigma^q\, d\sigma=\frac{\omega_{n-1}\delta_{pq}}{n}\hbox{ and }\int_{\snmoinsun}\sigma^i\sigma^j\sigma^p\sigma^q\, d\sigma=C_n \left(\delta_{pq}\delta_{ij}+\delta_{p i}\delta_{q j}+\delta_{p j}\delta_{q i}\right)$$
for some $C_n>0$, we get that for some $A',B',C'\in\rr$
$$\hbox{Weyl}_g\otimes B=A'(Weyl_g)_{ip ip}+ B'(Weyl_g)_{ii jj}+ C'(Weyl_g)_{ij ji}=0$$
due to the symmetries of the Weyl tensor.
\end{proof}
Conversely, situations where $\hbox{Weyl}_g\otimes B\neq 0$ are in Premoselli-Robert \cite{PR}.

\section{More general results}\label{sec:gene}
\subsection{Generalization of Theorem \ref{th:estim-co:intro}}
Theorem \ref{th:estim-co:intro} is a consequence of the more general theorem :
\begin{theorem}\label{th:main} Let $(M,g)$ be a compact Riemannian manifold of dimension $n$ without boundary. Let $k\in\nn$ be such that $2\leq 2k<n$. We consider a family of operators $P_\alpha\to P_0$ of type (SCC), a family of potentials $(V_\alpha)_\alpha\in L^1(M)$, a family of functions $(v_\alpha)_\alpha\in C^{2k}(M)$ and $(f_\alpha)_\alpha\in L^\infty(M)$ such that
$$P_\alpha v_\alpha= V_\alpha v_\alpha+f_\alpha\hbox{ in }M\hbox{ for all }\alpha\in\nn.$$
We assume that there exists $(\xa)_\alpha\in M$ and $(\ma)_\alpha\in \rr_{>0}$ such that $\lim_{\alpha\to +\infty}\ma =0$ and that for all $R>0$, there exists $\varpi(R)>0$ such that
\begin{equation}\label{cpct:ua:resc}
\Vert \ma^{\frac{n-2k}{2}} v_\alpha(\tilde{\hbox{exp}}_{\xa}(\ma\cdot))\Vert_{C^{2k-1}(B_R(0))}\leq \varpi(R)\hbox{ for all }\alpha>0,
\end{equation}
where $\tilde{\hbox{exp}}_p$ denotes an exponential map at $p\in M$ as in Definition \ref{def:exp:chart}. We assume that for all $\lambda>0$, there exists $R_\lambda>0$ such that
\begin{equation}\label{est:weak:V}
d_g(x,\xa)^{2k}|V_\alpha(x)|\leq \lambda\hbox{ for all }x\in M\setminus B_{R_\lambda\ma}(\xa).
\end{equation}
Then for all $\nu\in (0,1)$, there exists $C_\nu>0$ independent of $\alpha$  such that
\begin{equation}\label{est:co:eps}
|v_\alpha(x)|\leq C_\nu \frac{\Vert f_\alpha\Vert_\infty}{\ma^{ \nu}+d_g(x,\xa)^{\nu}}+C_\nu\left(\frac{\ma^{1-\nu}}{\ma^{2-\nu}+d_g(x,\xa)^{2-\nu}}\right)^{\frac{n-2k}{2}}
\end{equation}
 for all $x\in M$ and $\alpha\in\nn$.
\end{theorem}
Theorem \ref{th:main} is proved in Step \ref{step:nu:proof.3.1} of Section \ref{sec:strong}.

\subsection{Generalizations of Theorems \ref{th:teaser:2} and \ref{th:teaser:3}}
These theorems are direct consequences of the following more general results:
\begin{theorem}\label{th:general} Let $(M,g)$ be a compact Riemannian manifold of dimension $n$ without boundary and let $k\in\nn$ be such that $2\leq 2k<n$. We consider a family of operators $(P_\alpha)_{\alpha>0}\to P_0$ of type (SCC) a family of functions $(\ua)_\alpha\in C^{2k}(M)$ such that \eqref{eq:ua} holds for all $\alpha>0$.
We assume that there exists $u_0\in C^{2k}(M)$ and a Bubble $B=(B_\alpha)_\alpha$ with parameters $U\in D_k^2(\rn)-\{0\}$, $(\xa)_\alpha\in M$,  $(\ma)_\alpha\in\rr_{>0}$ as in Definition \ref{def:bubble} such that
$$\ua=u_0+B_\alpha+o(1)\hbox{ in }H_k^2(M)\hbox{ as }\alpha\to 0.$$
Then $\lim_{\alpha\to 0}\ma^{\ea}=1$, $U\in L^{\crit-1}(\rn)$ and the localization of $x_0:=\lim_{\alpha\to 0}\xa$ is constrained as follows:

\smallskip\noindent$\bullet$ If $n\geq 2k+4$, then $|\nabla^{k-1}U|\in L^2(\rn)$ and 
\begin{eqnarray}
\lim_{\alpha\to 0} \frac{\ea}{\ma^2}\frac{n}{(\crit)^2}\int_{\rn}|U|^{\crit}\, dX&=& \hbox{Weyl}_g\otimes B \label{estim:1}\\
&&+\int_{\rn}\left(A_0^{(k-1)}-A_{g}^{(k-1)}\right)_{x_0}\left(\nabla^{k-1}_\xi U,\nabla^{k-1}_\xi U\right) \, dX\nonumber\\
&&-\frac{n-2k}{2}\left(\int_{\rn}|U|^{\crit-2}U\, dX\right)u_0(x_0){\bf 1}_{n=2k+4}\nonumber
\end{eqnarray}
where $A_{0}^{(k-1)}$ (resp. $A_{g}^{(k-1)}$) is the coefficient of the term of order $2k-2$ in $P_0$ (resp. the GJMS operator \eqref{def:gjms}) via the exponential chart $\tilde{\hbox{exp}}_{x_0}$.

\smallskip\noindent$\bullet$ If $2k<n<2k+4$, then
$$\lim_{\alpha\to 0}\frac{\ea}{\ma^{\frac{n-2k}{2}}}=-\crit \frac{\int_{\rn}|U|^{\crit-2}U\, dX}{\int_{\rn}|U|^{\crit}\, dX}u_0(x_0)$$
\end{theorem}  
When $u_0\equiv 0$, we can lower the critical dimension to $2k+2$:
\begin{theorem}\label{th:no:unull}[$u_0\equiv 0$] Let $(M,g)$ be a compact Riemannian manifold of dimension $n$ without boundary and let $k\in\nn$ be such that $2\leq 2k<n$. We consider a family of operators $(P_\alpha)_{\alpha>0}\to P_0$ of type (SCC) and a family of functions $(\ua)_\alpha\in C^{2k}(M)$ such that \eqref{eq:ua} holds for all $\alpha>0$. We assume that there exists a Bubble $B=(B_\alpha)_\alpha$ with parameters $U\in D_k^2(\rn)-\{0\}$, $(\xa)_\alpha\in M$  and $(\ma)_\alpha\in\rr_{>0}$ as in Definition \ref{def:bubble} such that
$$\ua= B_\alpha+o(1)\hbox{ in }H_k^2(M)\hbox{ as }\alpha\to 0.$$
Then $\lim_{\alpha\to 0}\ma^{\ea}=1$, $U\in L^{\crit-1}(\rn)$ and the localization of $x_0:=\lim_{\alpha\to 0}\xa$ is constrained as follows
 
\smallskip\noindent$\bullet$ If $\{n>2k+2\}$, then $|\nabla^{k-1}U|\in L^2(\rn)$ and 
\begin{equation}\label{estim:2}
\lim_{\alpha\to 0}\frac{\ea}{\ma^2}=\frac{ \hbox{Weyl}_g\otimes B +\int_{\rn}\left(A_0^{(k-1)}-A_{g}^{(k-1)}\right)_{x_0}\left(\nabla^{k-1}_\xi U,\nabla^{k-1}_\xi U\right) \, dX}{\frac{n}{(\crit)^2}\int_{\rn}|U|^{\crit}\, dX}
\end{equation}
where $A_{0}^{(k-1)}$ (resp. $A_{g}^{(k-1)}$) is the coefficient of the term of order $2k-2$ in $P_0$ (resp. the GJMS operator \eqref{def:gjms}) via the exponential chart $\tilde{\hbox{exp}}_{x_0}$.

\smallskip\noindent$\bullet$ If $\{n=2k+2\}$, then
$$\lim_{\alpha\to 0}\frac{\ea}{\ma^2\ln(\frac{1}{\ma})}= a(k,n,U)\int_{\snmoinsun}\left(A_0^{(k-1)}-A_{g}^{(k-1)}\right)_{x_0}\left(\nabla^{k-1}_\xi |x|^{2k-n},\nabla^{k-1}_\xi |x|^{2k-n}\right)\, d\sigma$$
where $$a(k,n,U):=\frac{\left(\crit C_{n,k}\int_{\rn}|U|^{\crit-2}U\, dX\right)^2}{n\int_{\rn}|U|^{\crit}\, dX}\hbox{ with }C_{n,k}\hbox{ as in }\eqref{def:Cnk}$$
\noindent$\bullet$ If $\{n=2k+1\}$, then
$$\lim_{\alpha\to 0}\frac{\ea}{\ma}=-\crit\frac{\left(\int_{\rn}|U|^{\crit-2}U\, dX\right)^2}{\int_{\rn}|U|^{\crit}\, dX} m_{P_0}(x_0)$$
where $m_{P_0}(x_0)$ is the mass of $P_0$ at $x_0$ defined in \eqref{def:mass:intro}.
\end{theorem}  
Theorems \ref{th:general} and \ref{th:no:unull} are proved in Section \ref{sec:appli}.

\subsection{Other consequences}The following results are consequences of Theorem \ref{th:general} and \ref{th:no:unull}:
\begin{theorem}\label{th:teaser:26} Let $(M,g)$ be a compact Riemannian manifold of dimension $n$ without boundary and let $k\in\nn$ be such that $2\leq 2k<n$. Let $x_0\in M$ be a point and $(P_\alpha)_{\alpha>0}\to P_0$ be of type (SCC). We assume that there exists $(h_\alpha)_\alpha\in C^{k-1,\theta}(M)$ such that
$$ P_\alpha:=\Delta_g^k+\Delta_g^{\frac{k-1}{2}}\left(h_\alpha\Delta_g^{\frac{k-1}{2}}\right)+\hbox{lot for }\alpha>0$$
and $h_0\in C^{k-1,\theta}(M)$ such that $\lim_{\alpha\to 0}h_\alpha=h_0$ in $C^{k-1,\theta}(M)$. Assume that we are in one of the following situations:
\begin{itemize}
\item $n=2k+1$ and $m_{P_0}(x_0)>0$;
\item $n\geq  2k+2$ and 
$$h_0(x_0)<\frac{k(3n(n-2)-4k^2+4)}{12 n(n-1)}R_g(x_0),$$
\end{itemize}
where $R_g$ is the scalar curvature of $(M,g)$. Then there is \underline{no} family $(\ua)_{\alpha>0}\in C^{2k}(M)$ such that 
$$P_\alpha \ua=|u_\alpha|^{\crit-2-\ea}u_\alpha\hbox{ in }M,\; (\ea)_\alpha\in [0,\crit-2)$$
and 
\begin{equation*}
\ua= B_\alpha^++o(1)\hbox{ in }H_k^2(M)\hbox{ as }\alpha\to 0 
\end{equation*}
where $(B_\alpha^+)_\alpha$ is a standard bubble as in \eqref{def:pos:bubble} centered at  $\xa\to x_0\in M$ and $\ea\to 0$ in $\rr$ as $\alpha\to 0$. 
\end{theorem}

\begin{theorem}\label{th:teaser:25} Let $(M,g)$ be a compact Riemannian manifold of dimension $n$ without boundary and let $k\in\nn$ be such that $2\leq 2k<n$. Let $u_0\in C^{2k}(M)$ be a function, $x_0\in M$ be a point and $(P_\alpha)_{\alpha>0}\to P_0$ be of type (SCC). We assume that there exists $(h_\alpha)_\alpha\in C^{k-1,\theta}(M)$ such that
$$ P_\alpha:=\Delta_g^k+ \Delta_g^{\frac{k-1}{2}}\left(h_\alpha\Delta_g^{\frac{k-1}{2}}\right)+\hbox{lot for }\alpha>0$$
and $h_0\in C^{k-1,\theta}(M)$ such that $\lim_{\alpha\to 0}h_\alpha=h_0$ in $C^{k-1,\theta}(M)$. Assume that we are in one of the following situations:
\begin{itemize}
\item $2k<n<2k+4$ and $  u_0(x_0)>0$;
\item $n=2k+4$, $h_0(x_0)<\frac{k(3n(n-2)-4k^2+4)}{12 n(n-1)}R_g(x_0) + \kappa'_k   u_0(x_0)$ (see \eqref{def:kappa});
\item $n> 2k+4$ and $h_0(x_0)<\frac{k(3n(n-2)-4k^2+4)}{12 n(n-1)}R_g(x_0)$:
\end{itemize}
Then there is \underline{no} family $(\ua)_{\alpha>0}\in C^{2k}(M)$ such that 
$$P_\alpha \ua=|u_\alpha|^{\crit-2-\ea}u_\alpha\hbox{ in }M,\; (\ea)_\alpha\in [0,\crit-2)$$
and 
\begin{equation*}
\ua=u_0+  B_\alpha^++o(1)\hbox{ in }H_k^2(M)\hbox{ as }\alpha\to 0 
\end{equation*}
where $(B_\alpha^+)_\alpha$ is a standard bubble as in \eqref{def:pos:bubble} centered at  $\xa\to x_0\in M$ and $\ea\to 0$ in $\rr$ as $\alpha\to 0$. 
\end{theorem}
For $n=2k+4$, we have that
\begin{equation}\label{def:kappa}
\kappa_k':=\frac{(n-2k)\int_{\rn}U_{1,0}^{\crit-1}\, dx}{2\int_{\rn}(\Delta_\xi^{\frac{k-1}{2}}U_{1,0})^2\, dx}\hbox{ and }\kappa_k:=(2k+4)^{k-1}\kappa_k'.
\end{equation}
Theorems \ref{th:teaser:26} and \ref{th:teaser:26} are proved in Section \ref{sec:extra}. More general operators are also considered in Section \ref{sec:extra}.
\section{Preliminary material: Sobolev and Hardy inequalities}
It follows from the Euclidean Sobolev embedding that there exists $K(n,k)>0$ such that
\begin{equation}\label{sobo:ineq:rn}
\left(\int_{\rn}|u|^{\crit}\, dx\right)^\frac{2}{\crit}\leq K(n,k)\int_{\rn}(\Delta_\xi^{\frac{k}{2}}u)^2\, dx\hbox{ for all }u\in D_k^2(\rn).
\end{equation}
Concerning compact manifolds, see Hebey \cite{hebey.cims}, $H_k^2(M)\hookrightarrow L^{\crit}(M)$ continuously, so that there exists $C_S(k)>0$ such that 
\begin{equation}\label{sobo:ineq:M}
\Vert u\Vert_{\crit}\leq C_S(k)\Vert u\Vert_{H_k^2}\hbox{ for all }u\in H_k^2(M).
\end{equation}
Independently, the Hardy inequality states that for any $u\in H_k^2(M)$, then $u^2d_g(\cdot,x_0)^{-2k}\in L^1(M)$. Moreover, there exists $C_H(k)>0$ such that 
\begin{equation}\label{hardy:ineq:intro}
\int_M\frac{u^2\, dv_g}{d_g(x,x_0)^{2k}} \leq C_H(k)\Vert u\Vert_{H_k^2}^2\hbox{ for all }u\in H_k^2(M).
\end{equation}
This inequality is \eqref{ineq:hardy:app} that is proved in Appendix \ref{app:hardy}.

\section{Preliminary analysis: adapted chart and convergence up to rescaling}\label{sec:prelim}
\begin{theorem}\label{prop:prelim:1} Let $(M,g)$ be a compact Riemannian manifold of dimension $n$ without boundary and let $k\in\nn$ be such that $2\leq 2k<n$. We consider a family of operators $(P_\alpha)_{\alpha>0}\to P_0$ of type (SCC) and  a family of functions $(\ua)_\alpha\in C^{2k}(M)$ such that
\begin{equation}\label{eq:ua:prop}
P_\alpha \ua= |\ua|^{\crit-2-\ea}\ua\hbox{ in }M\hbox{ for all }\alpha\in\nn.
\end{equation}
We assume that there exists $u_0\in C^{2k}(M)$ and a Bubble $B=(B_\alpha)_\alpha$ with parameters $(\xa)_\alpha\in M$ and $(\ma)_\alpha\in\rr_{>0}$ (see Definition \ref{def:bubble}) such that
\begin{equation}\label{hyp:ua:bubble}
\ua=u_0+B_\alpha+o(1)\hbox{ in }H_k^2(M)\hbox{ as }\alpha\to 0.
\end{equation}
Then, $\lim_{\alpha\to 0}\ma^{\ea}=c_0\in (0,1]$ and there exists $U\in D_k^2(\rn)\cap C^{2k}(\rn)$ such that
\begin{equation}\label{lim:tua:c2k}
\lim_{\alpha\to +\infty}\ma^{\frac{n-2k}{2}}\ua(\tilde{\hbox{exp}}_{\xa}(\ma \cdot))=U\hbox{ in }C^{2k}_{loc}(\rn)\hbox{ and }\Delta_\xi^k U=c_0^{\frac{n-2k}{2}}|U|^{\crit-2}U\hbox{ in }\rn.
\end{equation}
\end{theorem}
The proof goes through several steps. We consider a bubble and $(\ma)_\alpha$, $\eta$ and $U\in D_k^2(\rn)$ as in Definition \ref{def:bubble}. Since $\ua=u_0+B_\alpha+o(1)$ in $H_k^2(M)$ as $\alpha\to 0$,  we get that there exists $C>0$ such that 
\begin{equation}\label{bnd:hk2}
\Vert \ua\Vert_{H_k^2}\leq C\hbox{ for all }\alpha>0\hbox{ small enough.}
\end{equation}
Since $U\in D_k^2(\rn)$, Sobolev's inequality \eqref{sobo:ineq:rn} yields  $U\in L^{\crit}(\rn)$ and therefore for any $R>0$, direct computations yield
\begin{equation}\label{lim:norm:ua:1}
\int_{\Omega_\alpha}|\ua|^{\crit}\, dv_g\leq  C\hbox{Vol}_g(\Omega_\alpha)+ C\left\{\begin{array}{c}
\int_{\frac{\tilde{\hbox{exp}}_{\xa}^{-1}(\Omega_\alpha)}{\ma}\cap B_{R}(0)}|U|^{\crit}\, dx\\
\hbox{or }\int_{\Omega_\alpha\cap B_{R\ma}(\xa)}|B_\alpha|^{\crit}\, dv_g\end{array}\right\}+\eps(R)+o(1)
\end{equation}
where $\lim_{R\to +\infty}\eps(R)=0$. Similarly, for any $R>0$,  a change of variable yields
\begin{equation}\label{lim:norm:ua:2}
\int_{ B_{R\ma}(\xa)}|\ua|^{\crit}\, dv_g=\int_{ B_{R}(0)}|U|^{\crit}\, dx+o(1)\hbox{ as }\alpha\to 0.
\end{equation}
Let $(y_\alpha)_\alpha\in M$ and $(\nu_\alpha)_\alpha\in \rr_{>0}$ be such that
\begin{equation}\label{def:nua}
|\ua(y_\alpha)|=\sup_M|\ua|=\nu_\alpha^{-\frac{n-2k}{2}}.
\end{equation}
We claim that $\lim_{\alpha\to +\infty}\nu_\alpha=0$. We prove the claim by contradiction and we assume that there exists $C>0$ such that $|\ua(x)|\leq C$ for all $x\in M$ and $\alpha>0$, possibly for a subfamily. Therefore, for any $R>0$, we have that $\int_{ B_{R\ma}(\xa)}|\ua|^{\crit}\, dv_g\to 0$ as $\alpha\to 0$, contradicting \eqref{lim:norm:ua:2}. Then $\nu_\alpha\to 0$ as $\alpha\to 0$, and the claim is proved.

\medskip\noindent  We define $\tau_\alpha:=\nu_\alpha^{1-\frac{n-2k}{4k}\ea}$. Since $\ea\to 0$ and $\nu_\alpha\to 0$ as $\alpha\to 0$, we get that $\lim_{\alpha\to 0}\tau_\alpha=0$. We define
\begin{equation*}
\bua(X):=\nu_\alpha^{\frac{n-2k}{2}}\ua\left(\tilde{\hbox{exp}}_{y_\alpha}(\tau_\alpha X)\right)\hbox{ for all }X\in B_{\delta\tau_\alpha^{-1}}(0)
\end{equation*}
where $\delta\in (0, i_g(M))$ and $i_g(M)>0$ is the injectivity radius of $(M,g)$. In the sequel, the notation $T\star S$ will be any linear combination  of contractions of the tensors $T$ and $S$. As one checks, equation \eqref{eq:ua:prop} rewrites
\begin{equation}\label{eq:ua:prop:bar}
\Delta_{\bga}^k\bua+\sum_{j=0}^{2k-2}\tau_\alpha^{2k-j}\bar{B}_\alpha^j\star\nabla^j\bua=|\bua|^{\crit-2-\ea}\bua\hbox{ in }B_{\delta\tau_\alpha^{-1}}(0)
\end{equation}
where $\bga:=(\tilde{\hbox{exp}}_{y_\alpha}g)(\tau_\alpha\cdot)$ is the pull-back metric and for all $j=0,...,2k-2$, $(\bar{B}_\alpha^j)_\alpha$ is a family of $(j,0)-$tensors such that there exists $C_j>0$ such that $\Vert \bar{B}_\alpha^j\Vert_{C^{0,\theta}}\leq C_j$ for all $\alpha>0$. It follows from the definition of $\nu_\alpha$ that
\begin{equation*}
|\bua(X)|\leq |\bua(0)|=1.
\end{equation*}
It then follows from elliptic regularity theory (see \cite{ADN} and Theorem \ref{th:2} and \ref{th:2bis} in Appendix \ref{sec:regul:adn}) that there exists $V\in C^{2k}(\rn)$ such that, up to extraction, 
\begin{equation}\label{ineq:32}
\lim_{\alpha\to 0}\bua=V\hbox{ in }C^{2k}_{loc}(\rn)\hbox{ with }|V(X)|\leq |V(0)|=1\hbox{ for all }X\in\rn.
\end{equation}
We claim that
\begin{equation}\label{ineq:33}
V\in L^{\crit}(\rn)\hbox{ and }\lim_{\alpha\to 0}\nu_\alpha^{\ea}=c_1\in (0,1].
\end{equation}
We prove the claim. Given $R>0$, we have that 
\begin{equation}\label{ineq:34}
\int_{B_R(0)}|\bua|^{\crit}\, dv_{\bar{g}_\alpha}=\left(\nu_\alpha^{\ea}\right)^{\frac{n}{\crit-2}}\int_{B_{R\tau_\alpha}(y_\alpha)}|\ua|^{\crit}\, dv_g.
\end{equation}
Since $H_k^2(M)\hookrightarrow L^{\crit}(M)$ continuously (see \eqref{sobo:ineq:M}) and \eqref{bnd:hk2} holds, then there exists $C>0$ such that $\int_M|\ua|^{\crit}\, dv_g\leq C$ for all $\alpha>0$. Therefore, letting $\alpha\to 0$ and then $R\to +\infty$ in \eqref{ineq:34}, we get that
$$\int_{\rn}|V|^{\crit}\, dX\leq C\left(\liminf_{\alpha\to 0}\nu_\alpha^{\ea}\right)^{\frac{n}{\crit-2}}.$$
Since $\ea\geq 0$ and $\nu_\alpha\to 0$, we have that $\nu_\alpha^{\ea}\leq 1$ for $\alpha\to 0$, and then $V\in L^{\crit}(\rn)$. Moreover, since $V\in C^0(\rn)$ and $V\not\equiv 0$, we then get that $\liminf_{\alpha\to 0}\nu_\alpha^{\ea}>0$. This proves the claim.

\smallskip\noindent It follows from the claim that $\nu_\alpha\asymp \tau_\alpha$ when $\alpha\to 0$.

\medskip\noindent We claim that 
\begin{equation}\label{ineq:37}
d(x_\alpha,y_\alpha)=O(\ma+\nu_\alpha)\hbox{ when }\alpha\to 0.
\end{equation}
We argue by contradiction and we assume that $\frac{d(x_\alpha,y_\alpha)}{\ma+\nu_\alpha}\to +\infty$ when $\alpha\to 0$. We fix $R>0$. We then have that $B_{R\ma}(\xa)\cap B_{R\tau_\alpha}(y_\alpha)=\emptyset $ when $\alpha\to 0$ since  $\nu_\alpha\asymp \tau_\alpha$ when $\alpha\to 0$. Therefore, taking $\Omega_\alpha:=B_{R\tau_\alpha}(y_\alpha)$ in \eqref{lim:norm:ua:1}, we get that
$$\int_{B_{R\tau_\alpha}(y_\alpha)}|\ua|^{\crit}\, dv_g\leq C\nu_\alpha^n+\eps(R)+o(1)\hbox{ as }\alpha\to 0$$
where $\eps(R)\to 0$ as $R\to +\infty$. Passing to the limit $\alpha\to 0$ and $R\to +\infty$ and using \eqref{ineq:34}, \eqref{ineq:33} and \eqref{ineq:32}, we get that $0\geq c_1^{-n/(\crit-2)}\int_{\rn}|V|^{\crit}\, dx$, contradicting $V\not\equiv 0$. This proves the claim.

\medskip\noindent We claim that $\ma=O(\nu_\alpha)$ as $\alpha\to 0$. Here again, we argue by contradiction and assume that $\nu_\alpha=o(\ma)$ as $\alpha\to 0$. It then follows from \eqref{ineq:37} that $d(\xa, y_\alpha)=O(\ma)$ as $\alpha\to 0$. We fix $R>0$ and we set $\Omega_\alpha:=B_{R\tau_\alpha}(y_\alpha)$. It follows from \eqref{lim:norm:ua:1} that
$$\int_{B_{R\tau_\alpha}(y_\alpha)}|\ua|^{\crit}\, dv_g\leq o(1)+C\int_{\frac{\tilde{\hbox{exp}}_{\xa}^{-1}(B_{R\tau_\alpha}(y_\alpha))}{\ma}\cap B_{R}(0)}|U|^{\crit}\, dx +\eps(R)+o(1)\hbox{ as }\alpha\to 0.$$
We set $Y_\alpha\in\rn$ such that $y_\alpha=\tilde{\hbox{exp}}_{y_\alpha}(\ma Y_\alpha)$: since $d(\xa, y_\alpha)=O(\ma)$, we have that $|Y_\alpha|=O(1)$. We then get that
$$\frac{\tilde{\hbox{exp}}_{\xa}^{-1}(B_{R\tau_\alpha}(y_\alpha))}{\ma}\subset B_{C_1R\frac{\tau_\alpha}{\ma}}(Y_\alpha)\hbox{ for all }\alpha>0$$
for some $C_1>0$. Therefore, since $U\in L^{\crit}(\rn)$, $\tau_\alpha\asymp\nu_\alpha$ and $\nu_\alpha=o(\ma)$ as $\alpha\to 0$, we get that $\lim_{R\to +\infty}\lim_{\alpha\to 0}\int_{B_{R\tau_\alpha}(y_\alpha)}|\ua|^{\crit}\, dv_g=0$. Using the arguments of the preceding claim, we also get a contradiction. This proves the claim.

\medskip\noindent We claim that $\nu_\alpha=O( \ma)$ as $\alpha\to 0$. Here again, we argue by contradiction and we assume that $\ma=o(\nu_\alpha)$ as $\alpha\to 0$. We set $X_\alpha\in\rn$ such that $\xa=\tilde{\hbox{exp}}_{y_\alpha}(\tau_\alpha X_\alpha)$. With  \eqref{ineq:37} and $\tau_\alpha\asymp \nu_\alpha$, we get that $d(\xa, y_\alpha)=O(\tau_\alpha)$ as $\alpha\to 0$ and therefore $|X_\alpha|\leq C$ for all $\alpha\to 0$. Arguing as in the preceding claims, we get that 
\begin{eqnarray*}
\int_{B_{R\ma}(x_\alpha)}|\ua|^{\crit}\, dv_g&\leq &C\int_{\frac{\tilde{\hbox{exp}}_{y_\alpha}^{-1}(B_{R\ma}(x_\alpha)}{\tau_\alpha}}|\bua|^{\crit}\, dx\\
&\leq &C\int_{B_{CR\ma/\tau_\alpha}(X_\alpha)}|\bua|^{\crit}\, dx\leq C\left(\frac{\ma}{\nu_\alpha}\right)^n
\end{eqnarray*}
and then for any $R>0$, $\int_{B_{R\ma}(x_\alpha)}|\ua|^{\crit}\, dv_g=o(1)$ as $\alpha\to 0$. This contradicts \eqref{lim:norm:ua:2} since $U\not\equiv 0$. This proves the claim.

\medskip\noindent It then follows from the claims above that
\begin{equation}\label{ppty:mua:nua}
\nu_\alpha\asymp\ma\, ,\, \ma^{\ea}\to c_0\in (0,1]\hbox{ and }d(\xa,y_\alpha)=O(\ma)\hbox{ as }\alpha\to 0.
\end{equation}
We define
\begin{equation}
\label{def:tua}
\tua(X):=\ma^{\frac{n-2k}{2}}\ua(\tilde{\hbox{exp}}_{\xa}(\ma X))\hbox{ for all }X\in B_{\delta/\ma}(0)
\end{equation}
Equation \eqref{eq:ua:prop} rewrites
\begin{equation}\label{eq:tua:prop}
\Delta_{\tga}^k\tua+\sum_{j=0}^{2k-2}\ma^{2k-j}\tilde{B}_\alpha^j\star\nabla^j\tua=\left(\ma^{\ea}\right)^{\frac{n-2k}{2}}|\tua|^{\crit-2-\ea}\tua\hbox{ in }B_{\delta\ma^{-1}}(0)
\end{equation}
where $\tga:=(\tilde{\hbox{exp}}_{\xa}g)(\ma\cdot)$ is the pull-back metric and for all $j=0,...,2k-2$, $(\tilde{B}_\alpha^j)_\alpha$ is a family of $(j,0)-$tensors such that there exists $C_j>0$ such that $\Vert \tilde{B}_\alpha^j\Vert_{C^{0,\theta}}\leq C_j$ for all $\alpha>0$.

\smallskip\noindent We set $\tilde{Y}_\alpha\in\rn$ such that $y_\alpha=\tilde{\hbox{exp}}_{\xa}(\ma \tilde{Y}_\alpha)$. It follows from \eqref{ppty:mua:nua} that $|\tilde{Y}_\alpha|\leq C$ for all $\alpha\to 0$, and then there exists $\tilde{Y}_0\in\rn$ such that $\tilde{Y}_\alpha\to \tilde{Y}_0$ as $\alpha\to 0$. It then follows from \eqref{ppty:mua:nua} and \eqref{def:nua} that there exists $C>0$ and $\lambda>0$ such that
$$|\tua(X)|\leq C\hbox{ for all }X\in B_{\delta/\ma}(0)\hbox{ and }\lim_{\alpha\to 0}|\tua(\tilde{Y}_\alpha)|=\lambda>0.$$
It then follow from elliptic theory (see \cite{ADN}, Theorem \ref{th:2} and \ref{th:2bis} of Appendix \ref{sec:regul:adn}) that there exists $\tilde{U}\in C^{2k}(\rn)$ such that $\lim_{\alpha\to 0}\tua=\tilde{U}$ in $C^{2k}_{loc}(\rn)$. Passing to the limit in \eqref{eq:tua:prop} yields $\Delta_\xi^k\tilde{U}=c_0^{\frac{n-2k}{2}}|\tilde{U}|^{\crit-2}\tilde{U}$ in $\rn$. It follows from the definition of the bubble that $\lim_{\alpha\to 0}\tua=U$ in $L^{\crit}_{loc}(\rn)$. Therefore, $U=\tilde{U}$ a.e. so that $U\in C^{2k}(\rn)$ and \eqref{lim:tua:c2k} holds. This ends the proof of Theorem \ref{prop:prelim:1}.

\section{Strong pointwise control: proof of Theorem \ref{th:estim-co:intro}}\label{sec:strong}
This section is devoted to the prove of the sharp pointwise control of Theorem \ref{th:estim-co:intro}, namely:
\begin{theorem}\label{prop:prelim:2} Under the assumption of Theorem \ref{prop:prelim:1}, there exists $C>0$ such that
\begin{equation}\label{ineq:c0}
|\ua(x)|\leq C\Vert u_0\Vert_\infty^{(\crit-1)^2}+C \left(\frac{\ma}{\ma^2+d_g(x,\xa)^2}\right)^{\frac{n-2k}{2}}\hbox{ for all }x\in M\hbox{ and }\alpha>0.
\end{equation}
In particular, $|U(X)|\leq C(1+|X|)^{2k-n}$ for all $X\in \rn$ and $U\in L^{\crit-1}(\rn)$, where $U$ is as in Definition \ref{def:bubble}.
\end{theorem}
The proof relies on several steps:
\begin{step}\label{prop:weak:estim} Under the assumptions and notations of Theorem \ref{prop:prelim:1}, we have that
\begin{equation*}
d_g(x,\xa)^{2k}|\ua(x)|^{\crit-2-\ea}\leq C\hbox{ for all }\alpha>0\hbox{ and }x\in M.
\end{equation*}
\end{step}
\begin{proof} The proof is similar to Chapter 4 in Druet-Hebey-Robert \cite{DHR} and uses the properties \eqref{lim:norm:ua:1} and \eqref{lim:norm:ua:2}. Define $w_\alpha(x):=d_g(x,\xa)^{2k}|\ua(x)|^{\crit-2-\ea}$ for all $x\in M$ and let $z_\alpha\in M$ be such that $\max_M w_\alpha=w_\alpha(z_\alpha)$. Assume by contradiction that
$$\lim_{\alpha\to 0}\max_Mw_\alpha=+\infty.$$
In particular $w_\alpha(z_\alpha)\to +\infty$. Define
$$\lambda_\alpha:=|\ua(z_\alpha)|^{-\frac{2}{n-2k}}\hbox{ and }l_\alpha:=\lambda_\alpha^{1-\ea/(\crit-2)}.$$
Since $w_\alpha(z_\alpha)\to +\infty$ as $\alpha\to 0$, we get that
\begin{equation}\label{lim:lambda:alpha}
\lim_{\alpha\to 0}\lambda_\alpha=\lim_{\alpha\to 0}l_\alpha=0\hbox{ and }\lim_{\alpha\to 0}\frac{d_g(\xa, z_\alpha)}{l_\alpha}=+\infty.
\end{equation}
Let us set
$$\hat{u}_\alpha(X):=\lambda_\alpha^{\frac{n-2k}{2}}\ua(\tilde{\hbox{exp}}_{z_\alpha}(l_\alpha X))\hbox{ for all }X\in B_{\delta\tau_\alpha^{-1}}(0)$$
where $\delta\in (0, i_g(M))$. As in  \eqref{eq:ua:prop:bar}, equation \eqref{eq:ua:prop} rewrites
\begin{equation}\label{eq:ua:prop:hat}
\Delta_{\hat{g}_\alpha}^k\bua+\sum_{j=0}^{2k-2}l_\alpha^{2k-j}\hat{B}_\alpha^j\star\nabla^j\hat{u}_\alpha=|\hat{u}_\alpha|^{\crit-2-\ea}\hat{u}_\alpha\hbox{ in }B_{\delta l_\alpha^{-1}}(0)
\end{equation}
where $\hat{g}_\alpha:=(\tilde{\hbox{exp}}_{z_\alpha}g)(l_\alpha\cdot)$ is the pull-back metric and for all $j=0,...,2k-2$, $(\hat{B}_\alpha^j)_\alpha$ is a family of $(j,0)-$tensors such that there exists $C_j>0$ such that $\Vert \hat{B}_\alpha^j\Vert_{C^{0,\theta}}\leq C_j$ for all $\alpha>0$. Let us fix $R>0$. For any $X\in B_R(0)$, it follows from the definitions of $w_\alpha$ and $z_\alpha$ that
$$d_g(\xa, \tilde{\hbox{exp}}_{z_\alpha}(l_\alpha X))^{2k}|\ua(\tilde{\hbox{exp}}_{z_\alpha}(l_\alpha X))|^{\crit-2-\ea}\leq d_g(\xa, z_\alpha)^{2k}|\ua(z_\alpha)|^{\crit-2-\ea}.$$
Using the triangle inequality, the definition of $\hat{u}_\alpha$ and \eqref{lim:lambda:alpha}, for $\alpha\to 0$ small enough, we get that
$$|\hat{u}_\alpha(X)|\leq \left(\frac{1}{1-\frac{R l_\alpha}{d_g(\xa, z_\alpha)}}\right)^{\frac{2k}{\crit-2-\ea}}\hbox{ for all }X\in B_R(0).$$
So that $(\hat{u}_\alpha)_\alpha$ is uniformly bounded on $B_R(0)$. It then follows from elliptic theory (see Theorems \ref{th:2} and \ref{th:2bis} of Appendix \ref{sec:regul:adn}) that there exists $\hat{u}\in C^{2k}(\rn)$ such that $\lim_{\alpha\to 0}\hat{u}_\alpha=\hat{u}$ in $C^{2k}_{loc}(\rn)$. Moreover, $|\hat{u}(0)|=\lim_{\alpha\to 0}|\hat{u}_\alpha(0)|=1$. With a change of variable, we get that 
$$\int_{B_{ l_\alpha}(z_\alpha)}|\ua|^{\crit}\, dv_g=\left(\lambda_\alpha^{\ea}\right)^{-\frac{n}{\crit-2}}\int_{B_1(0)}|\hat{u}_\alpha|^{\crit}\, dv_{\hat{g}_\alpha}\geq \int_{B_1(0)}|\hat{u}|^{\crit}\, dx+o(1)$$
since $\lambda_\alpha^{\ea}\leq 1$. Independently, since $B_{ l_\alpha}(z_\alpha)\cap B_{ R\ma}(\xa)=\emptyset$ for $\alpha\to 0$ (otherwise $d_g(\xa, z_\alpha)=O(\ma+l_\alpha)$: absurd), it follows from \eqref{lim:norm:ua:1} that
$$\lim_{\alpha\to 0}\int_{B_{ l_\alpha}(z_\alpha)}|\ua|^{\crit}\, dv_g=0.$$
Which yields $\int_{B_1(0)}|\hat{u}|^{\crit}\, dx=0$: absurd since $\hat{u}\in C^0(\rn)$ and $|\hat{u}(0)|=1$. This ends Step \ref{prop:weak:estim}.\end{proof}

\begin{step}\label{prop:cv:u0} Under the assumptions and notations of Theorem \ref{prop:prelim:1}, we have that
\begin{equation*}
\lim_{\alpha\to 0}\ua=u_0\hbox{ in }C^{2k}_{loc}(M-\{x_0\}),\hbox{ where }x_0:=\lim_{\alpha\to 0}\xa.
\end{equation*}
\end{step}
\begin{proof} It follows from Proposition \ref{prop:weak:estim} that for all $\delta>0$, there exists $C(\delta)>0$ such that $|\ua(x)|\leq C(\delta)$ for all $x\in M-B_\delta(x_0)$. Since $\ua$ solves \eqref{eq:ua:prop}, it follows from elliptic theory (see Theorems \ref{th:2} and \ref{th:2bis} of Appendix \ref{sec:regul:adn}) that there exists $\hat{u}_0\in C^{2k}(M-\{x_0\})$ such that $\lim_{\alpha\to 0}\ua=\hat{u}_0$ in $C^{2k}_{loc}(M-\{x_0\})$. Independently, it follows from  Assumption \eqref{hyp:ua:bubble} that  $\lim_{\alpha\to 0}\ua=u_0$ in $L^2(M)$. Therefore $\hat{u}_0=u_0$. This ends Step \ref{prop:cv:u0}.\end{proof}

\begin{step}\label{prop:weak-strong} Under the assumptions and notations of Theorem \ref{prop:prelim:1}, we have that
\begin{equation*}
\lim_{R\to +\infty}\lim_{\alpha\to 0}\sup_{x\in M\setminus B_{R\ma}(\xa)}d_g(x,\xa)^{2k}|\ua(x)-u_0(x)|^{\crit-2-\ea}=0.
\end{equation*}
\end{step}
\begin{proof} Here again the proof follows Chapter 4 of Druet-Hebey-Robert \cite{DHR}. We sketch the proof since it is similar to the proof of Step \ref{prop:weak:estim}. We argue by contradiction and we assume that there exists $(z_\alpha)_\alpha\in M$ such that 
$$\lim_{\alpha\to 0}\frac{d_g(\xa, z_\alpha)}{\ma}=+\infty\hbox{ and }\lim_{\alpha\to 0}d_g(\xa, z_\alpha)^{2k}|\ua(z_\alpha)-u_0(z_\alpha)|^{\crit-2-\ea}=c_1>0.$$
It follows from Step \ref{prop:cv:u0} that $z_\alpha\to x_0$ as $\alpha\to 0$. Therefore $d_g(\xa, z_\alpha)\to 0$ and $|\ua(z_\alpha)|\to +\infty$ as $\alpha\to 0$. We define $(\lambda_\alpha)_\alpha$ and $(l_\alpha)_\alpha$ as in the proof of Step \ref{prop:weak:estim}. The rescaled function $\hat{u}_\alpha$ then converges to a nonzero function on a punctured domain. The following integrals are estimated on $B_{\rho l_\alpha}(z_\alpha)$ with $\rho>0$ small enough.  We refer to Chapter 4 of Druet-Hebey-Robert \cite{DHR} for details. This ends Step \ref{prop:weak-strong}.\end{proof}

\begin{step}\label{step:34} Under the assumption of Theorem \ref{prop:prelim:1}, we define $v_\alpha:=\ua-u_0$. There exists $(V_\alpha)_\alpha, (f_\alpha)_\alpha\in L^\infty(M)$ such that
\begin{equation}\label{eq:P:V:f}
P_\alpha v_\alpha= V_\alpha v_\alpha+f_\alpha
\end{equation}
where there exists $C>0$ such that
\begin{equation}
\Vert f_\alpha\Vert_\infty\leq C\Vert u_0\Vert_\infty^{\crit-1 }.\label{bnd:falpha}
\end{equation} 
and for  any $\delta>0$, there exists $R_\delta>0$ such that for all $R>R_\delta$ we have that
\begin{equation}\label{ineq:45}
d_g(x,\xa)^{2k}|V_\alpha(x)|\leq \delta\hbox{ for all }x\in M\setminus B_{R\ma}(\xa).
\end{equation}
\end{step}
\begin{proof} We define $v_\alpha:=\ua-u_0$. We  have that \eqref{eq:P:V:f} holds with
$$V_\alpha:=\frac{{\bf 1}_{|u_0|<\frac{1}{2}|v_\alpha|}|u_0+v_\alpha|^{\crit-2-\ea}(u_0+v_\alpha)}{v_\alpha}$$
$$\hbox{and }f_\alpha:={\bf 1}_{\frac{1}{2}|v_\alpha|\leq |u_0|}|u_0+v_\alpha|^{\crit-2-\ea}(u_0+v_\alpha)-P_\alpha u_0.$$
Therefore
\begin{equation}\label{ineq:V:f}
|V_\alpha|\leq C|\va|^{\crit-2-\ea}\hbox{ and }\Vert f_\alpha\Vert_\infty\leq C\Vert u_0\Vert_\infty^{\crit-1-\ea}+C\Vert u_0\Vert_{C^{2k}}\end{equation}
Passing to the weak limit $\ua\rightharpoonup u_0$ in \eqref{eq:ua:prop}, we get that
$$P_0u_0=\Delta_g^k u_0+\sum_{i=0}^{k-1} (-1)^i\nabla^i(A^{(i)}_0\nabla^i u_0)=|u_0|^{\crit-2}u_0\hbox{ in }M.$$
It follows from elliptic theory (see Theorem \ref{th:2} and \ref{th:2bis} of Appendix \ref{sec:regul:adn}) that $  \Vert u_0\Vert_{C^{2k}}\leq C(k,L)\Vert u_0\Vert_\infty^{\crit-1}$. Moreover, using the uniform coercivity of an operator of type (SCC), we get that $\Vert u_0\Vert_{\crit}^2\leq C\Vert u_0\Vert_{\crit}^{\crit}$, so that $\Vert u_0\Vert_{\crit}\geq c$ if $u_0\not\equiv 0$. Therefore, $\Vert f_\alpha\Vert_\infty\leq C\Vert u_0\Vert_\infty^{\crit-1 }$. It follows from \eqref{ineq:V:f} and Step \ref{prop:weak-strong} that for any $\delta>0$, there exists $R_\delta>0$ such that
\begin{equation*}
d_g(x,\xa)^{2k}|V_\alpha(x)|\leq \delta\hbox{ for all }x\in M\setminus B_{R\ma}(\xa).
\end{equation*}
This ends Step \ref{step:34}.\end{proof}

\begin{step}\label{step:nu:proof.3.1} {\bf Proof of Theorem \ref{th:main}.} We consider a family of operators $P_\alpha\to P_0$ of type (SSC), a family of potentials $(V_\alpha)_\alpha\in L^1(M)$, $(v_\alpha)_\alpha\in C^{2k}(M)$ and $(f_\alpha)_\alpha\in L^\infty(M)$ such that 
$$P_\alpha v_\alpha=V_\alpha v_\alpha+f_\alpha\hbox{ in }M\hbox{ for all }\alpha\in M.$$
We also let $(\xa)_\alpha\in M$ and $(\ma)_\alpha\in (0,+\infty)$ such that \eqref{cpct:ua:resc} and \eqref{est:weak:V} hold. We then claim that for all $\nu\in (0,1)$, there exists $C=C_\nu>0$ such that
\begin{eqnarray*}
|v_\alpha(x)|  &\leq& C_\nu  \frac{\Vert f_\alpha\Vert_\infty }{\left(\ma^\nu+d_g(x,\xa)^{\nu}\right)^{\frac{n-2k}{2}}}+C_\nu   \left(\frac{\ma^{1-\nu} }{\ma^{2-\nu}+d_g(x,\xa)^{2-\nu}}\right)^{\frac{n-2k}{2}}
\end{eqnarray*} 
for all $x\in M$ and all $\alpha>0$.
\end{step}
This analysis requires sharp pointwise control on the Green's function for elliptic operators of high order of Theorems \ref{th:Green:main} and \ref{th:Green:pointwise:BIS} of Section \ref{sec:green1}.
\begin{proof} Let $\eta\in C^\infty(\rr)$ be such that $\eta(t)=0$ if $t\leq 1$, $\eta(t)=1$ if $t\geq 2$ and $0\leq\eta\leq 1$. Define 
$$\eta_{R,\alpha}(x):=\eta\left(\frac{d_g(x,\xa)}{R\ma}\right)\hbox{ for all }x\in M.$$
We fix $\nu\in (0,1)$ and we set  $\gamma:=\frac{n-2k}{2}\nu$. We fix $\lambda_\gamma$ as in Theorem \ref{th:Green:pointwise:BIS}. It follows from \eqref{est:weak:V} that there exists $R:=R_\lambda>0$ such that
$$d_g(x,\xa)^{2k}|\eta_{R,\alpha}V_\alpha(x)|\leq \lambda_\gamma\hbox{ for all }x\in M\hbox{ and }\alpha>0.$$
Up to taking $\lambda_\gamma$ smaller, we have the uniform coercivity of $P_\alpha-\eta_{R,\alpha}V_\alpha$. Let $\hat{G}_\alpha$ be the Green's function for $P_\alpha-\eta_{R,\alpha}V_\alpha$ given by Theorem \ref{th:Green:main}, so that the pointwise estimates of Theorem \ref{th:Green:pointwise:BIS} hold uniformly with respect to $\alpha\in\nn$. It follows from the symmetry of $P_\alpha$ and the form of the coefficients of $P_\alpha$ that there exist tensors $\bar{A_\alpha}_{lm}$ such that, integrating by parts, for any smooth domain $\Omega\subset M-\{\xa\}$, we have that for all $u,v\in C^{2k}(\overline{\Omega})$,
\begin{equation}
\int_\Omega (P_\alpha u)v\, dv_g=\int_\Omega u(P_\alpha v)\, dv_g+\sum_{l+m<2k}\int_{\partial\Omega}\bar{A_\alpha}_{lm}\star\nabla^lu\star \nabla^m v\, d\sigma_g
\end{equation}
Since $(P_\alpha-\eta_{R,\alpha}V_\alpha)v_\alpha=f_\alpha$ on $M\setminus B_{2R\ma}(\xa)$, for any $z\in M\setminus B_{3R\ma}(\xa)$, Green's representation formula yields
\begin{eqnarray*}
v_\alpha(z)&=&\int_{M\setminus B_{2R\ma}(\xa)}\hat{G}_\alpha(z,\cdot) f_\alpha\, dv_g\\
&&+\sum_{l+m<2k}\int_{\partial(M\setminus B_{2R\ma}(\xa))}\bar{A_\alpha}_{lm}\star\nabla^l_y\hat{G}_\alpha(z,y)\star \nabla^m v_\alpha(y)\, d\sigma_g(y)\\
&=&\int_{M\setminus B_{2R\ma}(\xa)}\hat{G}_\alpha(z,\cdot) f_\alpha\, dv_g\\
&&-\sum_{l+m<2k}\int_{\partial B_{2R\ma}(\xa)}\bar{A_\alpha}_{lm}\star\nabla^l_y\hat{G}_\alpha(z,y)\star \nabla^m v_\alpha (y)\, d\sigma_g(y)
\end{eqnarray*}
We first deal with the boundary term. It follows from  \eqref{cpct:ua:resc} that there exists $C(R)>0$ such that for all $m<2k$,
$$|\nabla^m v_\alpha(y)|\leq C(R)\ma^{-\frac{n-2k}{2}-m}\hbox{ for all }y\in \partial B_{2R\ma}(\xa)$$
Since $d(z,\xa)>3R\ma$, we get that $d(y,\xa)<\frac{2}{3}d(z,\xa)$ for all $y\in \partial B_{2R\ma}(\xa)$, it then follows from the pointwise control of Theorem \ref{th:Green:pointwise:BIS} that for all $l<2k$,
\begin{eqnarray*}
|\nabla^l_y\hat{G}_\alpha(z,y)|&\leq& C(\gamma) \left(\frac{\max\{d(z,\xa),d(y,\xa)\}}{\min\{d(z,\xa),d(y,\xa)\}}\right)^{\gamma+l} d_g(z,y)^{2k-n-l}\\
&\leq & C(\gamma)\left(\frac{ d(z,\xa) }{\ma}\right)^{\gamma+l} d_g(z,\xa)^{2k-n-l}\leq C(\gamma) \ma^{-\gamma-l}d_g(z,\xa)^{2k-n+\gamma}
\end{eqnarray*}
and then
\begin{eqnarray*}
&&\left|\sum_{l+m<2k}\int_{\partial B_{2R\ma}(\xa)}\bar{A_\alpha}_{lm}\star\nabla^l_y\hat{G}_\alpha(z,y)\star \nabla^m v_\alpha (y)\, d\sigma_g(y)\right|\\
&&\leq C(\gamma) \sum_{l+m<2k}\int_{\partial B_{2R\ma}(\xa)} \ma^{-\frac{n-2k}{2}-m}\ma^{-\gamma-l}d_g(z,\xa)^{2k-n+\gamma}\,  d\sigma_g(y)\\
&&\leq C(\gamma,R) \sum_{l+m<2k}\ma^{n-1} \ma^{-\frac{n-2k}{2}-\gamma-m-l} d_g(z,\xa)^{2k-n+\gamma}\\
&&\leq C(\gamma,R) \sum_{l+m<2k} \frac{\ma^{\frac{n-2k}{2}-\gamma } }{d_g(z,\xa)^{n-2k-\gamma}} \ma^{2k-1-(m+l)} \leq C(\gamma,R)   \frac{\ma^{\frac{n-2k}{2}-\gamma } }{d_g(z,\xa)^{n-2k-\gamma}} 
\end{eqnarray*}
for all $z\in M$ such that $d_g(z,\xa)>3R\ma$. We now deal with the interior integral. Using the pointwise estimates of Theorem \ref{th:Green:pointwise:BIS}, we get that
\begin{eqnarray*}
&&\left|\int_{M\setminus B_{2R\ma}(\xa)}\hat{G}_\alpha(z,\cdot) f_\alpha\, dv_g\right|\leq  \Vert f_\alpha\Vert_\infty\int_{M\setminus B_{2R\ma}(\xa)}|\hat{G}_\alpha(z,\cdot)|\, dv_g\\
&&\leq C(\gamma)\Vert f_\alpha\Vert_\infty\int_{M\setminus B_{2R\ma}(\xa)}\left(\frac{\max\{d_g(z,\xa),d_g(y,\xa)\}}{\min\{d_g(z,\xa),d_g(y,\xa)\}}\right)^\gamma d_g(z,y)^{2k-n}\, dv_g(y)\\
&&\leq C(\gamma)\Vert f_\alpha\Vert_\infty\left(\int_{d_g(z,\xa)<d_g(y,\xa)/2}+\int_{d_g(y,\xa)<d_g(z,\xa)/2}+\int_{d_g(z,\xa)/2<d_g(y,\xa)<2d_g(z,\xa)}\right)\\
&&\leq C(\gamma)\Vert f_\alpha\Vert_\infty\left(\int_M\frac{d_g(y, \xa)^{2k+\gamma-n}}{d_g(z,\xa)^\gamma}\, dv_g(y)+\int_M\frac{d_g(z, \xa)^{2k+\gamma-n}}{d_g(y,\xa)^\gamma}\, dv_g(y)+\int_M d_g(z,y)^{2k-n}\, dv_g(y)\right)
\end{eqnarray*} and then
\begin{equation*}
\left|\int_{M\setminus B_{2R\ma}(\xa)}\hat{G}_\alpha(z,\cdot) f_\alpha\, dv_g\right|\leq C\Vert f_\alpha\Vert_\infty d_g(z,\xa)^{-\gamma}.
\end{equation*}
Putting these inequalities together and since $\gamma:=\frac{n-2k}{2}\nu$, we get that 
\begin{eqnarray*}
|v_\alpha(z)|  &\leq& C(\gamma)  \Vert u_0\Vert_\infty^{\crit-1 }d_g(z,\xa)^{-\frac{n-2k}{2}\nu}+ C(\gamma,R)    \left(\frac{\ma^{1-\nu} }{d_g(z,\xa)^{2-\nu}}\right)^{\frac{n-2k}{2}} \\
&\leq& C(\gamma,R)  \Vert f_\alpha\Vert_\infty\left(\ma^\nu+d_g(z,\xa)^{\nu}\right)^{-\frac{n-2k}{2}}+C(\gamma,R)   \left(\frac{\ma^{1-\nu} }{\ma^{2-\nu}+d_g(z,\xa)^{2-\nu}}\right)^{\frac{n-2k}{2}}
\end{eqnarray*}
for all $z\in M$ such that $d_g(z,\xa)>3R\ma$. The same estimate holds on $B_{R\ma}(\xa)$ due to \eqref{cpct:ua:resc}. This ends Step \ref{step:nu:proof.3.1}. \end{proof}

\begin{step} Under the assumption of Theorem \ref{prop:prelim:1}, there exists $C>0$ such that \eqref{ineq:c0} holds.
\end{step}
\begin{proof} We first consider $v_\alpha:=\ua-u_0$, $f_\alpha$, $V_\alpha$ defined in Step \ref{step:34} satisfying \eqref{eq:P:V:f}. Using \eqref{bnd:falpha}, \eqref{ineq:45} and \eqref{lim:tua:c2k}, for any $\nu\in (0,1)$, it follows from Step \ref{step:nu:proof.3.1} that there exists $C=C_\nu>0$ such that
\begin{eqnarray}
|u_\alpha(x)|  &\leq& C_\nu  \frac{\Vert u_0\Vert_\infty^{\crit-1 }}{\left(\ma^\nu+d_g(x,\xa)^{\nu}\right)^{\frac{n-2k}{2}}}+C_\nu   \left(\frac{\ma^{1-\nu} }{\ma^{2-\nu}+d_g(x,\xa)^{2-\nu}}\right)^{\frac{n-2k}{2}}\label{est:nu}
\end{eqnarray} 
for all $x\in M$ and all $\alpha>0$. We let $G_\alpha$ be the Green's function for the operator $P_\alpha$. We then have that
$$\ua(x)=\int_M G_\alpha(x,y)|\ua|^{\crit-2-\ea}\ua(y)\, dv_g(y)$$
for all $x\in M$. The estimates \eqref{ctrl:G} on the Green's function yield $|G_\alpha(x,y)|\leq C d_g(x,y)^{2k-n}$ for all $x,y\in M,\, x\neq y$. Therefore, with these pointwise controls and \eqref{est:nu}, for any $x\in M$, we get that
\begin{eqnarray*}
|\ua(x)|&\leq&  C(\nu)  \Vert u_0\Vert_\infty^{(\crit-1)^2 }\int_M d_g(x,y)^{2k-n}\left(\ma^\nu+d_g(y,\xa)^{\nu}\right)^{-\frac{n-2k}{2}(\crit-1-\ea)}\, dv_g(y)\\
&&+ C(\nu )\int_M d_g(x,y)^{2k-n}\left(\frac{\ma^{1-\nu} }{\ma^{2-\nu}+d_g(y,\xa)^{2-\nu}}\right)^{\frac{n-2k}{2}\cdot (\crit-1-\ea)}\, dv_g(y).
\end{eqnarray*}
For $\nu>0$ small enough, see \cite{GR.CalcVar} for instance, we then get that
\begin{equation*}
|\ua(z)|\leq C  \Vert u_0\Vert_\infty^{(\crit-1)^2 }+C   \left(\frac{\ma}{\ma^{2}+d_g(z,\xa)^{2}}\right)^{\frac{n-2k}{2}}
\end{equation*}
for all $z\in M$. This proves \eqref{ineq:c0} and ends the proof of Theorem \ref{prop:prelim:2}.\end{proof}
As mentioned above, this proves Theorem \ref{th:estim-co:intro}.

\section{Green's function with Hardy potential: construction and first estimates}\label{sec:green1}
\begin{defi}\label{def:p} We say that an operator $P$ is of type $O_{k,L}$ if
$$\bullet\; P:=\Delta_g^k+\sum_{i=0}^{k-1} (-1)^i\nabla^i(A^{(i)} \nabla^i)\hbox{ for }\alpha>0$$
where for all $i=0,...,k-1$, $ A^{(i)}\in C^{i}_{\chi}(M,\Lambda_{S}^{(0,2i)}(M))$ is a family $C^{i}-$field of $(0,2i)-$tensors on $M$ such that:\par
$\bullet$ For any $i$ , $A^{(i)}(T,S)=A^{(i)}(S,T)$ for any $(i,0)-$tensors $S$ and $T$. \par
$\bullet$ $\Vert A^{(i)}\Vert_{C^{i}}\leq L$ for all $i=0,...,k-1$\par
$\bullet$ $P$ is  coercive and 
$$\int_M u P u\, dv_g=\int_M (\Delta_g^{\frac{k}{2}}u)^2\, dv_g+\sum_{i=0}^{k-1}\int_M A^{(i)} (\nabla^iu,\nabla^i u)\, dv_g\geq \frac{1}{L}\Vert u\Vert_{H_k^2}^2$$
for all $u\in H_k^2(M)$.\end{defi}
The proofs of Theorems \ref{th:estim-co:intro}  and \ref{th:main} rely on pointwise estimates of the Green's function for elliptic operators with Hardy potential. For $x_0\in M$ and $\lambda>0$, we define
\begin{equation*}
\Pl(x_0):=\left\{\begin{array}{cc}
V\in L^1(M)\hbox{ such that}\\
|V(x)|\leq \lambda d_g(x,x_0)^{-2k}\hbox{ for all }x\in \Mmx\end{array}\right\}.
\end{equation*}
The core of this paper is devoted to the proof of the following theorems:
\begin{theorem}\label{th:Green:main} Let $(M,g)$ be a compact Riemannian manifold of dimension $n$ without boundary. Fix $k\in\nn$ such that $2\leq 2k<n$, $L>0$, $\lambda>0$ and $x_0\in M$. We consider an operator $P\in O_{k,L}$ and a potential $V\in \Pl(x_0)$ such that $P-V$ is coercive.

\smallskip\noindent Then there exists $G:(\Mmx)\times (\Mmx)\setminus\{(z,z)/z\in \Mmx\}\to \rr$ such that:
\begin{itemize}
\item For all $x\in \Mmx$, $G(x,\cdot)\in L^q(M)$ for all $1\leq q<\frac{n}{n-2k}$
\item For all $x\in \Mmx$, $G(x,\cdot)\in L^{\crit}_{loc}(M\setminus\{x\})$
\item For all $f\in L^{\frac{2n}{n+2k}}(M)\cap L^p_{loc}(\Mmx)$, $p>\frac{n}{2k}$, we let $\varphi\in H_k^2(M)$ such that $P\varphi=f$ in the weak sense. Then $\varphi\in C^0(\Mmx)$ and
$$\varphi(x)=\int_M  G(x,\cdot) f\, dv_g\hbox{ for all }x\in \Mmx.$$
\end{itemize}
Moreover, such a function $G$ is unique. It is the Green's function for $P-V$. In addition, for all $x\in \Mmx$, $G(x,\cdot)\in H_{2k}^p(M\setminus\{x_0,x\})$ for all $1<p<\infty$.
\end{theorem}
In addition, we get the following pointwise control:
\begin{theorem}\label{th:Green:pointwise:BIS} Let $(M,g)$ be a compact Riemannian manifold of dimension $n$ without boundary. Fix $k\in\nn$ such that $2\leq 2k<n$, $L>0$, $\lambda>0$ and $x_0\in M$.  We consider an operator $P\in O_{k,L}$ and a potential $V\in L^1(M)$ such that $P-V$ is coercive and
$$d_g(x,x_0)^{2k}|V(x)|\leq \lambda\hbox{ for all }x\in M.$$
We let $G$ be the Green's function of $P-V$ as in Theorem \ref{th:Green:main}. 

\smallskip\noindent Then for any $\gamma\in (0,n-2k)$, there exists $\lambda_\gamma>0$ such that, if $\lambda<\lambda_\gamma$, then we have the following pointwise estimates: for any $x,y\in \Mmx$ such that $x\neq y$, we have that 
\begin{itemize}
\item $$|G(x,y)|\leq C(\gamma,L,\lambda, k)\left(\frac{\max\{d_g(x,x_0),d_g(y,x_0)\}}{\min\{d_g(x,x_0),d_g(y,x_0)\}}\right)^\gamma d_g(x,y)^{2k-n}$$
\item If $d_g(x,x_0)<d(y, x_0)$ and $l\leq 2k-1$, we have that
$$|\nabla_y^lG(x,y)|\leq C(\gamma,L,\lambda, k,l)\left(\frac{\max\{d_g(x,x_0),d_g(y,x_0)\}}{\min\{d_g(x,x_0),d_g(y,x_0)\}}\right)^\gamma d_g(x,y)^{2k-n-l}$$
\item If $d_g(y,x_0)<d(x, x_0)$ and $l\leq 2k-1$, we have that
$$|\nabla_y^lG(x,y)|\leq C(\gamma,L,\lambda, k,l)\left(\frac{\max\{d_g(x,x_0),d_g(y,x_0)\}}{\min\{d_g(x,x_0),d_g(y,x_0)\}}\right)^{\gamma+l} d_g(x,y)^{2k-n-l}$$
\end{itemize}
where $C(\gamma,L,\lambda, k)$ depends only on $(M,g)$, $L$, $\lambda$ and $l$. Moreover, these estimates are uniform with respect to $x_0\in M$.
\end{theorem}

We fix $k\in\nn$ such that $2\leq 2k<n$ and $L>0$. We consider an operator $P\in O_{k,L}$ (see Definition \ref{def:p}).

\subsection{The Hardy-type potential} We fix $x_0\in M$. We claim that there exists $\lambda_0=\lambda_0(k, L)$ such for all $V_0\in \Pl(x_0) $ with $0<\lambda<\lambda_0$, then 
\begin{equation}\label{coer:PV:0}
\int_M (Pu-V_0 u)u\, dv_g\geq \frac{1}{2L}\Vert u\Vert_{H_k^2}^2\hbox{ for all }u\in H_k^2(M)
\end{equation}
and there exists a family $(V_\eps)_{\eps>0} \in L^\infty(M)$ such that:
\begin{equation}\label{hyp:Ve}
\left\{\begin{array}{cc}
lim_{\eps\to 0}V_\eps(x)=V_0(x)&\hbox{ for a.e. }x\in \Mmx\\
V_\eps\in \Pl(x_0)&\hbox{ for all }\eps>0\\
P-V_\eps&\hbox{ is uniformly coercive for all }\eps>0
\end{array}\right\}
\end{equation}
in the sense that
\begin{equation}\label{coer:PV}
\int_M (Pu-V_\eps u)u\, dv_g\geq \frac{1}{2L}\Vert u\Vert_{H_k^2}^2\hbox{ for all }u\in H_k^2(M)\hbox{ and }\eps>0
\end{equation}

\smallskip\noindent We prove the claim. 

The coercivity of $P$ and the Hardy inequality \eqref{hardy:ineq:intro} yield
$$\int_M u(P-V_0)u\, dv_g\geq \frac{1}{L}\Vert u\Vert_{H_k^2}^2-\lambda \int_M \frac{u^2}{d_g(x,x_0)^{2k}}\, dv_g\geq \left(\frac{1}{L} -\lambda C_H(k)\right)\Vert u\Vert_{H_k^2}^2$$
for all $u\in H_k^2(M)$. For $\eta\in C^\infty(\rr)$ such that $\eta(t)=0$ for $t\leq 1$ and $\eta(t)=1$ for $t\geq 2$, define $V_\eps(x):=\eta(d_g(x,x_0)/\eps) V_0(x)$ for all $\eps>0$ and a.e. $x\in M$. As one checks, the claim holds with $0<\lambda_0<(C_H(k)L)^{-1}$. This proves the claim.

\medskip\noindent For any $\eps>0$,  we let $G_\eps$ be the Green's function for the operator $P-V_\eps$. Since $V_\eps\in L^\infty(M)$, the existence of $G_\eps$ follows from Theorem \ref{th:green:nonsing} of the Appendix \ref{app:green} (see the concluding remark of the Theorem).

\medskip\noindent{\bf Step 1: Integral bounds.} We choose $f\in C^{0}(M)$ and we fix $\eps>0$. Since $P-V_\eps$ is coercive, it follows from variational methods that there exists a unique function $\varphi_\eps\in H_k^2(M)$ such that 
$$(P-V_\eps)\varphi_\eps=f\hbox{ in }M\hbox{ in the weak sense.}$$
It follows from point (ii) of Definition   \ref{def:p} of $P\in O_{k,L}$ that $A^{(i)}\in C^{i}$ for all $i=0,...,2k-1$, we then get with Theorem \ref{th:3} and Sobolev's embedding theorem that $\varphi_\eps\in C^{2k-1}(M)$ and $\varphi_\eps\in H_{2k}^p(M)$ for all $p>1$. The coercivity hypothesis \eqref{coer:PV} yields
\begin{eqnarray*}
\frac{1}{2L}\Vert \varphi_\eps\Vert_{H_k^2}^2\leq \int_M (P\varphi_\eps-V_\eps\varphi_\eps)\varphi_\eps\, dv_g=\int_M f\varphi_\eps\, dv_g\leq \Vert f\Vert_{\frac{2n}{n+2k}}\Vert\varphi_\eps\Vert_{\frac{2n}{n-2k}}
\end{eqnarray*}
With Sobolev's embedding Theorem and the inequality \eqref{sobo:ineq:M}, we get that
\begin{equation}\label{eq:14}
C_S(k)^{-1}\Vert\varphi_\eps\Vert_{\crit}\leq \Vert\varphi_\eps \Vert_{H_k^2}\leq 2LC_S(k)\Vert f\Vert_{\frac{2n}{n+2k}}
\end{equation}
for all $f\in C^{0}(M)$. We fix $p>1$ such that
$$\frac{n}{2k}<p<\frac{n}{2k-1}\hbox{ and }\theta_p:=2k-\frac{n}{p}\in (0,1).$$
We fix $\delta\in (0, i_g(M)/2)$, where $i_g(M)>0$ is the injectivity radius of $(M,g)$. Since $V_\eps\in \Pl(x_0)$ for all $\eps>0$ and $P\in O_{k,L}$, it follows from regularity theory (Theorem \ref{th:2} of Appendix \ref{sec:regul:adn}) and Sobolev's embedding theorem that
\begin{eqnarray*}
\Vert\varphi_\eps\Vert_{C^{0,\theta_p}(M\setminus B_\delta(x_0))}&\leq & C(p,\delta,k)\Vert \varphi_\eps\Vert_{H_{2k}^p(M\setminus B_\delta(x_0))}\\
&\leq & C(p,\delta,k, L,\lambda_0) \left(\Vert f\Vert_{L^p(M\setminus B_{\delta/2}(x_0))}+\Vert \varphi_\eps\Vert_{L^{\crit}(M\setminus B_{\delta/2}(x_0))}\right).
\end{eqnarray*}
With \eqref{eq:14} and noting that $\frac{n}{2k}>\frac{2n}{n+2k}$, we get that
\begin{equation*}
\Vert\varphi_\eps\Vert_{C^{0,\theta_p}(M\setminus B_\delta(x_0))}\leq  C(p,\delta,k, L,\lambda_0) 
\Vert f\Vert_{L^p(M)}.
\end{equation*}
Since $\varphi_\eps\in H_{2k}^p(M)$ for all $p>1$, for any $x\in \Mmx$, Green's representation formula (see Theorem \ref{th:green:nonsing}) yields
$$\varphi_\eps(x)=\int_M G_\eps(x,y)f(y)\, dv_g(y)\hbox{ for all }x\in M- \{x_0\},$$
and then when $d_g(x,x_0)>\delta$, we get
$$\left|\int_M G_\eps(x,y)f(y)\, dv_g(y)\right|\leq C(p,\delta,k, L,\lambda) 
\Vert f\Vert_{L^p(M)}$$
for all $f\in C^{0}(M)$ and $p\in \left(\frac{n}{2k},\frac{n}{2k-1}\right)$. Via duality, we then deduce that 
\begin{equation}\label{eq:16}
\Vert G_\eps(x,\cdot)\Vert_{L^q(M)}\leq C(q,\delta, k, L,\lambda_0)\hbox{ for all }q\in \left(1,\frac{n}{n-2k}\right)\hbox{ and }d_g(x,x_0)>\delta. 
\end{equation}
We now fix $x\in M$ such that $d_g(x,x_0)>\delta$. We take $f\in C^0(M)$ such that $f\equiv 0$ in $B_{\delta/2}(x)$, so that $(P-V_\eps)\varphi_\eps=0$ in $B_{\delta/2}(x)$. Since $V_\eps\in \Pl(x_0)$ for all $\eps>0$ and $P\in O_{k,L}$, it follows from regularity theory (Theorems \ref{th:1} and \ref{th:2} of Appendix \ref{sec:regul:adn}) and Sobolev's embedding theorem that for any $p>\frac{n}{2k}$,
\begin{eqnarray*}
\Vert\varphi_\eps\Vert_{C^{0,\theta_p}(B_{\delta/4}(x))}&\leq & C(p,\delta,k)\Vert \varphi_\eps\Vert_{H_{2k}^p(B_{\delta/4}(x))}\\
&\leq & C(p,\delta,k, L,\lambda_0) \Vert \varphi_\eps\Vert_{L^{\crit}(B_{\delta/2}(x))}.
\end{eqnarray*}
With \eqref{eq:14}, we get that
\begin{equation*}
\Vert\varphi_\eps\Vert_{C^{0,\theta_p}(  B_{\delta/4}(x))}\leq  C(p,\delta,k, L,\lambda_0) 
\Vert f\Vert_{L^{\frac{2n}{n+2k}}(M)}.
\end{equation*}
Since $\varphi_\eps\in H_{2k}^p(M)$ for all $p>1$, Green's representation formula (see Theorem \ref{th:green:nonsing}) yields
$$\varphi_\eps(x)=\int_M G_\eps(x,y)f(y)\, dv_g(y),$$
and then 
$$\left|\int_M G_\eps(x,y)f(y)\, dv_g(y)\right|\leq C(p,\delta,k, L,\lambda) 
\Vert f\Vert_{L^{\frac{2n}{n+2k}}(M)}$$
for all $f\in C^{0}(M)$ vanishing in $B_{\delta/2}(x)$. Via duality, we then deduce that 
\begin{equation}\label{eq:16:crit}
\Vert G_\eps(x,\cdot)\Vert_{L^{\crit}(M\setminus B_{\delta/2}(x))}\leq C(\delta, k, L,\lambda_0)\hbox{ when }d_g(x,x_0)>\delta. 
\end{equation}

\medskip\noindent{\bf Step 2: passing to the limit $\eps\to 0$ and Green's function for $P-V_0$.}\par
\noindent We fix $\delta>0$ and $x\in M$ such that $d_g(x,x_0)>\delta$. For all $\eps>0$, we have that 
\begin{equation}\label{eq:Ge}
P G_\eps(x,\cdot)-V_\eps G_\eps(x,\cdot)=0\hbox{ in }M-\{x\}.
\end{equation}
Since $V_\eps\in \Pl(x_0)$ for all $\eps>0$, we have that $|V_\eps(y)|\leq C(\lambda, \delta)$ for all $y\in M\setminus B_{\delta/2}(x_0)$ and $\eps>0$. Since $P\in O_{k,L}$, it follows from the control \eqref{eq:16} and standard regularity theory (see Theorem \ref{th:2}) that given $\nu\in (0,1)$, we have that for any $r>0$,
\begin{equation}\label{est:G:1}
\Vert G_\eps(x,\cdot)\Vert_{C^{2k-1,\nu}(M-(B_r(x_0)\cup B_r(x))}\leq C(\delta, k, \lambda, L,r,\nu,\lambda_0)\hbox{ for all }d_g(x,x_0)>\delta.
\end{equation}
It then follows from Ascoli's theorem that, up to extraction of a subfamily, there exists $G_0(x,\cdot)\in C^{2k-1}(M-\{x,x_0\})$ such that
\begin{equation}\label{lim:Ge}
\lim_{\eps\to 0}G_\eps(x,\cdot)=G_0(x,\cdot)\hbox{ in }C^{2k-1}_{loc}(M-\{x,x_0\}).
\end{equation}
By Theorem \ref{th:2} again, we also get that
\begin{equation}\label{lim:Ge:H}
\lim_{\eps\to 0}G_\eps(x,\cdot)=G_0(x,\cdot)\hbox{ in }H^p_{2k,loc}(M-\{x,x_0\})\hbox{ for all }p>1.
\end{equation}
Moreover, via the monotone convergence theorem, passing to the limit in \eqref{eq:16}, we get that 
\begin{equation}\label{eq:17}
\Vert G_0(x,\cdot)\Vert_{L^q(M)}\leq C(q,\delta, k, L,\lambda_0)\hbox{ for all }q\in \left(1,\frac{n}{n-2k}\right)\hbox{ and }d_g(x,x_0)>\delta,
\end{equation}
and then $G_0(x,\cdot)\in L^q(M)$ for all $q\in \left(1,\frac{n}{n-2k}\right)$ and $x\neq x_0$. Similarly, using \eqref{eq:16:crit}, we get that
\begin{equation}\label{eq:16:crit:final}
\Vert G_0(x,\cdot)\Vert_{L^{\crit}(M\setminus B_{\delta/2}(x))}\leq C(\delta, k, L,\lambda_0)\hbox{ when }d_g(x,x_0)>\delta. 
\end{equation}
So that $G_0(x,\cdot)\in L^{\crit}_{loc}(M\setminus\{x\})$.

\medskip\noindent{\bf Step 3: Representation formula.} We fix $f\in L^{\frac{2n}{n+2k}}(M)\cap L^p_{loc}(\Mmx)$, $p>\frac{n}{2k}>1$. Via the coercivity of $P-V_\eps$ and $P-V_0$, it follows from variational methods (see also Theorem \ref{th:3}) that there exists  $\varphi_\eps\in H_k^2(M)\cap H_{2k}^{\frac{2n}{n+2k}}(M)$ and $\varphi_0 \in H_k^2(M)$ such that 
\begin{equation}\label{eq:phieps:phi0}
(P-V_0)\varphi_0=f\hbox{ and }(P-V_\eps)\varphi_\eps=f\hbox{ in }M.
\end{equation}
We claim that 
\begin{equation}\label{lim:phi:eps}
\lim_{\eps\to 0}\varphi_\eps=\varphi_0\hbox{ in }H_k^2(M).
\end{equation}
We prove the claim. It follows from \eqref{eq:14} that there exists $C=C(k,L,\Vert f\Vert_{\frac{2n}{n+2k}})>0$ independent of $\eps>0$ such that 
$$\Vert\varphi_\eps\Vert_{H_k^2}\leq C(k,L,\Vert f\Vert_{\frac{2n}{n+2k}})\hbox{ and }\Vert\varphi_0\Vert_{H_k^2}\leq C(k,L,\Vert f\Vert_{\frac{2n}{n+2k}})$$ 
for all $\eps>0$. We have that
\begin{equation}\label{eq:phiemoinphi0}
(P-V_\eps)(\varphi_\eps-\varphi_0)=(V_\eps-V_0)\varphi_0\hbox{ weakly in }H_k^2(M).
\end{equation}
The uniform coercivity \eqref{coer:PV} then yields
\begin{eqnarray*}
\frac{1}{2L}\Vert \varphi_\eps-\varphi_0\Vert_{H_k^2}^2&\leq &\int_M (\varphi_\eps-\varphi_0)(P-V_\eps)(\varphi_\eps-\varphi_0)\, dv_g\\
&=&\int_M (V_\eps-V_0)\varphi_0(\varphi_\eps-\varphi_0)\, dv_g\\
&\leq & \sqrt{\int_M |V_\eps-V_0|\cdot |\varphi_0|^2\, dv_g}\sqrt{\int_M |V_\eps-V_0|\cdot |\varphi_\eps-\varphi_0|^2\, dv_g}\\
&\leq & \sqrt{\int_M |V_\eps-V_0|\cdot |\varphi_0|^2\, dv_g}\sqrt{2\lambda \int_Md_g(\cdot,x_0)^{-2k} |\varphi_\eps-\varphi_0|^2\, dv_g}
\end{eqnarray*}
The Hardy inequality \eqref{hardy:ineq:intro} yields
\begin{eqnarray*}
\Vert \varphi_\eps-\varphi_0\Vert_{H_k^2}^2&\leq & 2L\sqrt{\int_M |V_\eps-V_0|\cdot |\varphi_0|^2\, dv_g}\sqrt{2\lambda C_H(k)} \Vert \varphi_\eps-\varphi_0\Vert_{H_k^2}
\end{eqnarray*}
and then 
\begin{equation}\label{eq:21}
\Vert \varphi_\eps-\varphi_0\Vert_{H_k^2}\leq 2L\sqrt{2\lambda C(k)} \sqrt{\int_M |V_\eps-V_0|\cdot |\varphi_0|^2\, dv_g} 
\end{equation}
Since $V_\eps,V_0\in \Pl(x_0)$, we have that $|V_\eps-V_0|(x)\cdot |\varphi_0|^2(x)\leq 2\lambda d_g(x,x_0)^{-2k}|\varphi_0|^2(x)\in L^1(M)$ due to the Hardy inequality \eqref{hardy:ineq:intro}. Since $\lim_{\eps\to 0}V_\eps(x)=V_0(x)$ for a.e. $x\in M$, Lebesgue's convergence theorem yields
$$\lim_{\eps\to 0}\int_M |V_\eps-V_0|\cdot |\varphi_0|^2\, dv_g=0.$$
And then \eqref{eq:21} yields \eqref{lim:phi:eps} and the claim is proved.

\medskip\noindent Since $\varphi_\eps\in H_{2k}^{\frac{2n}{n+2k}}$ is a solution to \eqref{eq:phieps:phi0}, $P\in O_{k,L}$, $V_\eps\in \Pl(x_0)$ and $f\in L^p_{loc}(M\setminus \{x_0\})$, $p>\frac{n}{2k}$,  it  follows from regularity theory (Theorems \ref{th:1} and \ref{th:2}) that $\varphi_\eps\in H_{2k,loc}^p(M\setminus\{x_0\})$ and that for any $\delta>0$, using \eqref{lim:phi:eps}, we get
$$\Vert \varphi_\eps\Vert_{H_{2k}^p(M\setminus B_r(x_0))}\leq C( r, k,L,\lambda_0)\left(\Vert f\Vert_{L^p(M\setminus B_{r/2}(x_0))}+\Vert \varphi_\eps\Vert_{L^2(M\setminus B_{r/2}(x_0))}\right)\leq C( r, k,L,\lambda_0).$$
Since $p>n/2k$, it follows from Sobolev's embedding theorem that $\varphi_\eps\in C^0(M\setminus\{x_0\})$ and that
\begin{equation}\label{ineq:co:phi}
\Vert \varphi_\eps\Vert_{C^0(M\setminus B_r(x_0))}\leq C(k,r)\Vert \varphi_\eps\Vert_{H_{2k}^p(M-B_r(x_0))}\leq C( r, k,L,\lambda_0)
\end{equation}
for all $\eps>0$. We now take $\eps,\eps'>0$. As one check, $(P-V_\eps)(\varphi_{\eps}-\varphi_{\eps'})=(V_\eps-V_{\eps'})\varphi_{\eps'}$. It then follows from regularity theory (Theorems \ref{th:1} and \ref{th:2}) that for any $r>0$, we have that for $\lambda<\lambda_0$ and any $\eps,\eps'>0$, we have that for $p>n/(2k)$, 
\begin{eqnarray*}
&&\Vert \varphi_\eps-\varphi_{\eps'}\Vert_{H_{2k}^p(M\setminus B_r(x_0))}\\
&&\leq C( p,r, k,L,\lambda_0)\left(\Vert (V_\eps-V_{\eps'})\varphi_{\eps'}\Vert_{L^p(M\setminus B_{r/2}(x_0))}+\Vert \varphi_\eps-\varphi_{\eps'}\Vert_{L^2(M\setminus B_r(x_0))}\right).
\end{eqnarray*}
Note that \eqref{lim:phi:eps} yields $\Vert\varphi_\eps-\varphi_0\Vert_2\to 0$ as $\eps\to 0$. Therefore, with \eqref{hyp:Ve}, \eqref{ineq:co:phi}, Lebesgue's convergence theorem, we get that $(\varphi_\eps)_\eps$ has a limit in $H_{2k,loc}^p(M\setminus \{x_0\})$, and since $\Vert\varphi_\eps-\varphi_0\Vert_2\to 0$ as $\eps\to 0$, we get that this limit is $\varphi_0$ and $\lim_{\eps\to 0}\varphi_\eps=\varphi_0$ in  $H_{2k,loc}^p(M\setminus \{x_0\})$. With Sobolev's embedding Theorem, we get that $\varphi_0\in C^0(M\setminus\{x_0\})$ and 
\begin{equation}\label{lim:C0}
\lim_{\eps\to 0}\varphi_\eps=\varphi_0\hbox{ in }C^0_{loc}(\Mmx)
\end{equation}
We now write Green's formula for $\varphi_\eps$ to get
$$\varphi_\eps(x)=\int_{M}G_\eps(x,\cdot)f\, dv_g\hbox{ for }x\neq x_0\hbox{ and for all }\eps>0.$$
With \eqref{eq:16}, \eqref{lim:Ge}, \eqref{eq:17}, \eqref{eq:16:crit} and \eqref{lim:C0}, we pass to the limit in Green's formula to get
$$\varphi_0(x)=\int_{M}G_0(x,\cdot)f\, dv_g.$$
This yields the existence of a Green's function for $P-V_0$ in Theorem \ref{th:Green:main}. Concerning uniqueness, let us consider another Green's function as in Theorem \ref{th:Green:main}, say $\bar{G}_0$, and, given $x\in \Mmx$, let us define $H_x:=G_0(x,\cdot)-\bar{G}_0(x,\cdot)$. We then get that $H_x\in L^q(M)$ for all $1\leq q<\frac{n}{n-2k}$ and
$$\int_MH_x f\, dv_g=0\hbox{ for all }f\in C^0(M).$$
By density, this identity is also valid for all $f\in L^{q'}(M)$ where $\frac{1}{q}+\frac{1}{q'}=1$. By duality, this yields $H_x\equiv 0$, and then $\bar{G}_0=G_0$, which proves uniqueness. This ends the proof of Theorem \ref{th:Green:main}.

\medskip\noindent{\bf Step 4: First pointwise control.} As above, we fix $\delta>0$ and we take $x\in M$ such that $d_g(x,x_0)>\delta$. It follows from \eqref{eq:Ge}, \eqref{est:G:1}, \eqref{lim:Ge} and regularity theory (see Theorem \ref{th:2}) that for all $l\in \{0,...,2k-1\}$, we have that
\begin{equation}\label{est:G:2}
| \nabla_y^lG_\eps(x,y)|\leq  C(\delta, k, \lambda_0, L)\hbox{ for  }\{d_g(x,y)\geq\delta,\, d_g(x,x_0)\geq\delta,\,d_g(y,x_0)\geq\delta\}.
\end{equation}
and
\begin{equation}\label{est:G:3}
| \nabla_y^lG_0(x,y)|\leq  C(\delta, k, \lambda_0, L)\hbox{ for  }\{d_g(x,y)\geq\delta,\, d_g(x,x_0)\geq\delta,\,d_g(y,x_0)\geq\delta\}.
\end{equation}
We fix $\gamma\in (0,n-2k)$. Since $G_\eps(x,\cdot)$ satisfies \eqref{eq:Ge} in the weak sense, it follows from Lemma \ref{lem:main} that there exists $\lambda=\lambda(\gamma,L,\delta)>0$, there exists $C=C(\gamma,p, L,\delta)>0$ such that
\begin{equation*}
d_g(y,x_0)^{\gamma}|G_\eps(x,y)|\leq C  \Vert G_\eps(x,\cdot)\Vert_{L^p(B_{\delta/2}(x_0))}\hbox{ for all }y\in B_{\delta/3}(x_0)-\{x_0\}
\end{equation*}
when $d_g(x,x_0)\geq \delta$. It then follows from \eqref{eq:16} that
\begin{equation}\label{ineq:36}
d_g(y,x_0)^{\gamma}|G_\eps(x,y)|\leq C(\delta,k,L,\gamma) \hbox{ for all }y\in B_{\delta/2}(x_0)-\{x_0\}\hbox{ and }d_g(x,x_0)\geq \delta.
\end{equation}
These inequalities are valid for $\eps>0$, and then for $\eps=0$. In order to get the full estimates of Theorem \ref{th:Green:pointwise:BIS}, we now perform infinitesimal versions of these estimates.

\section{Asymptotics for the Green's function close to the singularity}\label{sec:green2}
This section is devoted to the proof of an infinitesimal version of \eqref{est:G:2} and \eqref{ineq:36} for $x,y$ close to the singularity $x_0$. 

\begin{theorem}\label{th:G:close} Let $(M,g)$ be a compact Riemannian manifold of dimension $n$. Fix $k\in\nn$ such that $2\leq 2k<n$, $L>0$, $\lambda>0$ and $x_0\in M$. Fix an operator $P$ of type $O_{k,L}$ (see Definition \ref{def:p}), $V\in \Pl(x_0)$ and a family $(V_\eps)$ as in \eqref{hyp:Ve}. For $\lambda>0$ sufficiently small, let $G$ (resp. $G_\eps$) be the Green's function for $P-V$ (resp. $P-V_\eps$). Let us fix $\omega,U$ two open subsets of $\rn$ such that
\begin{equation*}
\omega\subset\subset \rn-\{0\}\, ,\, U\subset\subset \rn \hbox{ and }\overline{\omega}\cap\overline{U}=\emptyset.
\end{equation*}
We let $\alpha_0(M,g,U,\omega)>0$ be such that $|\alpha X|<i_g(M)/2$ for all $0<\alpha<\alpha_0$ and $X\in U\cup \omega$.  We fix $\gamma\in (0, n-2k)$. Then there exists $\lambda=\lambda(\gamma)>0$, there exists $C(U,\omega, M,g,\lambda, k, L)>0$ such that 
\begin{equation}\label{ineq:G:close}
\left| |X|^\gamma \alpha^{n-2k}G_\eps(\tilde{\hbox{exp}}_{x_0}(\alpha X),\tilde{\hbox{exp}}_{x_0}(\alpha Y) )\right|\leq C(U,\omega, M,g,\lambda, k, L)
\end{equation}
for all $X\in U-\{0\}$, $Y\in\omega$, $\alpha\in (0,\alpha_0)$ and $\eps\geq 0$.
\end{theorem}

\noindent{\it Proof of Theorem \ref{th:G:close}.} We first set $\omega',U'$ two open subsets of $\rn$ such that
\begin{equation*}
\omega\subset\subset \omega'\subset\subset \rn-\{0\}\, ,\, U\subset\subset U'\subset\subset \rn \hbox{ and }\overline{\omega'}\cap\overline{U'}=\emptyset.
\end{equation*}
We fix $f\in C^\infty_c(\omega')$ and for any $0<\alpha<\alpha_0$, we set 
$$f_\alpha(x):=\frac{1}{\alpha^{\frac{n+2k}{2}}}f\left(\frac{\tilde{\hbox{exp}}_{x_0}^{-1}(x)}{\alpha}\right)\hbox{ for all }x\in M.$$
As one checks, $f_\alpha\in C^\infty_c(\tilde{\hbox{exp}}_{x_0}(\alpha \omega'))$ and $\tilde{\hbox{exp}}_{x_0}(\alpha \omega')\subset\subset M-\{x_0\}$. It follows from Theorem \ref{th:3} that there exists $\varphi_{\alpha,\eps}\in H_{2k}^q(M)$ for all $q>1$ be such that
\begin{equation}\label{eq:phi:alpha}
P\varphi_{\alpha,\eps}-V_\eps\varphi_{\alpha,\eps}=f_\alpha\hbox{ in }M.
\end{equation}
It follows from Sobolev's embedding theorem that $\varphi_{\alpha,\eps}\in C^{2k-1}(M-\{x_0\})$. We define
\begin{equation}\label{def:tilde:phi}
\tilde{\varphi}_{\alpha,\eps}(X):= \alpha^{\frac{n-2k}{2}}\varphi_{\alpha,\eps}\left(\tilde{\hbox{exp}}_{x_0}(\alpha X)\right)\hbox{ for all }X\in \rn-\{0\},\, |\alpha X|<i_g(M).
\end{equation}
A change of variable and upper-bounds for the metric yield
\begin{eqnarray*}
\Vert f_\alpha\Vert_{\frac{2n}{n+2k}}^{\frac{2n}{n+2k}}&=&\int_M |f_\alpha(x)|^{\frac{2n}{n+2k}}\, dv_g=\int_{\tilde{\hbox{exp}}_{x_0}(\alpha \omega')} |f_\alpha(x)|^{\frac{2n}{n+2k}}\, dv_g\\
&=& \int_{\omega'} |f(X)|^{\frac{2n}{n+2k}}\, dv_{g_\alpha}(X)\leq C\int_{\omega'} |f(X)|^{\frac{2n}{n+2k}}\, dX,
\end{eqnarray*}
where here and in the sequel, $\tilde{\hbox{exp}}_p^\star g$ denotes the pull-back metric of $g$ via the map $\tilde{\hbox{exp}}_p$, $p\in M$ and  $g_\alpha:=(\tilde{\hbox{exp}}_{x_0}^\star g)(\alpha\cdot)$. Therefore
\begin{equation}\label{eq:31}
\Vert f_\alpha\Vert_{L^{\frac{2n}{n+2k}}(M)}\leq C(k)\Vert f\Vert_{L^{\frac{2n}{n+2k}}(\omega')}.
\end{equation}
where $C(k)$ depends only on $(M,g)$ and $k$. With \eqref{coer:PV}, \eqref{eq:phi:alpha} and the Sobolev inequality \eqref{sobo:ineq:M}, we get
\begin{eqnarray*}
\frac{1}{2L}\Vert \varphi_{\alpha,\eps}\Vert_{H_k^2(M)}^2&\leq& \int_M \varphi_{\alpha,\eps}(P-V_\eps)\varphi_{\alpha,\eps}\, dv_g=\int_M f_\alpha\varphi_{\alpha,\eps} \, dv_g\\
&\leq & \Vert f_\alpha\Vert_{L^{\frac{2n}{n+2k}}(M)} \Vert \varphi_{\alpha,\eps}\Vert_{L^{\frac{2n}{n-2k}}(M)}\leq C_S(k) \Vert f_\alpha\Vert_{L^{\frac{2n}{n+2k}}(M)} \Vert \varphi_{\alpha,\eps}\Vert_{H_k^2(M)}.
\end{eqnarray*}
Therefore, using again the Sobolev inequality \eqref{sobo:ineq:M} and \eqref{eq:31}, we get that
\begin{equation}\label{eq:33}
\Vert \varphi_{\alpha,\eps}\Vert_{L^{\frac{2n}{n-2k}}(M)}\leq C(k,L) \Vert f\Vert_{L^{\frac{2n}{n+2k}}(\omega')} 
\end{equation}
With $g_\alpha$ as above, equation \eqref{eq:phi:alpha} rewrites
\begin{equation}\label{eq:tpl}
\Delta_{g_\alpha}^k\tilde{\varphi}_{\alpha,\eps}+\sum_{l=0}^{2k-2}\alpha^{2k-l}B_l(\tilde{\hbox{exp}}_{x_0}(\alpha \cdot))\star\nabla_{g_\alpha}^l\tilde{\varphi}_{\alpha,\eps}-\alpha^{2k}V_\eps(\tilde{\hbox{exp}}_{x_0}(\alpha X))\tilde{\varphi}_{\alpha,\eps}=f
\end{equation}
weakly locally in $\rn$ where the $B_l$'s are $(l,0)-$tensors that are bounded in $L^\infty$ due to Definition \eqref{def:p}. Since $V_\eps$ satisfies \eqref{hyp:Ve}, we have that
\begin{equation*}
\left|\alpha^{2k}V_\eps(\tilde{\hbox{exp}}_{x_0}(\alpha X))\right|\leq \lambda |X|^{-2k}\hbox{ for all }X\in U'-\{0\}
\end{equation*}
Since $f(X)=0$ for all $X\in U'$ and $\tilde{\varphi}_{\alpha,\eps}\in H_{2k,loc}^q(U')$, it follows from the regularity Lemma \ref{lem:main} that there exists $\lambda=\lambda(\gamma)>0$ such that for any  $\delta>0$ such that $B_{\delta}(0)\subset\subset U'$, there exists $C( L, \delta,\gamma, U')>0$ such that
\begin{equation*}
|X|^\gamma \left|\tilde{\varphi}_{\alpha,\eps}(X)\right|\leq 
C( L, \delta,\gamma, U,U')\Vert \tilde{\varphi}_{\alpha,\eps} \Vert_{L^{\crit}(U')}\hbox{ for all }X\in B_{\delta}(0)-\{0\}
\end{equation*}
Since the coefficients of $P-V_\eps$ are uniformly bounded outside $x_0$, classical elliptic regularity yields
\begin{equation*}
 \left|\tilde{\varphi}_{\alpha,\eps}(X)\right|\leq C( L, \delta,\gamma,U,U')\Vert \tilde{\varphi}_{\alpha,\eps} \Vert_{L^{\crit}(U')}\hbox{ for all }X\in U- B_{\delta}(0)
\end{equation*}
These two inequalities yield the existence of $C( L, \delta,\gamma,U,U')$ such that
\begin{equation}\label{eq:44}
|X|^\gamma \left|\tilde{\varphi}_{\alpha,\eps}(X)\right|\leq C\Vert \tilde{\varphi}_{\alpha,\eps} \Vert_{L^{\crit}(U')}\hbox{ for all }X\in U-\{0\}
\end{equation}
Arguing as in the proof of \eqref{eq:31}, we have that
\begin{equation}\label{eq:32}
\Vert \tilde{\varphi}_{\alpha,\eps}\Vert_{L^{\frac{2n}{n-2k}}(U')}\leq C(k)\Vert  \varphi_{\alpha,\eps}\Vert_{L^{\frac{2n}{n-2k}}(M)}.
\end{equation}
Putting together \eqref{def:tilde:phi}, \eqref{eq:33}, \eqref{eq:44} and \eqref{eq:32} we get that
\begin{equation}\label{ineq:G:a}
|X|^\gamma \left| \alpha^{\frac{n-2k}{2}}\varphi_{\alpha,\eps}\left(\tilde{\hbox{exp}}_{x_0}(\alpha X)\right) \right|\leq C( L, \delta, \lambda,\gamma,U,U')\Vert f\Vert_{L^{\frac{2n}{n+2k}}(\omega')}
\end{equation}
for all $X\in U-\{0\}$. For $\alpha>0$, we define
\begin{equation}\label{def:G:t:close}
\tilde{G}_{\alpha,\eps}(X,Y):=\alpha^{n-2k}G_\eps(\tilde{\hbox{exp}}_{x_0}(\alpha X),\tilde{\hbox{exp}}_{x_0}(\alpha Y) )\hbox{ for }(X,Y)\in U'\times \omega',\, X\neq 0
\end{equation}
It follows from Green's representation formula for $G_\eps$, $\eps>0$, and \eqref{eq:phi:alpha} that
$$\varphi_{\alpha,\eps}\left(\tilde{\hbox{exp}}_{x_0}(\alpha X)\right)=\int_M G_\eps\left(\tilde{\hbox{exp}}_{x_0}(\alpha X),y\right)f_\alpha(y)\, dv_g$$
for all $X\in U-\{0\}$. 
With a change of variable, we then get that
\begin{equation}\label{ineq:G:b}
\alpha^{\frac{n-2k}{2}}\varphi_{\alpha,\eps}\left(\tilde{\hbox{exp}}_{x_0}(\alpha X)\right)=\int_{\omega'} \tilde{G}_{\alpha,\eps}(X,Y) f(Y)\, dv_{g_\alpha}
\end{equation}
for all $X\in U-\{0\}$. Putting together \eqref{ineq:G:a} and \eqref{ineq:G:b}, we get that
\begin{equation*}
\left||X|^\gamma\int_{\omega'} \tilde{G}_{\alpha,\eps}(X,Y) f(Y)\, dv_{g_\alpha}
\right|\leq C( L, \delta, \lambda,\gamma,U,U',\omega') \Vert f\Vert_{L^{\frac{2n}{n+2k}}(\omega')}\end{equation*}
for all $f\in C^\infty_c(\omega')$ and $ X\in U-\{0\}$. It then follows from duality arguments that
\begin{equation}\label{ineq:35}
\Vert |X|^\gamma \tilde{G}_{\alpha,\eps}(X,\cdot)\Vert_{L^{\crit}(\omega')}\leq C( L, \delta, \lambda,\gamma,U,U',\omega')\hbox{ for }X\in U-\{0\} 
\end{equation}
Since $G_\eps(x,\cdot)$ is a solution to $(P-V_\eps)G_\eps(x,\cdot)=0$ in $M-\{x_0,x\}$, as in \eqref{eq:tpl},
we get that
\begin{eqnarray*}
&&\Delta_{g_\alpha}^k\tilde{G}_{\alpha,\eps}(X,\cdot)+\sum_{l=0}^{2k-2}\alpha^{2k-l}B_l(\tilde{\hbox{exp}}_{x_0}(\alpha \cdot))\star\nabla_{g_\alpha}^l\tilde{G}_{\alpha,\eps}(X,\cdot)\\
&&-\alpha^{2k}V_\eps(\tilde{\hbox{exp}}_{x_0}(\alpha \cdot))\tilde{G}_{\alpha,\eps}(X,\cdot)=0\hbox{ in }\omega'\subset\subset \rn-\{0,X\} 
\end{eqnarray*}
Since $\omega'\subset\subset \rn-\{0,X\}$, there exists $c_{\omega'}>0$ such that $|Y|\geq c_{\omega'}$ for all $Y\in\omega'$. Since $V_\eps$ satisfies \eqref{hyp:Ve}, we have that
\begin{equation*}
\left|\alpha^{2k}V_\eps(\tilde{\hbox{exp}}_{x_0}(\alpha Y))\right|\leq \lambda c_{\omega'}^{-2k}\hbox{ for all }Y\in \omega'.
\end{equation*}
It then follows from elliptic regularity theory (see Theorem \ref{th:2}) that
\begin{equation*}
|X|^\gamma |\tilde{G}_{\alpha,\eps}(X,Y)|\leq C(k,L, \lambda,\omega',U')\Vert |X|^\gamma \tilde{G}_{\alpha,\eps}(X,\cdot)\Vert_{L^{\crit}(\omega')}
\end{equation*}
for all $Y\in \omega\subset\subset\omega'$ and $X\in U-\{0\}$. The conclusion \eqref{ineq:G:close}  of Theorem \ref{th:G:close} then follows from this inequality, \eqref{ineq:35}, Definition \eqref{def:G:t:close} of $\tilde{G}_{\alpha,\eps}$ and the limit \eqref{lim:Ge}.

\section{Asymptotics for the Green's function far from the singularity}\label{sec:green3}
This section is devoted to the proof of an infinitesimal version of \eqref{est:G:2} and \eqref{ineq:36} for $x,y$ far from the singularity $x_0$. 

\begin{theorem}\label{th:G:far} We fix $p\in M-\{x_0\}$ and $\omega,U$ two open subsets of $\rn$ such that
\begin{equation*}
\omega\subset\subset \rn \, ,\, U\subset\subset \rn \hbox{ and }\overline{\omega}\cap\overline{U}=\emptyset.
\end{equation*}
We let $\alpha_0>0$ be such that 
\begin{equation}\label{hyp:alpha}
|\alpha X|<\frac{1}{2}\min\{i_g(M),d_g(x_0,p)\} \hbox{ for all }0<\alpha<\alpha_0\hbox{ and }X\in U\cup \omega.
\end{equation}
Then for all $\gamma\in (0,n-2k)$, there exists $\lambda=\lambda(\gamma)>0$, there exists $C(U,\omega, L,\alpha_0,\gamma,\lambda)>0$ such that 
\begin{equation}\label{ineq:G:far}
\left|  \alpha^{n-2k}G_\eps(\tilde{\hbox{exp}}_{p}(\alpha X),\tilde{\hbox{exp}}_{p}(\alpha Y) )\right|\leq C(U,\omega, L,\alpha_0,\gamma,\lambda)
\end{equation}
for all $X\in U $ and $Y\in\omega$ and $\alpha\in (0,\alpha_0)$ and $\eps\geq 0$ small enough.
\end{theorem}

\noindent{\it Proof of Theorem \ref{th:G:far}.} We first set $\omega',U'$ two open subsets of $\rn$ such that
\begin{equation*}
\omega\subset\subset \omega'\subset\subset \rn \, ,\, U\subset\subset U'\subset\subset \rn \hbox{ and }\overline{\omega'}\cap\overline{U'}=\emptyset
\end{equation*}
and \eqref{hyp:alpha} still holds for $X\in U'\cup\omega'$. We fix $f\in C^\infty_c(\omega')$ and for any $0<\alpha<\alpha_0$, we set 
$$f_\alpha(x):=\frac{1}{\alpha^{\frac{n+2k}{2}}}f\left(\frac{\tilde{\hbox{exp}}_{p}^{-1}(x)}{\alpha}\right)\hbox{ for all }x\in M.$$
As one checks, $f_\alpha\in C^\infty_c(\tilde{\hbox{exp}}_{p}(\alpha \omega'))$ and $\tilde{\hbox{exp}}_{p}(\alpha \omega')\subset\subset M-\{x_0\}$. It follows from Theorem \ref{th:3} that there exists $\varphi_{\alpha,\eps}\in H_{2k}^q(M)$ for all $q>1$ such that
\begin{equation}\label{eq:phi:alpha:far}
P\varphi_{\alpha,\eps}-V_\eps\varphi_{\alpha,\eps}=f_\alpha\hbox{ in }M.
\end{equation}
It follows from Sobolev's embedding theorem that $\varphi_{\alpha,\eps}\in C^{2k-1}(M-\{x_0\})$. We define
\begin{equation}\label{def:tilde:phi:far}
\tilde{\varphi}_{\alpha,\eps}(X):= \alpha^{\frac{n-2k}{2}}\varphi_{\alpha,\eps}\left(\tilde{\hbox{exp}}_{p}(\alpha X)\right)\hbox{ for all }X\in \rn,\, |\alpha X|<i_g(M).
\end{equation}
As in \eqref{eq:31}, a change of variable yields
\begin{equation}\label{eq:31:far}
\Vert f_\alpha\Vert_{L^{\frac{2n}{n+2k}}(M)}\leq C(k)\Vert f\Vert_{L^{\frac{2n}{n+2k}}(\omega')}.
\end{equation}
As for \eqref{eq:33}, we also get that
\begin{equation}\label{eq:33:far}
\Vert \varphi_{\alpha,\eps}\Vert_{L^{\frac{2n}{n-2k}}(M)}\leq C(k,L) \Vert f\Vert_{L^{\frac{2n}{n+2k}}(\omega')} 
\end{equation}
Taking $g_\alpha:=(\tilde{\hbox{exp}}_{p}^\star g)(\alpha\cdot)$, equation \eqref{eq:phi:alpha:far} rewrites
\begin{equation}\label{eq:tpl:far}
\Delta_{g_\alpha}^k\tilde{\varphi}_{\alpha,\eps}+\sum_{l=0}^{2k-2}\alpha^{2k-l}B_l(\tilde{\hbox{exp}}_{p}(\alpha \cdot))\star\nabla_{g_\alpha}^l\tilde{\varphi}_{\alpha,\eps}-\alpha^{2k}V_\eps(\tilde{\hbox{exp}}_{p}(\alpha X))\tilde{\varphi}_{\alpha,\eps}=f
\end{equation}
weakly in $\rn$ where the $B_l$'s are $(l,0)-$tensors that are bounded in $L^\infty$ due to Definition \eqref{def:p}. Since $V_\eps$ satisfies \eqref{hyp:Ve}, we have that
\begin{equation*}
\left|\alpha^{2k}V_\eps(\tilde{\hbox{exp}}_{p}(\alpha X))\right|\leq \lambda \alpha^{2k}d_g(x_0,(\tilde{\hbox{exp}}_{p}(\alpha X) )^{-2k}\hbox{ for all }X\in U' 
\end{equation*}
With \eqref{hyp:alpha}, we have that
$$d_g(x_0,(\tilde{\hbox{exp}}_{p}(\alpha X) )\geq d_g(x_0,p)-\alpha|X|\geq \frac{2}{3}d_g(x_0,p)$$
for all $X\in U'$, and therefore, we get that 
\begin{equation*}
\left|\alpha^{2k}V_\eps(\tilde{\hbox{exp}}_{p}(\alpha X))\right|\leq \lambda\left(\frac{d_g(x_0,p)}{\alpha}\right)^{-2k} \leq C(\lambda_0)\hbox{ for all }X\in U' 
\end{equation*}
Since $f(X)=0$ for all $X\in U'$, it follows from standard regularity theory (see Theorem \ref{th:2}) that there exists $C( k,L, U, U',\omega,\omega',\alpha_0)>0$ such that
\begin{equation}\label{eq:64}
  \left|\tilde{\varphi}_{\alpha,\eps}(X)\right|\leq C\Vert \tilde{\varphi}_{\alpha,\eps} \Vert_{L^{\crit}(U')}\hbox{ for all }X\in U
\end{equation}
Arguing as in the proof of \eqref{eq:31}, we have that
\begin{equation}\label{eq:32:far}
\Vert \tilde{\varphi}_{\alpha,\eps}\Vert_{L^{\frac{2n}{n-2k}}(U')}\leq C(k)\Vert  \varphi_{\alpha,\eps}\Vert_{L^{\frac{2n}{n-2k}}(M)}.
\end{equation}
Putting together \eqref{def:tilde:phi:far}, \eqref{eq:64}, \eqref{eq:32:far} and \eqref{eq:33:far} we get that
\begin{equation}\label{ineq:G:a:far}
  \left| \alpha^{\frac{n-2k}{2}}\varphi_{\alpha,\eps}\left(\tilde{\hbox{exp}}_{p}(\alpha X)\right) \right|\leq C( k,L, U, U',\lambda) \Vert f\Vert_{L^{\frac{2n}{n+2k}}(\omega')} \hbox{ for all }X\in U.
\end{equation}
We now just follow verbatim the proof of Theorem \ref{th:G:close} above to get the conclusion \eqref{ineq:G:far}  of Theorem \ref{th:G:far}. We leave the details to the reader.

\section{Proof of Theorem \ref{th:Green:pointwise:BIS}}\label{sec:green4}
We let $(M,g)$, $k$, $\lambda$, $L$, $P$, $V$ as in the statement of Theorem \ref{th:Green:pointwise:BIS}. With $\lambda>0$ small enough, we let $G_0$ be the Green's function of $P-V$ as in Theorem \ref{th:Green:main}. Given $\gamma\in (0,n-2k)$, we let $\lambda_\gamma>0$ as in \eqref{ineq:36} and Theorems \ref{th:G:close} and \ref{th:G:far} hold when $0<\lambda<\lambda_\gamma$. We prove here the first estimate of Theorem \ref{th:Green:pointwise:BIS}, that is the existence of $C=C(\gamma,\lambda, L)>0$ such that
\begin{equation}\label{est:G:final}
|G(x,y)|\leq C(\gamma,\lambda, L) \left(\frac{\max\{d_g(x,x_0),d_g(y,x_0)\}}{\min\{d_g(x,x_0),d_g(y,x_0)\}}\right)^\gamma d_g(x,y)^{2k-n}
\end{equation}
for all $x,y\in M-\{x_0\}$, $x\neq y$. We argue by contradiction and we assume that there is a family of operators $(P_i)_{i\in\nn}\in O_{k,L}$, a family of potentials $(V_i)_{i\in\nn}\in {\mathcal P}_{\lambda_\gamma}$, sequences $(x_i), (y_i)\in M-\{x_0\}$ such that $x_i\neq y_i$ for all $i\in\nn$ and
\begin{equation}\label{hyp:absurd}
\lim_{i\to +\infty}\frac{d_g(x_i,y_i)^{n-2k}|G_i(x_i,y_i)|}{\left(\frac{\max\{d(x_i,x_0),d(y_i,x_0)\}}{\min\{d(x_i,x_0),d(y_i,x_0)\}}\right)^\gamma }=+\infty,
\end{equation}
where $G_i$ denotes the Green's function of $P_i-V_i$ for all $i\in\nn$. We distinguish 5 cases:

\smallskip\noindent{\it Case 1:} $d_g(x_i, y_i)=o(d_g(x_i, x_0))$ as $i\to +\infty$. It then follows from the triangle inequality that $d_g(x_i,y_i)=o(d_g(y_i,x_0))$ and $d_g(x_i, x_0)=(1+o(1))d_g(y_i, x_0)$. Therefore
$$\left(\frac{\max\{d(x_i,x_0),d(y_i,x_0)\}}{\min\{d(x_i,x_0),d(y_i,x_0)\}}\right)^\gamma =1+o(1)$$
and then \eqref{hyp:absurd} yields
\begin{equation}\label{hyp:absurd:1}
\lim_{i\to +\infty} d_g(x_i,y_i)^{n-2k}|G_i(x_i,y_i)|=+\infty
\end{equation}
We let $Y_i\in\rn$ be such that $y_i:=\tilde{\hbox{exp}}_{x_i}(d_g(x_i, y_i)Y_i)$. In particular, $|Y_i|=1$, so, up to a subsequence, there exists $Y_\infty\in\rn$ such that $\lim_{n\to +\infty}Y_i=Y_\infty$ with $|Y_\infty|=1$
We apply Theorem \ref{th:G:far} with $p:=x_i$, $\alpha:=d_g(x_i, y_i)$, $U=B_{1/3}(0)$, $\omega=B_{1/3}(Y_\infty)$: for $i\in\nn$ large enough, taking $X=0$ and $Y=Y_i$ in \eqref{ineq:G:far}, we get that
$$d_g(x_i, y_i)^{n-2k}|G_i(x_i,y_i )|=d_g(x_i, y_i)^{n-2k}|G_i(\tilde{\hbox{exp}}_{x_i}(d_g(x_i, y_i)\cdot 0),\tilde{\hbox{exp}}_{x_i}(d_g(x_i, y_i)\cdot Y_i) )|\leq C(L,\gamma,\lambda)$$
which contradicts \eqref{hyp:absurd:1}. This ends Case 1.

\smallskip\noindent The case  $d_g(x_i, y_i)=o(d_g(x_i, x_0))$ as $i\to +\infty$ is equivalent to Case 1.

\medskip\noindent{\it Case 2:} $d_g(x_i, x_0)=o(d_g(x_i, y_i))$ and $d_g(x_i, y_i)\not\to 0$ as $i\to +\infty$. Therefore  \eqref{hyp:absurd} rewrites
\begin{equation}\label{hyp:absurd:2bis}
\lim_{i\to +\infty}  d(x_i,x_0)^{\gamma}|G_i(x_i,y_i)| =+\infty
\end{equation}
This is a contradiction with \eqref{ineq:36} when $\eps=0$. This ends Case 2.

\medskip\noindent{\it Case 3:} $d_g(x_i, x_0)=o(d_g(x_i, y_i))$ and $d_g(x_i, y_i)\to 0$ as $i\to +\infty$. Then $d_g(x_i,x_0)=o(d_g(y_i, x_0))$ and  $d_g(x_i, y_i)=(1+o(1))d_g(y_i, x_0)$. Therefore  \eqref{hyp:absurd} rewrites
\begin{equation}\label{hyp:absurd:2}
\lim_{i\to +\infty}   d_g(y_i, x_i)^{n-2k-\gamma}d(x_i,x_0)^{\gamma}|G_i(x_i,y_i)| =+\infty
\end{equation}
We let $X_i, Y_i\in\rn$ be such that  $x_i:=\tilde{\hbox{exp}}_{x_0}(d_g(x_i, y_i)X_i)$ and  $y_i:=\tilde{\hbox{exp}}_{x_0}(d_g(x_i, y_i)Y_i)$. In particular, $\lim_{i\to +\infty}|X_i|=0$ and $|Y_i|=1+o(1)$. So, up to a subsequence, there exists $Y_\infty\in\rn$ such that $\lim_{n\to +\infty}Y_i=Y_\infty$ with $|Y_\infty|=1$. We apply Theorem \ref{th:G:close} with  $\alpha:=d_g(x_i, y_i)$, $U=B_{1/3}(0)$, $\omega=B_{1/3}(Y_\infty)$: for $i\in\nn$ large enough, taking $X=X_i\neq 0$ and $Y=Y_i$ in \eqref{ineq:G:close}, we get that
$$|X_i|^\gamma d_g(x_i,y_i)^{n-2k}|G_i(\tilde{\hbox{exp}}_{x_0}(d_g(x_i, y_i)X_i),\tilde{\hbox{exp}}_{x_0}(d_g(x_i, y_i)Y_i))|\leq C(\lambda,k,L),$$
and, coming back to the definitions of $X_i$ and $Y_i$, we get that
$$d_g(x_i,x_0)^\gamma d_g(x_i,y_i)^{n-2k-\gamma}|G_i(x_i,y_i)|\leq C(\lambda,k,L),$$
which contradicts \eqref{hyp:absurd:2}. This ends Case 3.

\medskip\noindent{\it Case 4:} $d_g(y_i, x_0)=o(d_g(x_i, y_i))$ as $i\to +\infty$. Since the Green's function is symmetric, this is similar to Case 2 and 3.

\medskip\noindent{\it Case 5:}  $d_g(x_i, x_0)\asymp  d_g( y_i, x_0)\asymp  d_g(x_i, y_i) $. Then \eqref{hyp:absurd} rewrites
\begin{equation}\label{hyp:absurd:3}
\lim_{i\to +\infty} d_g(x_i,y_i)^{n-2k}|G_i(x_i,y_i)|=+\infty
\end{equation}
\smallskip{\it Case 5.1:} $d_g(x_i, y_i)\not\to 0$ as $i\to +\infty$. Then it follows from \eqref{est:G:2} that $|G_i(x_i,y_i)|\leq C(\lambda,k,L)$ for all $i$, which contradicts \eqref{hyp:absurd:3}.\par

\smallskip\noindent {\it Case 5.2:} $d_g(x_i, y_i)\to 0$ as $i\to +\infty$. We let $X_i, Y_i\in\rn$ be such that  $x_i:=\tilde{\hbox{exp}}_{x_0}(d_g(x_i, y_i)X_i)$ and  $y_i:=\tilde{\hbox{exp}}_{x_0}(d_g(x_i, y_i)Y_i)$. In particular, there exists $c>0$ such that $c^{-1}<|X_i|,|Y_i|<c$ and $|X_i-Y_i|\geq c^{-1}$ for all $i$. So, up to a subsequence, there exists $X_\infty,Y_\infty\in\rn$ such that $\lim_{n\to +\infty}X_i=X_\infty\neq 0$  and $\lim_{n\to +\infty}Y_i=Y_\infty\neq 0$ and $X_\infty\neq Y_\infty$. We apply Theorem \ref{th:G:close} with  $\alpha_i:=d_g(x_i, y_i)$, $U=B_{r_0}(X_\infty)$, $\omega=B_{r_0}(Y_\infty)$ for some $r_0>0$ small enough. So for $i\in\nn$ large enough, taking $X=X_i\neq 0$ and $Y=Y_i$ in \eqref{ineq:G:close}, we get that
\begin{equation*}
|X_i|^\gamma \alpha_i^{n-2k}\left|   G_i(\tilde{\hbox{exp}}_{x_0}(\alpha_i X_i),\tilde{\hbox{exp}}_{x_0}(\alpha_i Y_i) )\right|\leq C(U,\omega,L,\gamma,\lambda)
\end{equation*}
and, coming back to the definitions of $X_i$ and $Y_i$, we get that
$$  d_g(x_i,y_i)^{n-2k }|G_i(x_i,y_i)|\leq C,$$
which contradicts \eqref{hyp:absurd:3}. This ends Case 5.

\medskip\noindent Therefore, in all 5 cases, we have obtained a contradiction with \eqref{hyp:absurd}. This proves \eqref{est:G:final}, and then the first estimate of  Theorem \ref{th:Green:pointwise:BIS}. The proof of the estimates on the derivative uses the same method by contradiction, with a few more cases to study using regularity theory (Theorem \ref{th:2}). We leave the details to the reader.

\section{The regularity Lemma}\label{sec:lemma}
\begin{defi}\label{def:P:loc} Let $(X,g)$ be a Riemannian manifold of dimension $n$ and let us fix an interior point $x_0\in Int(X)$, $k\in\nn$ such that $2\leq 2k<n$ and $\delta,L>0$. We say that an operator $P$ is of type $O_{k,L}(B_\delta(x_0))$ if 
\begin{itemize}
\item[(i)] $P:=\Delta_g^k+\sum_{i=0}^{k-1} (-1)^i\nabla^i(A^{(i)} \nabla^i)$, where for all $i=0,...,k-1$, $ A^{(i)}\in C^{i}_{\chi}(M,\Lambda_{S}^{(0,2i)}(B_\delta(x_0)))$ is a family $C^{i}-$field of $(0,2i)-$tensors on $B_\delta(x_0)$.
\item[(ii)] For any $i$ , $A^{(i)}(T,S)=A^{(i)}(S,T)$ for any $(i,0)-$tensors $S$ and $T$. \par
\item[(iii)]  $\Vert A^{(i)}\Vert_{C^{i}}\leq L$ for all $i=0,...,k-1$\par
\end{itemize}
\end{defi}

\begin{lemma}\label{lem:main} Let $(X,g)$ be a Riemannian manifold  of dimension $n$ and let us fix an interior point $x_0\in Int(X)$, $k\in\nn$ such that $2\leq 2k<n$ and $\delta,L>0$. Fix $p>1$ and $\delta_1,\delta_2>0$ such that $0<\delta_1<\delta_2$. We consider a differential operator $P\in O_{k,L}(x_0,\delta_2)$.
Then for all $0<\gamma<n-2k$, there exists $\lambda=\lambda(\gamma,p, L,\delta_1,\delta_2)> 0$ and $C_0=C_0(X,g,\gamma,p, L,\delta_1,\delta_2)>0$ such that for any $V\in L^1(B_{\delta_2}(x_0))$ such that
$$|V(x)|\leq \lambda d_g(x,x_0)^{-2k}\hbox{ for all }x\in B_{\delta_2}(x_0),$$
then for any $\varphi\in H_k^2(B_{\delta_2}(x_0))\cap H_{2k,loc}^s(B_{\delta_2}(x_0)-\{x_0\})$ (for some $s>1$)  such that
$$P\varphi-V\cdot\varphi=0\hbox{ weakly in }H_k^2(B_{\delta_2}(x_0)),$$
then we have that
\begin{equation}\label{ineq:lem}
d_g(x,x_0)^{\gamma}|\varphi(x)|\leq C_0\cdot \Vert\varphi\Vert_{L^p(B_{\delta_2}(x_0))}\hbox{ for all }x\in B_{\delta_1}(x_0)-\{x_0\}.
\end{equation}
Moreover, for any $0< l<2k$, there exists $C_l=C_l(X,g,\gamma,p, L,\delta_1,\delta_2)>0$ such that
\begin{equation}\label{ineq:lem:l}
d_g(x,x_0)^{\gamma+l}|\nabla^l\varphi(x)|\leq C_l\cdot \Vert\varphi\Vert_{L^p(B_{\delta_2}(x_0))}\hbox{ for all }x\in B_{\delta_1}(x_0)-\{x_0\}
\end{equation}
\end{lemma}
The sequel of this section is devoted to the proof of this Lemma. We split the proof in three steps.

\medskip\noindent{\bf Step 1: Proof of \eqref{ineq:lem} when $V\equiv 0$ around $x_0$.} As a preliminary step, we prove the Lemma under the additional assumption that $V$ vanishes around $x_0$:
\begin{lemma}\label{lem:bis} Let $(X,g)$ be a Riemannian manifold  of dimension $n$ and let us fix an interior point $x_0\in Int(X)$, $k\in\nn$ such that $2\leq 2k<n$ and $\delta,L>0$. Fix $p>1$ and $\delta_1,\delta_2>0$ such that $0<\delta_1<\delta_2$. We consider a differential operator $P\in O_{k,L}(x_0,\delta_2)$.
Then for all $0<\gamma<n-2k$, there exists $\lambda=\lambda(\gamma,p, L,\delta_1,\delta_2)> 0$ and $C_0=C_0(\gamma,p, L,\delta_1,\delta_2)>0$ such that for any $V\in L^1(B_{\delta_2}(x_0))$ such that
$$|V(x)|\leq \lambda d_g(x,x_0)^{-2k}\hbox{ for all }x\in B_{\delta_2}(x_0),$$
and $$V\hbox{ vanishes around }x_0.$$
then for any $\varphi\in H_k^2(B_{\delta_2}(x_0))\cap H_{2k,loc}^s(B_{\delta_2}(x_0)-\{x_0\})$ (for some $s>1$)  such that
$$P\varphi-V\cdot\varphi=0\hbox{ weakly in }H_k^2(B_{\delta_2}(x_0)),$$
then we have that
\begin{equation}\label{ineq:lem:bis}
d_g(x,x_0)^{\gamma}|\varphi(x)|\leq C_0\cdot \Vert\varphi\Vert_{L^p(B_{\delta_2}(x_0))}\hbox{ for all }x\in B_{\delta_1}(x_0)-\{x_0\}.
\end{equation}
\end{lemma}

 For the reader's convenience, we set $\delta:=\delta_1$ and we assume that $\delta_2=3\delta_1=3\delta$. The general case follows the same proof by changing $2\delta$, $2.9\delta$, etc, into various radii $\delta',\delta^{\prime\prime},...$ such that $\delta_1<\delta'<\delta^{\prime\prime}<\delta_2$, etc.

\smallskip\noindent We prove \eqref{ineq:lem:bis} by contradiction. We assume that there exists $\gamma\in (0,n-2k)$, $p>1$, $L>0$, $\delta>0$ such that for all $\lambda>0$ and all $C>0$, there exists a differential operator $P_{\lambda,C}$ and a function $V_{\lambda,C}\in L^1(B_{3\delta}(x_0))$ such that
\begin{itemize}
\item[(i)] $P_{\lambda, C}:=\Delta_g^k+\sum_{i=0}^{k-1} (-1)^i\nabla^i((A^{(i)})_{\lambda,C}\nabla^i)$, with $(A^{(i)})_{\lambda, C}$  a symmetric $(0, 2i)-$tensor for all $i=0,...,k-1$, 
\item[(ii)] $(A^{(i)})_{\lambda, C}\in C^i$ and $\Vert (A^{(i)})_{\lambda, C}\Vert_{C^{i}(B_\delta(x_0))}\leq L$ for all $i=0,...,k-1$
\item[(iii)] $P_{\lambda, C}$ is symmetric in the $L^2-$sense (this will not be used in this Step 1).
\end{itemize}
with
$$|V_{\lambda, C}(x)|\leq \lambda d_g(x,x_0)^{-2k}\hbox{ for all }x\in B_{3\delta}(x_0)-\{x_0\}\hbox{ and }V_{\lambda,C}\equiv 0\hbox{ around }x_0$$
and there exists $\varphi_{\lambda,C}\in H_k^2(B_{3\delta}(x_0))\cap   H_{2k,loc}^s(B_{3\delta}(x_0)-\{x_0\})$ (for some $s>1$)  such that
$$P_{\lambda, C}\varphi_{\lambda, C}-V_{\lambda, C}\cdot\varphi_{\lambda, C}=0\hbox{ weakly in }H_k^2(B_{3\delta}(x_0)),$$
and there exists $x_{\lambda, C}\in B_\delta(x_0)-\{x_0\}$ such that  
$$d_g(x_{\lambda, C},x_0)^{\gamma}|\varphi_{\lambda, C}(x_{\lambda, C})|> C\cdot \Vert\varphi_{\lambda, C}\Vert_{L^p(B_{3\delta}(x_0))}.$$
For convenience, we will  write
$$P_{\lambda, C}:=\Delta_g^ku+\sum_{l=0}^{2k-2}  (B_l)_{\lambda,C}\star\nabla^l.$$
where the tensors $(B_l)_{\lambda,C}$ are bounded in $C^0(B_\delta(x_0))$.

\smallskip\noindent We choose $C:=\lambda^{-1}$, we define $P_\lambda:=P_{\lambda, \lambda^{-1}}$, $V_\lambda:=V_{\lambda, \lambda^{-1}}$ and 
$$\psi_\lambda:=\frac{\varphi_{\lambda, \lambda^{-1}}}{\Vert \varphi_{\lambda, \lambda^{-1}}\Vert_{L^p(B_{3\delta}(x_0)}}$$
so that we have that $\psi_\lambda\in H_k^2(B_{3\delta}(x_0))$ is such that
\begin{equation}\label{eq:75}
\left\{\begin{array}{l}
(P_\lambda-V_\lambda)\psi_\lambda=0 \hbox{ weakly in }H_k^2(B_{3\delta}(x_0))\cap H_{2k,loc}^s(B_{3\delta}(x_0)-\{x_0\})\\
\Vert\psi_\lambda\Vert_{L^p(B_{3\delta}(x_0)}=1\\
|V_{\lambda}(x)|\leq \lambda d_g(x,x_0)^{-2k}\hbox{ for all }x\in B_{3\delta}(x_0)-\{x_0\}\\
V_{\lambda}\equiv 0\hbox{ around }x_0\\
\sup_{x\in \overline{B_\delta(x_0)}} d_g(x,x_0)^{\gamma}|\psi_\lambda(x)|>\frac{1}{\lambda}\to +\infty\hbox{ as }\lambda\to 0
\end{array}\right\}
\end{equation} 
With our assumption that $V_\lambda$ vanishes around $x_0$, we get that $V_\lambda\in L^\infty(B_{3\delta}(x_0))$. Then, by regularity theory (see Theorem \ref{th:1}), we get that $\psi_\lambda\in C^0(B_{2\delta}(x_0))$. Therefore, there exists $x_\lambda\in \overline{B_\delta(x_0)}$ such that
\begin{equation}\label{lim:xlambda}
d_g(x_\lambda,x_0)^{\gamma}|\psi_\lambda(x_\lambda)|=\sup_{x\in \overline{B_\delta(x_0)}} d_g(x,x_0)^{\gamma}|\psi_\lambda(x)|>\frac{1}{\lambda}\to +\infty
\end{equation}
as $\lambda\to 0$. 

\medskip\noindent{\bf Step 1.1:} we claim that $\lim_{\lambda\to 0}x_\lambda=x_0$. 

\smallskip\noindent We prove the claim. For any $r>0$, we have that $|V_\lambda(x)|\leq \lambda r^{-2k}$ for all $x\in B_{3\delta}(x_0)\setminus B_r(x_0)$. So, with regularity theory (see Theorem \ref{th:1} and \ref{th:2}), we get that for all $q>1$, then $\Vert \psi_\lambda\Vert_{H_{2k}^q(B_{2\delta}(x_0)\setminus B_{2r}(x_0))}= C(r, q, L,p, )\Vert\psi_\lambda\Vert_{L^p(B_{3\delta}(x_0)}=C(r,q, L,p)$. Taking $q>\frac{n}{2k}$, we get that $|\psi_\lambda(x)|\leq C(r,q, L,p)$ for all $x\in B_{2\delta}(x_0)\setminus B_{2r}(x_0)$. With \eqref{lim:xlambda}, this forces $\lim_{\lambda\to 0}x_\lambda=x_0$. The claim is proved.

\medskip\noindent{\bf Step 1.2: Convergence after rescaling.} We set $r_\lambda:=d_g(x_\lambda,x_0)>0$ and we define
\begin{equation*}
\tilde{\psi}_\lambda(X):=\frac{\psi_\lambda(\tilde{\hbox{exp}}_{x_0}(r_\lambda X))}{\psi_\lambda(x_\lambda)}\hbox{ for }X\in\rn-\{0\}\hbox{ such that }|X|<\frac{\delta}{r_\lambda}.
\end{equation*}
We define $X_\lambda\in \rn$ such that $x_\lambda=\tilde{\hbox{exp}}_{x_0}(r_\lambda X_\lambda)$. In particular $|X_\lambda|=1$. With the definition of $x_\lambda$, for any $X\in \rn$ such that $0<|X|<\frac{\delta}{r_\lambda}$, we have that
\begin{eqnarray*}
r_\lambda^\gamma |X|^\gamma|\tpl(X)|&=& \frac{d_g(x_0, \tilde{\hbox{exp}}_{x_0}(r_\lambda X))^\gamma |\psi_\lambda(\tilde{\hbox{exp}}_{x_0}(r_\lambda X))|}{|\psi_\lambda(x_\lambda)|}\leq \frac{d_g(x_0, x_\lambda)^\gamma |\psi_\lambda(x_\lambda)|}{|\psi_\lambda(x_\lambda)|}=r_\lambda^\gamma.
\end{eqnarray*}
Therefore, we get that
\begin{equation}\label{bnd:tpl}
|X|^\gamma |\tpl(X)|\leq 1\hbox{ for all }X\in \rn\hbox{ such that }0<|X|<\frac{\delta}{r_\lambda}\hbox{ and }\tpl(X_\lambda)=1.
\end{equation}
Defining $g_\lambda:=(\tilde{\hbox{exp}}_{x_0}^\star g)(r_\lambda\cdot)$ the pull-back metric of $g$, the equation satisfied by $\tpl$ in \eqref{eq:75} rewrites
\begin{equation}\label{eq:tpl:bis}
\Delta_{g_\lambda}^k\tpl+\sum_{l=0}^{2k-2}r_\lambda^{2k-l}(B_l)_{\lambda}(\tilde{\hbox{exp}}_{x_0}(r_\lambda \cdot))\star\nabla_{g_\lambda}^l\tpl-r_\lambda^{2k}V_\lambda(\tilde{\hbox{exp}}_{x_0}(r_\lambda X))\tpl=0
\end{equation}
weakly in $B_{3\delta/r_\lambda}(0)-\{0\}$. Note that
\begin{equation}\label{bnd:Vl}
|r_\lambda^{2k}V_\lambda(\tilde{\hbox{exp}}_{x_0}(r_\lambda X))|\leq \lambda |X|^{-2k}\hbox{ for all }\lambda>0\hbox{ and }0<|X|<\frac{3\delta}{r_\lambda}.
\end{equation}
With the bound \eqref{bnd:tpl} and the bounds of the coefficients $(B_l)_\lambda$, it follows from regularity theory (see Theorems \ref{th:1} and \ref{th:2}) that for any $R>0$ and any $0<\nu<1$, there exists $C(R)>0$ such that
$$\Vert \tpl\Vert_{C^{2k-1,\nu}(B_R(0)-B_{R^{-1}}(0))}\leq C(R,\nu)\hbox{ for all }\lambda>0.$$
It follows from Ascoli's theorem, that, up to extraction, there exists $\tilde{\psi}\in C^{2k-1}(\rn-\{0\})$ such that $\tpl\to \tilde{\psi}$ in $C^{2k-1}_{loc}(\rn-\{0\})$ as $\lambda\to 0$. Passing to the limit $\lambda\to 0$ in \eqref{eq:tpl:bis} and \eqref{bnd:Vl}, we get that
$$\Delta^k_\xi\tilde{\psi}=0\hbox{ weakly in }\rn-\{0\}.$$
Since $\tilde{\psi}\in C^{2k-1}(\rn-\{0\})$, regularity yields $\tilde{\psi}\in C^{2k}(\rn-\{0\})$. Still up to extraction, we define $X_0:=\lim_{\lambda\to 0}X_\lambda$, so that $|X_0|=1$. Finally, passing to the limit in \eqref{bnd:tpl}, we get that
 \begin{equation}\label{eq:psi}
\left\{\begin{array}{l}
\tilde{\psi}\in C^{2k}(\rn-\{0\})\\
\Delta^k_\xi\tilde{\psi}=0 \hbox{  in }\rn-\{0\}\\
\tilde{\psi}(X_0)=1\hbox{ with }|X_0|=1\\
|\tilde{\psi}(X)|\leq |X|^{-\gamma}\hbox{ for all }X\in \rn-\{0\}
\end{array}\right\}
\end{equation} 
By standard elliptic theory (see Theorem \ref{th:2}), for any $l=1,...,2k$, there exists $C_l>0$ such that
\begin{equation}\label{bnd:nabla:psi}
|\nabla^l\tilde{\psi}(X)|\leq C_l |X|^{-\gamma-l}\hbox{ for all }X\in \rn-\{0\}
\end{equation}
\noindent{\bf Step 1.3: Contradiction via Green's formula.} Let us consider the Poisson kernel of $\Delta^k$ at $X_0$, namely
$$\Gamma_{X_0}(X):=C_{n,k}|X-X_0|^{2k-n}\hbox{ for all }X\in \rn-\{X_0\},$$
where
\begin{equation}\label{def:Cnk}
C_{n,k}:=\frac{1}{(n-2)\omega_{n-1}\Pi_{i=1}^{k-1}(n-2k+2(i-1))(2k-2i)}.
\end{equation}
Let us choose $R>3$ and $0<\eps<1/2$ and define the domain
$$\Omega_{R,\eps}:=B_R(0)\setminus\left(B_{R^{-1}}(0)\cup B_{\eps}(X_0)\right).$$
Note that all the balls involved here have boundaries that do not intersect. Integrating by parts, we have that
\begin{equation}\label{ipp:gamma}
\int_{\Omega_{R,\eps}}(\Delta^k_\xi\Gamma_{X_0})\tilde{\psi}\, dX=\int_{\Omega_{R,\eps}}\Gamma_{X_0}(\Delta^k_\xi\tilde{\psi})\, dX+\int_{\partial\Omega_{R,\eps}}\sum_{i=0}^{k-1}\tilde{B}_i(\Gamma_{X_0},\tilde{\psi})\, d\sigma
\end{equation}
where
$$\tilde{B}_i(\Gamma_{X_0},\tilde{\psi}):=-\partial_\nu\Delta_\xi^{k-i-1}\Gamma_{X_0}\Delta_\xi^i\tilde{\psi}+\Delta_\xi^{k-i-1}\Gamma_{X_0}\partial_\nu\Delta_\xi^i\tilde{\psi}$$
We have that
$$\partial\Omega_{R,\eps}=\partial B_R(0)\cup \partial B_{R^{-1}}(0)\cup B_{\eps}(X_0).$$
Using that $\Gamma_{X_0}$ is smooth at $0$, that $\tilde{\psi}$ is smooth at $X_0$, using the bounds \eqref{bnd:nabla:psi} and the corresponding ones for $\Gamma_{X_0}$, for any $i=0,...,k-1$, we get that
\begin{eqnarray*}
&&\left|\int_{\partial B_R(0)}\tilde{B}_i(\Gamma_{X_0},\tilde{\psi})\, d\sigma\right| \\
&&\leq  C R^{n-1}(R^{2k-n-2(k-i-1)-1}R^{-\gamma-2i}+R^{2k-n-2(k-i-1)}R^{-\gamma-2i-1})\leq CR^{-\gamma}
\end{eqnarray*}
\begin{eqnarray*}
\left|\int_{\partial B_{R^{-1}}(0)}\tilde{B}_i(\Gamma_{X_0},\tilde{\psi})\, d\sigma\right| &\leq &  C (R^{-1})^{n-1}( (R^{-1})^{-\gamma-2i}+ (R^{-1})^{-\gamma-2i-1})\\
&\leq& CR^{2-n+\gamma+2i}\leq C R^{-(n-2k-\gamma)}
\end{eqnarray*}
\begin{eqnarray*}
&&\left|\int_{\partial B_{\eps}(X_0)}\tilde{B}_i(\Gamma_{X_0},\tilde{\psi})\, d\sigma\right| \leq  C \eps^{n-1}(\eps^{2k-n-2(k-i-1)-1} +\eps^{2k-n-2(k-i-1)} )\leq C\eps^{2i}
\end{eqnarray*}
and
$$\left|\int_{\partial B_{\eps}(X_0)}\Delta_\xi^{k-1}\Gamma_{X_0}\partial_\nu\tilde{\psi}\, d\sigma\right|\leq C \eps^{n-1}\eps^{2k-n-2(k-1)}\leq C\eps.$$
Therefore, since $0<\gamma<n-2k$, all the terms involving $R$ go to $0$ as $R\to +\infty$, the terms involving $\eps$ go to $0$ when $i\neq 0$. Since $\Delta^k_\xi\Gamma_{X_0}=0$, $\Delta^k_\xi\tilde{\psi}=0$, it follows from \eqref{ipp:gamma} and the inequalities above that
$$\int_{\partial B_{\eps}(X_0)}\partial_\nu\Delta_\xi^{k-1}\Gamma_{X_0} \tilde{\psi}\, d\sigma=o(1)\hbox{ as }\eps\to 0.$$
With the definition of $\Gamma_{X_0}$, we get that 
\begin{equation}\label{calc:delta:gamma}
-\partial_\nu\Delta_\xi^{k-1}\Gamma_{X_0}(X)=\frac{1}{\omega_{n-1}}|X-X_0|^{1-n}\hbox{ for }X\neq X_0.
\end{equation}
So that, with a change of variable, we get that
$$\int_{\partial B_1(0)} \tilde{\psi}(X_0+\eps X)\, d\sigma=o(1)\hbox{ as }\eps\to 0.$$
Passing to the limit, we get that $\tilde{\psi}(X_0)=0$, which is a contradiction with \eqref{eq:psi}. This proves \eqref{ineq:lem} when $V$ vanishes around $x_0$ and $\delta_2=3\delta_1=3\delta$. As noted at the beginning of this case, the case $0<\delta_1<\delta_2$ is similar. 

\medskip\noindent{\bf Step 2: The general case. Proof of Lemma \ref{lem:main}.} We let $P\in O_{k,L}(x_0,\delta_2)$ and $V$ as in the statement of Lemma \ref{lem:main}. We consider $\varphi\in H_k^2(B_{\delta_2}(x_0))\cap H_{2k,loc}^s(B_{\delta_2}(x_0)-\{x_0\})$ (for some $s>1$) such that
\begin{equation}\label{eq:P:V}
P\varphi-V\cdot\varphi=0\hbox{ weakly in }H_k^2(B_{\delta_2}(x_0)).
\end{equation}
We approximate $V$ by a potential $V_\eps$ vanishing around $x_0$ via a cut-off function. Potentially taking $\delta_2>0$ and $\lambda>0$ smaller, for any $\eps>0$, we  $\varphi_\eps\in H_k^2(B_{\delta_2}(x_0))\cap H_{2k}^q(B_{\delta_2}(x_0))$ for all $q>1$ be such that
\begin{equation}\label{eq:phi:eps}
\left\{\begin{array}{ll}
(P-V_\eps)\varphi_\eps=0&\hbox{ in }B_{\delta_2}(x_0)\\
\partial_\nu^i\varphi_\eps=\partial_\nu^i\varphi&\hbox{ on }\partial B_{\delta_2}(x_0)\hbox{ for }i=0,\cdots,k-1\\
\end{array}\right.
\end{equation}
As one checks, $\lim_{\eps\to 0}\varphi_\eps=\varphi$ in $H_k^2(B_{\delta_2'}(x_0)\cap C^0(\bar{B}(x_0,\delta_1)\setminus B_r(x_0))$ for all $r>0$. We then apply \eqref{ineq:lem:bis} to $\varphi_\eps$ and we get \eqref{ineq:lem} by letting $\eps\to 0$. The estimate on the derivatives is a consequence of the control \eqref{ineq:lem} and elliptic theory. This ends the proof of Lemma \ref{lem:main}.

\section{Rescaling far from the concentration point}\label{sec:11}
This section is devoted to the proof of the three controls and convergence below:
\begin{proposition}\label{prop:cv:1} We take the assumptions and notations of Theorems \ref{prop:prelim:1} and \ref{prop:prelim:2}. Let $(r_\alpha)_{\alpha}>0$ be such that 
$$\lim_{\alpha\to 0}\ra=0\, ,\, \lim_{\alpha\to 0}\frac{\ra}{\ma}=+\infty\hbox{ and }\{\ra=O(\sqrt{\ma})\hbox{ if }u_0\not\equiv 0\}.$$
Then for any $X\in\rn-\{0\}$, we have that
\begin{equation*}
\lim_{\alpha\to 0}\frac{\ra^{n-2k}}{\ma^{\frac{n-2k}{2}}}\ua(\tilde{\hbox{exp}}_{\xa}(\ra X))=\lim_{\alpha\to 0}\left(\frac{\ra}{\sqrt{\ma}}\right)^{n-2k}u_0(x_0)+\frac{K_0 c_0^{\frac{n-2k}{2}}}{|X|^{n-2k}}
\end{equation*} 
where
\begin{equation}\label{def:K0}
K_0:=C_{n,k}\int_{\rn}|U|^{\crit-2}U\, dX,
\end{equation}
where $C_{n,k}$ is defined in \eqref{def:Cnk}. Moreover, the convergence above holds in $C^{2k}_{loc}(\rn-\{0\})$.
\end{proposition}
\begin{proposition}\label{prop:cv:2} We take the assumptions and notations of Theorems \ref{prop:prelim:1} and \ref{prop:prelim:2}. In addition, we assume that $u_0\equiv 0$. Then for any $x\in M-\{x_0\}$, we have that
\begin{equation*}
\lim_{\alpha\to 0}\frac{\ua(x)}{\ma^{\frac{n-2k}{2}}} =\frac{c_0^{\frac{n-2k}{2}}K_0}{C_{n,k}}G_0(x,x_0)
\end{equation*} 
where $K_0$ is defined in \eqref{def:K0} and $G_0$ is the Green's function of the operator $P_0$. Moreover, the convergence above holds in $C^{2k}_{loc}(M-\{x_0\})$.
\end{proposition}

\begin{proposition}\label{prop:bnd:der} Under the assumptions and notations of Theorems \ref{prop:prelim:1} and \ref{prop:prelim:2}, for $l\in \{0,...,2k\}$, we have that
\begin{equation}\label{ineq:49}
|\nabla^l\ua(x)|\leq C\frac{\ma^{\frac{n-2k}{2}}}{\left(\ma+d_g(x,\xa)\right)^{n-2k+l}}\hbox{ if }\left\{\begin{array}{c}
d_g(x,\xa)\leq C\sqrt{\ma}\\
\hbox{or }u_0\equiv 0\hbox{ and }x\in M\end{array}\right\}.
\end{equation}
\end{proposition}
\noindent{\it Proof of Proposition \ref{prop:cv:1}:} For $(\ua)_\alpha$ and $(\ra)_\alpha$ as in the statement of Proposition \ref{prop:cv:1}. We set
$$\wa(X):=\frac{\ra^{n-2k}}{\ma^{\frac{n-2k}{2}}}\ua(\tilde{\hbox{exp}}_{\xa}(\ra X))$$
for all $X\in B_{i_g(M)\ra^{-1}}(0)\subset \rn$. As in Definition \ref{def:green}, let $G_\alpha$ be the Green's function for $P_\alpha$, $\alpha>0$, and $G_0$ be the Green's function for the limiting operator $P_0$. Let us fix $X\in \rn-\{0\}$. We let $R>0$ and $\delta\in (0,i_g(M))$ to be chosen later. Green's representation formula yields
\begin{eqnarray*}
\wa(X)&=& \frac{\ra^{n-2k}}{\ma^{\frac{n-2k}{2}}}\int_M G_\alpha(\tilde{\hbox{exp}}_{\xa}(\ra X),y)|\ua|^{\crit-2-\ea}\ua(y)\, dv_g(y)\\
&=& \frac{\ra^{n-2k}}{\ma^{\frac{n-2k}{2}}}\left(\int_{M\setminus B_\delta(\xa)}+\int_{B_\delta(\xa)\setminus B_{R\ma}(\xa)}+\int_{B_{R\ma}(\xa)}\right)
\end{eqnarray*}
Using Step \ref{prop:cv:u0}, $\lim_{\alpha\to 0}\ra=0$ and $G_\alpha\to G_0$ uniformly outside the diagonal, we get that 
\begin{eqnarray*}
&&\lim_{\alpha\to 0}\int_{M\setminus B_\delta(\xa)}G_\alpha(\tilde{\hbox{exp}}_{\xa}(\ra X),y)|\ua|^{\crit-2-\ea}\ua(y)\, dv_g(y)\\
&&=\int_{M\setminus B_\delta(x_0)}G_0(x_0,y)|u_0|^{\crit-2}u_0(y)\, dv_g(y).
\end{eqnarray*}
Since $u_0$ satisfies the limit equation $P_0u_0=|u_0|^{\crit-2}u_0$, we get that
\begin{eqnarray*}
&&\lim_{\delta\to 0}\lim_{\alpha\to 0}\int_{M\setminus B_\delta(\xa)}G_\alpha(\tilde{\hbox{exp}}_{\xa}(\ra X),y)|\ua|^{\crit-2-\ea}\ua(y)\, dv_g(y)\\
&&=\int_{M}G_0(x_0,y)|u_0|^{\crit-2}u_0(y)\, dv_g(y)=u_0(x_0).
\end{eqnarray*}
We now deal with the second integral. Using the pointwise controls \eqref{ineq:c0} and \eqref{ctrl:G}, we get that
\begin{eqnarray*}
&&\left|\frac{\ra^{n-2k}}{\ma^{\frac{n-2k}{2}}}\int_{B_\delta(\xa)\setminus B_{R\ma}(\xa)}G_\alpha(\tilde{\hbox{exp}}_{\xa}(\ra X),y)|\ua|^{\crit-2-\ea}\ua(y)\, dv_g(y)\right|\\
&&\leq C \frac{\ra^{n-2k}}{\ma^{\frac{n-2k}{2}}}\int_{B_\delta(\xa)\setminus B_{R\ma}(\xa)} d_g(y, \tilde{\hbox{exp}}_{\xa}(\ra X))^{2k-n}\left({\bf 1}_{u_0\neq 0}+\frac{\ma^{\frac{n-2k}{2}}}{d_g(y,\xa)^{n-2k}}\right)^{\crit-1-\ea}\, dv_g(y)\\
&&\leq C \frac{\ra^{n-2k}}{\ma^{\frac{n-2k}{2}}}\delta^{2k}{\bf 1}_{u_0\neq 0}+C \ra^{n-2k}\ma^{2k}\int_{B_\delta(0)\setminus B_{R\ma}(0)} |Y-\ra X|^{2k-n}|Y|^{-(n-2k)(\crit-1-\ea)}\, dY
\end{eqnarray*}
Keeping in mind that $X\neq 0$ is fixed, we have that
\begin{eqnarray*}
&&\int_{B_\delta(0)\setminus B_{R\ma}(0)} |Y-\ra X|^{2k-n}|Y|^{-(n-2k)(\crit-1-\ea)}\, dY\leq \int_{|Y-\ra X|>\ra|X|/2}+\int_{|Y-\ra X|\leq\ra|X|/2}\\
&&\leq C\ra^{2k-n}\int_{\rn\setminus B_{R\ma}(0)}|Y|^{-(n-2k)(\crit-1-\ea)}\, dY\\
&&+C\ra^{-(n-2k)(\crit-1-\ea)}\int_{B_{\ra|X|/2}(\ra X)} |Y-\ra X|^{2k-n}\, dY\\
&&\leq C\ra^{2k-n} (R\ma)^{-2k}+C\ra^{-n}
\end{eqnarray*}
Therefore,
\begin{eqnarray*}
&&\lim_{R\to +\infty,\delta\to 0}\lim_{\alpha\to 0}\left|\frac{\ra^{n-2k}}{\ma^{\frac{n-2k}{2}}}\int_{B_\delta(\xa)\setminus B_{R\ma}(\xa)}G_\alpha(\tilde{\hbox{exp}}_{\xa}(\ra X),y)|\ua|^{\crit-2-\ea}\ua(y)\, dv_g(y)\right|\\
&&\leq \lim_{R\to +\infty, \delta\to 0} \left(\left(\frac{\ra}{\sqrt{\ma}}\right)^{\frac{n-2k}{2}}\delta^{2k}{\bf 1}_{u_0\neq 0} + CR^{-2k}+ \left(\frac{\ma}{\ra}\right)^{2k}\right)=0
\end{eqnarray*}
Finally, using \eqref{def:tua}, the third integral reads
\begin{eqnarray*}
&&\frac{\ra^{n-2k}}{\ma^{\frac{n-2k}{2}}}\int_{B_{R\ma}(\xa)} G_\alpha(\tilde{\hbox{exp}}_{\xa}(\ra X),y)|\ua|^{\crit-2-\ea}\ua(y)\, dv_g(y)\\
&&=\frac{\ra^{n-2k}}{\ma^{\frac{n-2k}{2}}}\ma^n\int_{B_{R}(0)} G_\alpha(\tilde{\hbox{exp}}_{\xa}(\ra X),\tilde{\hbox{exp}}_{\xa}(\ma Y))|\tua|^{\crit-2-\ea}\tua(Y)\ma^{-(\crit-1-\ea)(n-2k)/2}\, dv_{\tga}(Y)\\
&&= (\ma^{\ea})^{\frac{n-2k}{2}}\int_{B_{R}(0)} \ra^{n-2k}G_\alpha(\tilde{\hbox{exp}}_{\xa}(\ra X),\tilde{\hbox{exp}}_{\xa}(\ma Y))|\tua|^{\crit-2-\ea}\tua(Y)\, dv_{\tga}(Y),
\end{eqnarray*}
where $\tga:=(\tilde{\hbox{exp}}_{\xa}^\star g)(\ma\cdot)$. It follows from the construction of the Green's function (see the proof of Theorem \ref{th:green:nonsing}) that 
\begin{equation*}
|G_\alpha(x,y)-C_{n,k}d_g(x,y)^{2k-n}|\leq C d_g(x,y)^{2k+1-n}\hbox{ for all }x,y\in M, \, x\neq y,
\end{equation*}
where $C_{n,k}$ is defined in \eqref{def:Cnk}. Therefore, using it follows from Theorem \ref{prop:prelim:1} that
\begin{eqnarray*}
&&\lim_{R\to +\infty}\lim_{\alpha\to 0}\frac{\ra^{n-2k}}{\ma^{\frac{n-2k}{2}}}\int_{B_{R\ma}(\xa)} G_\alpha(\tilde{\hbox{exp}}_{\xa}(\ra X),y)|\ua|^{\crit-2-\ea}\ua(y)\, dv_g(y)\\
&&= c_0^{\frac{n-2k}{2}}\int_{\rn} C_{n,k}|X|^{2k-n}|U|^{\crit-2}U(Y)\, dY=\frac{c_0^{\frac{n-2k}{2}}K_0}{|X|^{n-2k}}.
\end{eqnarray*}
Therefore,
\begin{equation}\label{eq:lim:pointwise}
\lim_{\alpha\to 0}\wa(X)=\lim_{\alpha\to 0}\left(\frac{\ra}{\sqrt{\ma}}\right)^{n-2k}u_0(x_0)+\frac{K_0 c_0^{\frac{n-2k}{2}}}{|X|^{n-2k}}
\end{equation}
where $K_0$ is as in \eqref{def:K0}. This proves the pointwise convergence. We are left with proving the $C^{2k}-$convergence. Equation \eqref{eq:ua:prop} rewrites
\begin{equation}\label{eq:wa:prop}
\Delta_{h_\alpha}^k\wa+\sum_{j=0}^{2k-2}\ra^{2k-j}\tilde{B}_\alpha^j\star\nabla^j\wa=\left(\frac{\ma}{\ra}\right)^{2k}\frac{(\ra^{\ea})^{n-2k}}{(\ma^{\ea})^{\frac{n-2k}{2}}}|\wa|^{\crit-2-\ea}\wa\hbox{ in }B_{\delta\ra^{-1}}(0)
\end{equation}
where $h_\alpha:=(\tilde{\hbox{exp}}_{\xa}g)(\ra\cdot)$ and for all $j=0,...,2k-2$, $(\tilde{B}_\alpha^j)_\alpha$ is a family of $(j,0)-$tensors such that there exists $C_j>0$ such that $\Vert \tilde{B}_\alpha^j\Vert_{C^{0,\theta}}\leq C_j$ for all $\alpha>0$. Since $\ma^{\ea}\to c_0\in (0,1]$, it follows from the hypotheses of Proposition \ref{prop:cv:1} that $\ra^{\ea}\to c_1\in (0,1]$. The pointwise control \eqref{ineq:c0} rewrites
$$|\wa(X)|\leq C+C|X|^{2k-n}\hbox{ for all }\alpha>0\hbox{ and }X\in B_R(0)-\{0\}.$$
It then follows from elliptic theory (Theorem \ref{th:2bis}) that, up to extraction, $(\wa)_\alpha$ has a limit in $C^{2k}_{loc}(\rn-\{0\})$. With \eqref{eq:lim:pointwise}, we get that the limit holds without extraction. This proves Proposition \ref{prop:cv:1}.\qed

\medskip\noindent{\it Proof of Proposition \ref{prop:cv:2}.} We assume here that $u_0\equiv 0$. The argument is very similar to the proof of Proposition \ref{prop:cv:1}. Via Green's representation formula, for any $x\in M-\{x_0\}$, we get that 
\begin{equation*}
\lim_{\alpha\to 0}\frac{\ua(x)}{\ma^{\frac{n-2k}{2}}} =c_0^{\frac{n-2k}{2}}\left(\int_{\rn}|U|^{\crit-2}U\, dX\right)G_0(x,x_0)=\frac{K_0 c_0^{\frac{n-2k}{2}}}{C_{n,k}}G_0(x,x_0).
\end{equation*}
The pointwise control \eqref{ineq:c0} yields 
$$\left|\ma^{-\frac{n-2k}{2}}\ua(x)\right|\leq Cd_g(x,\xa)^{2k-n}\hbox{ for all }x\in M-\{\xa\}\hbox{ and }\alpha>0.$$
Equation \eqref{eq:ua:prop} rewrites
$$P_\alpha\left(\ma^{-\frac{n-2k}{2}}\ua\right)=\ma^{\frac{n-2k}{2}(\crit-2-\ea)}\left|\ma^{-\frac{n-2k}{2}}\ua\right|^{\crit-2-\ea}\left(\ma^{-\frac{n-2k}{2}}\ua\right).$$
Elliptic theory then yields convergence of $(\ma^{-\frac{n-2k}{2}}\ua)_\alpha$ in $C^{2k}_{loc}(M-\{x_0\})$ up to as subsequence, and then, via uniqueness, for the full family. This proves Proposition \ref{prop:cv:2}.\qed

\medskip\noindent{\it Proof of Proposition \ref{prop:bnd:der}.} We argue by contradiction. We assume that there exists $l\in \{0,...,2k\}$ and $(z_\alpha)_\alpha\in M$ such that
\begin{equation}\label{hyp:contra:1}
\lim_{\alpha\to 0}\ma^{-\frac{n-2k}{2}}\left(\ma+d_g(z_\alpha,\xa)\right)^{n-2k+l}|\nabla^l\ua(z_\alpha)|=+\infty,
\end{equation} 
with $d_g(z_\alpha,\xa)=O(\sqrt{\ma})$ if $u_0\not\equiv 0$. It follows from the convergence \eqref{lim:tua:c2k} that $\ma=o(d_g(z_\alpha, \xa))$.

\smallskip\noindent Assume first that $d_g(z_\alpha, \xa)\to 0$ as $\alpha\to 0$. We set $\ra:=d_g(z_\alpha, \xa)$ and $Z_\alpha\in\rn$ be such that $z_\alpha=\tilde{\hbox{exp}}_{\xa}(\ra Z_\alpha)$. In particular $|Z_\alpha|=1$ for all $\alpha>0$. We use the notations of Proposition \ref{prop:cv:1}, so that $(\wa)_\alpha$ has a limit in $C^{2k}_{loc}(\rn-\{0\})$. In particular
$$\frac{d_g(z_\alpha,\xa)^{n-2k+l}|\nabla^l\ua(z_\alpha)|}{\ma^{\frac{n-2k}{2}}}=|\nabla^l\wa(Z_\alpha)|$$
is bounded as $\alpha\to 0$, contradicting  \eqref{hyp:contra:1} since $\ma=o(\ra)$.

\smallskip\noindent Assume now that $d_g(z_\alpha, \xa)\not\to 0$ as $\alpha\to 0$. In particular $u_0\equiv 0$. It then follows from Proposition \ref{prop:cv:2} that $(\ma^{-\frac{n-2k}{2}}|\nabla^l\ua(z_\alpha)|)_\alpha$ is bounded, as $\alpha\to 0$, contradicting  \eqref{hyp:contra:1}.

\smallskip These two cases prove Proposition \ref{prop:bnd:der}.\qed

\section{Pohozaev-Pucci-Serrin identity for polyharmonic operators}\label{sec:poho}
The following identities are essentially in Pucci-Serrin \cite{PucciSerrinIdentity} and are generalizations of the historical Pohozaev identity \cite{pohozaev}. We recall and prove them for the sake of completeness.
\begin{lemma}\label{lem:deltap} For any $v\in C^\infty(\Omega)$, where $\Omega$ is a domain of $\rn$, we have that
\begin{equation*}
\left\{\begin{array}{cc}
\Delta_\xi^p(x^i\partial_i v)=2p\Delta_\xi^p v+x^i\partial_i\Delta_\xi^p v&\hbox{ for all }p\in\nn\hbox{ and }\\
\partial_j\Delta_\xi^p(x^i\partial_i v)=(2p+1)\partial_j\Delta_\xi^p v+x^i\partial_i(\partial_j\Delta_\xi^p v)&\hbox{ for all }p\in\nn\hbox{ and }j=1,...,n.
\end{array}\right\}.
\end{equation*}
With the convention $\Delta_\xi^\frac{l}{2}=\nabla \Delta_\xi^\frac{l-1}{2}$ when $l$ is odd, these identities rewrite
\begin{equation*}
\Delta_\xi^\frac{l}{2}(x^i\partial_i v)=l\Delta_\xi^\frac{l}{2} v+x^i\partial_i(\Delta_\xi^\frac{l}{2} v)\hbox{ for all }l\in\nn.
\end{equation*}
\end{lemma}
This is straightforward iteration.
\begin{proposition}\label{prop:poho} Let $\Omega\subset\rn$ be a smooth bounded domain with $2\leq 2k<n$. Then for all $u\in C^{2k+1}(\rn)$, $f\in C^1_c(\rn)$ and $\eps\in [0,\crit-2)$, we have that
\begin{eqnarray*}
&&\int_\Omega\left(\Delta_\xi^k u-f|u|^{\crit-2-\eps}u\right) T(u)\, dx=\frac{n\eps}{\crit(\crit-\eps)}\int_\Omega f|u|^{\crit-\eps}\, dx+\frac{1}{\crit-\eps}\int_\Omega ( \nabla f,x)|u|^{\crit-\eps}\, dx\\
&&+\int_{\partial\Omega}\left((x,\nu)\left(\frac{(\Delta_\xi^\frac{k}{2}u)^2}{2}-\frac{f|u|^{\crit-\eps}}{\crit-\eps}\right)+S(u)\right)\, d\sigma
\end{eqnarray*}
where
$$T(u):=\frac{n-2k}{2}u+x^i\partial_i u$$
and
\begin{eqnarray*}
S(u)&:=&\sum_{i=0}^{E(k/2)-1}\left(-\partial_\nu\Delta_\xi^{k-i-1}u\Delta_\xi^iT(u)+\Delta_\xi^{k-i-1}u\partial_\nu\Delta_\xi^iT(u)\right)\\
&&-{\bf 1}_{\{k\hbox{ odd}\}}\partial_\nu(\Delta_\xi^\frac{k-1}{2} u)\Delta_\xi^\frac{k-1}{2}T(u)
\end{eqnarray*}
\end{proposition}
\begin{proof} Integrating by parts, for any $l\in\nn$, $l\geq 1$, $U,V\in C^{2l}(\rn)$, we have that
\begin{equation*}
\int_{\Omega }(\Delta^l_\xi U)V \, dX=\int_{\Omega }U(\Delta^l_\xi V)\, dX+\int_{\partial\Omega}\mathcal{B}^{(l)}(U,V)\, d\sigma
\end{equation*}
where
$$\mathcal{B}^{(l)}(U,V):=\sum_{i=0}^{l-1}\left(-\partial_\nu\Delta_\xi^{l-i-1}U\Delta_\xi^i V+\Delta_\xi^{l-i-1}U\partial_\nu\Delta_\xi^iV\right)$$

We first assume that $k=2p$ is even, with $p\in\nn$. Using Lemma \ref{lem:deltap}, we get 
\begin{eqnarray*}
&&\int_\Omega\left(\Delta_\xi^k u-f|u|^{\crit-2-\eps}u\right) T(u)\, dx=\int_\Omega\Delta_\xi^pu\Delta^p_\xi T(u)\, dx+\int_{\partial\Omega}\mathcal{B}^{(p)}(\Delta_\xi^p u,T(u))\, d\sigma\\
&&-\left(\frac{n-2k}{2}\int_\Omega f |u|^{\crit-\eps}\, dx+\int_\Omega fx^i\frac{\partial_i |u|^{\crit-\eps}}{\crit-\eps}\, dx\right)\\
&&=\int_\Omega\Delta_\xi^p\left(\frac{n}{2}\Delta^p_\xi u+x^i\partial_i(\Delta_\xi^p u)\right)\, dx+\int_{\partial\Omega}\mathcal{B}^{(p)}(\Delta_\xi^p u,T(u))\, d\sigma\\
&&-\left(\frac{n-2k}{2}\int_\Omega f |u|^{\crit-\eps}\, dx-\frac{1}{\crit-\eps}\int_\Omega (n f+(\nabla f,x))|u|^{\crit-\eps}\, dx+\int_{\partial\Omega}(x,\nu)\frac{f|u|^{\crit-\eps}}{\crit-\eps}\, dx\right)\\
&&=\int_\Omega\partial_i\left(\frac{x^i(\Delta_\xi^p u)^2}{2}\right)\, dx+\int_{\partial\Omega}\mathcal{B}^{(p)}(\Delta_\xi^p u,T(u))\, d\sigma\\
&&+\frac{n\eps}{\crit(\crit-\eps)}\int_\Omega f|u|^{\crit-\eps}\, dx+\frac{1}{\crit-\eps}\int_\Omega ( \nabla f,x)|u|^{\crit-\eps}\, dx-  \int_{\partial\Omega}(x,\nu)\frac{f|u|^{\crit-\eps}}{\crit-\eps}\, dx\\
&&=\frac{n\eps}{\crit(\crit-\eps)}\int_\Omega f|u|^{\crit-\eps}\, dx+\frac{1}{\crit-\eps}\int_\Omega ( \nabla f,x)|u|^{\crit-\eps}\, dx\\
&&+\int_{\partial\Omega}\left((x,\nu)\left(\frac{(\Delta_\xi^p u)^2}{2}-\frac{f|u|^{\crit-\eps}}{\crit-\eps}\right)+\mathcal{B}^{(p)}(\Delta_\xi^p u,T(u))\right)\, d\sigma
\end{eqnarray*}
which proves Proposition \ref{prop:poho} when $k=2p$ is even. When $k=2q+1$ is odd, we get that
\begin{eqnarray*}
\int_\Omega \Delta_\xi^k u T(u)\, dx&=& \int_\Omega \Delta_\xi^q(\Delta_\xi^{q+1}u)T(u)\, dx\\
&=& \int_\Omega\Delta_\xi^{q+1}u\Delta_\xi^q T(u)\, dx+\int_{\partial\Omega}\mathcal{B}^{(q)}(\Delta_\xi^{q+1}u,T(u))\, d\sigma\\
&=& \int_\Omega \sum_j\partial_j(\Delta_\xi^q u)\partial_j(\Delta_\xi^q T(u))\, dx\\
&&+\int_{\partial\Omega}\left(\mathcal{B}^{(q)}(\Delta_\xi^{q+1}u,T(u))-\partial_\nu(\Delta_\xi^q u)\Delta_\xi^qT(u)\right)\, d\sigma
\end{eqnarray*}
Using Lemma \ref{lem:deltap}, we get that
\begin{eqnarray*}
&&\int_\Omega \Delta_\xi^k u T(u)\, dx= \int_\Omega \sum_j\partial_j(\Delta_\xi^q u)\left(\frac{n}{2}\partial_j \Delta_\xi^q u+x^i\partial_i\partial_j \Delta_\xi^q u\right)\\
&&+\int_{\partial\Omega}\left(\mathcal{B}^{(q)}(\Delta_\xi^{q+1}u,T(u))-\partial_\nu(\Delta_\xi^q u)\Delta_\xi^qT(u)\right)\, d\sigma\\
&&= \int_\Omega \partial_i \left(x^i\frac{(\partial_j\Delta_\xi^q u)^2}{2}\right)\, dx +\int_{\partial\Omega}\left(\mathcal{B}^{(q)}(\Delta_\xi^{q+1}u,T(u))-\partial_\nu(\Delta_\xi^q u)\Delta_\xi^qT(u)\right)\, d\sigma\\
&&= \int_{\partial\Omega}(x,\nu)\frac{|\nabla\Delta_\xi^q u|^2}{2} \, d\sigma +\int_{\partial\Omega}\left(\mathcal{B}^{(q)}(\Delta_\xi^{q+1}u,T(u))-\partial_\nu(\Delta_\xi^q u)\Delta_\xi^qT(u)\right)\, d\sigma\end{eqnarray*}
Using the same computations as in the case when $k$ is even, we get the conclusion of Proposition \ref{prop:poho}.\end{proof}

\section{Consequences of the Pohozaev-Pucci-Serrin identity: proof of Theorems \ref{th:general} and \ref{th:no:unull}}\label{sec:appli}
It follows from Lee-Parker \cite{lp} that for all $p\in B_{\delta}(x_0)$, there exists $\varphi_p\in C^\infty(M)$, $\varphi_p>0$, such that, setting $g_p:=\varphi_p^{\frac{4}{n-2k}}g$, we have that 
\begin{equation}\label{def:phi}
\varphi_p(p)=1\, ,\, \nabla\varphi_p(p)=0\, ,\, \hbox{Ric}_{g_p}(p)=0.
\end{equation}
Moreover, the map $p\mapsto \varphi_p$ is continuous. We set $\phia:=\varphi_{\xa}\circ \tilde{\hbox{exp}}_{\xa}$ and $g_\alpha:=\tilde{\hbox{exp}}_{\xa}^\star g_{\xa}$.  We rewrite equation \eqref{eq:ua:prop} with the GJMS operator \eqref{def:gjms} as
\begin{equation*}
P_g^k\ua+\sum_{i=0}^{k-1} (-1)^i\nabla^i((A_\alpha^{(i)}-A_{g}^{(i)}) \nabla^i\ua)=|\ua|^{\crit-2-\ea}\ua\hbox{ in }M.
\end{equation*}
We define 
\begin{equation}\label{def:hua}
\hua:=(\varphi_{x_\alpha}^{-1}\ua)\circ \tilde{\hbox{exp}}_{\xa}=\phia\cdot \ua\circ \tilde{\hbox{exp}}_{\xa} : B_\delta(0)\subset\rn\to \rr.
\end{equation}
Using the conformal properties \eqref{invar:gjms} of the GJMS operator, we get that $P_{g_{\xa}}(\varphi_{\xa}^{-1}\ua)=\varphi_{\xa}^{-(\crit-1)}P_g^k\ua$. Plugging this expression in the above equation satisfied by $\ua$ yields
\begin{equation}\label{eq:hua}
\Delta_{\ga}^k\hua+\sum_{i=0}^{k-1} (-1)^i\nabla^i_{\ga}(\hat{A}_\alpha^{(i)}  \nabla^i_{\ga}\hua)=\phia^{-\ea}|\hua|^{\crit-2-\ea}\hua\hbox{ in }B_\delta(0).
\end{equation}
where $\hat{A}^{(i)}_\alpha$ is a smooth symmetric $(0, 2i)-$tensor for all $i=0,...,k-1$ and there exists $\Vert \hat{A}^{(i)}_\alpha\Vert_{C^{i,\theta}}\leq C$ for all $\alpha$. Moreover,
\begin{equation}\label{eq:76}
\hat{A}_\alpha^{(k-1)}:=\phia^{2-\crit}\left(((\tilde{\hbox{exp}}_{\xa}^{g_{\xa}})^\star (A_\alpha^{(k-1)}-A_g^{(k-1)})\right)+(\tilde{\hbox{exp}}_{\xa}^{g_{\xa}})^\star A^{(k-1)}_{\ga}.
\end{equation}

\medskip\noindent We apply the identity of Proposition \ref{prop:poho} to $\hua$ on $B_{\da}(0)$ with
\begin{equation*}
\da:=\left\{\begin{array}{cc}
\sqrt{\ma}&\hbox{if }u_0\not\equiv 0\\
\delta &\hbox{if }u_0\equiv 0.
\end{array}\right.
\end{equation*}
Using \eqref{eq:hua}, we get
\begin{equation}\label{pohi:IaIV}
III_\alpha-II_\alpha=I_\alpha+IV_\alpha
\end{equation}
where
\begin{equation}\label{def:Ialpha}
I_\alpha:=\frac{n\ea}{\crit(\crit-\ea)}\int_{B_{\da}(0)} \phia^{-\ea}|\hua|^{\crit-\ea}\, dx+\frac{1}{\crit-\ea}\int_{B_{\da}(0)} ( \nabla \phia^{-\ea},x)|\hua|^{\crit-\ea}\, dx
\end{equation}
\begin{equation}\label{def:IIalpha}
II_\alpha:=\int_{B_{\da}(0)}\left(\sum_{i=0}^{k-1} (-1)^i\nabla^i_{\ga}(\hat{A}_\alpha^{(i)}  \nabla^i_{\ga}\hua) \right) T(\hua)\, dx
\end{equation}
\begin{equation}\label{def:IIIalpha}
III_\alpha:=\int_{B_{\da}(0)}\left(\Delta_\xi^k\hua  -\Delta_{\ga}^k  \hua\right) T(\hua)\, dx
\end{equation}
\begin{equation}\label{def:IValpha}
IV_\alpha:=\int_{\partial B_{\da}(0)}\left((x,\nu)\left(\frac{(\Delta_\xi^\frac{k}{2}\hua)^2}{2}-\frac{\phia^{-\ea}|\hua|^{\crit-\ea}}{\crit-\ea}\right)+S(\hua)\right)\, d\sigma
\end{equation}
where
$$T(\hua):=\frac{n-2k}{2}\hua+x^i\partial_i \hua$$
and
\begin{eqnarray*}
S(\hua)&:=&\sum_{i=0}^{E(k/2)-1}\left(-\partial_\nu\Delta_\xi^{k-i-1}\hua\Delta_\xi^iT(\hua)+\Delta_\xi^{k-i-1}\hua\partial_\nu\Delta_\xi^iT(\hua)\right)\\
&&-{\bf 1}_{\{k\hbox{ odd}\}}\partial_\nu(\Delta_\xi^\frac{k-1}{2} \hua)\Delta_\xi^\frac{k-1}{2}T(\hua)
\end{eqnarray*}
We estimate these terms separately. First note that for $l=0,...,2k$, the pointwise control \eqref{ineq:49} writes
\begin{equation}\label{ineq:59}
|\nabla^l\hua(x)|\leq C\frac{\ma^{\frac{n-2k}{2}}}{\left(\ma+|x|\right)^{n-2k+l}}\hbox{ for all }x\in B_{\da}(0).
\end{equation}
Using \eqref{ineq:59}, we get that
\begin{equation}\label{ineq:60}
\int_{B_{\da}(0)}|x|^t|\nabla^l \hua|\left(|x|\cdot|\nabla \hua|+|\hua|\right)\, dx\leq C\left\{\begin{array}{cc}
\ma^{t+2k-l} &\hbox{ if }n-4k+l>t\\
\ma^{n-2k}\ln\frac{1}{\ma} &\hbox{ if }n-4k+l=t\\
\ma^{\frac{n+t-l}{2}} &\hbox{ if }n-4k+l<t\hbox{ and }u_0\not\equiv 0\\
\ma^{n-2k}  &\hbox{ if }n-4k+l<t\hbox{ and }u_0\equiv 0
\end{array}\right.
\end{equation}
We define $\va(x):=\ma^{\frac{n-2k}{2}}\hua( \ma x)=\ma^{\frac{n-2k}{2}}\phia^{-1}(\ma X)\ua\circ \tilde{exp}_{\xa}( \ma x))$ for all $x\in B_{\delta/\ma}(0)$. Since $\phia(0)=1$, \eqref{lim:tua:c2k} writes
\begin{equation}\label{lim:tua:c2k:bis}
\lim_{\alpha\to +\infty}\va=U\hbox{ in }C^{2k}_{loc}(\rn)\hbox{ and }\lim_{\alpha\to 0}\ma^{\ea}=c_0>0.
\end{equation}
\subsection{Estimate of $I_\alpha$.} Given $R>0$, we get on the one hand that
$$\left|\int_{B_{\da}(0)\setminus B_{R\ma}(0)} \phia^{-\ea}|\hua|^{\crit-\ea}\, dx\right|\leq C\int_{B_{\da}(0)\setminus B_{R\ma}(0)} \left(\frac{\ma^{\frac{n-2k}{2}}}{\left(\ma+|x|\right)^{n-2k}}\right)^{\crit-\ea}\, dx\leq \eta(R)$$
where $\lim_{R\to +\infty}\eta(R)=0$. On the other hand, with the change of variables $x=\ma X$ and \eqref{lim:tua:c2k:bis}, we get
\begin{eqnarray*}
\int_{  B_{R\ma}(0)} \phia^{-\ea}|\hua|^{\crit-\ea}\, dx&=&\left(\ma^{\ea}\right)^{\frac{n-2k}{2}}\int_{B_R(0)}\phia(\ma X)^{-\ea}|\va|^{\crit-\ea}\, dX\\
&\to &c_0^{\frac{n-2k}{2}}\int_{B_R(0)} |U|^{\crit } \, dX\hbox{ as }\alpha\to 0
\end{eqnarray*}
Combining these two estimates and using that $U\in L^{\crit}(\rn)$ yields
\begin{equation}\label{eq:81}
\lim_{\alpha\to 0}\int_{B_{\da}(0) } \phia^{-\ea}|\hua|^{\crit-\ea}\, dx=c_0^{\frac{n-2k}{2}}\int_{\rn} |U|^{\crit }\, dX
\end{equation}
The rough control \eqref{ineq:59} and $\nabla\phia(0)=0$ yield
\begin{eqnarray*}
&&\left|\int_{B_{\da}(0)} ( \nabla \phia^{-\ea},x)|\hua|^{\crit-\ea}\, dx\right|=\left|\int_{B_{\da}(0)} -\ea( \nabla \phia,x)\phia^{-\ea-1}|\hua|^{\crit-\ea}\, dx\right|\\
&&\leq C\ea \left|\int_{B_{\da}(0)} |x|^2\left(\frac{\ma^{\frac{n-2k}{2}}}{\left(\ma+|x|\right)^{n-2k}}\right)^{\crit-\ea}\, dx\right|=o(\ea)
\end{eqnarray*}
We then get that
\begin{equation}\label{estim:Ialpha}
I_\alpha=\left(\frac{n}{(\crit)^2}c_0^{\frac{n-2k}{2}}\int_{\rn} |U|^{\crit }\, dX\right) \ea+o(\ea)
\end{equation}
\subsection{Estimate of $II_\alpha$.} For convenience, we define
$$\theta_\alpha:=\left\{\begin{array}{cc}
\ma^2 &\hbox{ if }n>2k+2,\\
\ma^2\ln\frac{1}{\ma} &\hbox{ if }n=2k+2,\\
\end{array}\right.$$
We compute $II_\alpha$ in the general context. In section \ref{sec:extra}, we compute the same term in a specific situation. Using \eqref{ineq:60}, we get that
\begin{equation*}
\int_{B_{\da}(0)}\nabla^i_{\ga}(\hat{A}_\alpha^{(i)}  \nabla^i_{\ga}\hua) T(\hua)\, dx=o(\theta_\alpha) \hbox{ for }i<k-1\hbox{ and } n\geq 2k+2.
\end{equation*}
We are left with the term of order $k-1$. Since the Levi-Civita connexion satisfies $\nabla^{k-1}_{\ga}=\nabla^{k-1}_\xi+\hbox{lot}$ where lot denotes differential terms of order less that $k-2$, using \eqref{ineq:60}, we get that
\begin{eqnarray*}
&&(-1)^{k-1}\int_{B_{\da}(0)}\nabla^{k-1}_{\ga}(\hat{A}_\alpha^{(k-1)}  \nabla^{k-1}_{\ga}\hua) T(\hua)\, dx\\
&&=(-1)^{k-1}\int_{B_{\da}(0)}\nabla^{k-1}_{\xi}(\hat{A}_\alpha^{(k-1)}  \nabla^{k-1}_{\xi}\hua) T(\hua)\, dx+o(\theta_\alpha)\hbox{ for }n\geq 2k+2.
\end{eqnarray*}
Integrating $k-1$ times by parts and using \eqref{ineq:59} yields
\begin{eqnarray*}
&& (-1)^{k-1}\int_{B_{\da}(0)}\nabla^{k-1}_{\xi}(\hat{A}_\alpha^{(k-1)}  \nabla^{k-1}_{\xi}\hua) T(\hua)\, dx\\
&&=\int_{B_{\da}(0)}\hat{A}_\alpha^{(k-1)}  (\nabla^{k-1}_{\xi}\hua,\nabla^{k-1}_{\xi} T(\hua))\, dx+\sum_{p+q<2k-2}\int_{\partial B_{\da}(0)}\nabla^p\hua\star\nabla^q\hua\, d\sigma\\
&&= \int_{B_{\da}(0)}\hat{A}_\alpha^{(k-1)}(0)  (\nabla^{k-1}_{\xi}\hua,\nabla^{k-1}_{\xi} T(\hua))\, dx+o(\theta_\alpha)\hbox{ for }n\geq 2k+2.
\end{eqnarray*}
With a slight abuse of notation, we have that
\begin{equation*}
\nabla^{k-1}_{\xi} T(\hua)=\frac{n-2k}{2}\nabla^{k-1}_{\xi} \hua+\nabla^{k-1}_{\xi} (x^p\partial_p\hua)=\frac{n-2}{2}\nabla^{k-1}_{\xi} \hua+x^p\partial_p\nabla^{k-1}_{\xi} \hua.
\end{equation*}
Using the symmetry of $(X,Y)\mapsto \hat{A}_\alpha^{(k-1)} (0) (X,Y)$, integrating by parts and using the pointwise controls \eqref{ineq:59} and \eqref{ineq:60}, we get that
\begin{eqnarray*}
&& \int_{B_{\da}(0)}\hat{A}_\alpha^{(k-1)} (0) (\nabla^{k-1}_{\xi}\hua,x^p\partial_p\nabla^{k-1}_{\xi} \hua)\, dx\\
&&= \int_{B_{\da}(0)}x^p\partial_p\frac{\hat{A}_\alpha^{(k-1)} (0) (\nabla^{k-1}_{\xi}\hua,\nabla^{k-1}_{\xi} \hua)}{2}\, dx\\
&&=-\frac{n}{2}\int_{B_{\da}(0)} \hat{A}_\alpha^{(k-1)} (0) (\nabla^{k-1}_{\xi}\hua,\nabla^{k-1}_{\xi} \hua)\, dx+O\left(\int_{\partial B_{\da}(0)}|x|\cdot|\nabla^{k-1}_\xi\hua|^2\, d\sigma\right)\\
&&=-\frac{n}{2}\int_{B_{\da}(0)} \hat{A}_\alpha^{(k-1)} (0) (\nabla^{k-1}_{\xi}\hua,\nabla^{k-1}_{\xi} \hua)\, dx+o(\theta_\alpha)\hbox{ if }n\geq 2k+2
\end{eqnarray*}
Putting all these estimates together yields
\begin{eqnarray}
&&(-1)^{k-1}\int_{B_{\da}(0)}\nabla^{k-1}_{\xi}(\hat{A}_\alpha^{(k-1)}  \nabla^{k-1}_{\xi}\hua) T(\hua)\, dx\\
&&=-\int_{B_{\da}(0)} \hat{A}_\alpha^{(k-1)} (0) (\nabla^{k-1}_{\xi}\hua,\nabla^{k-1}_{\xi} \hua)\, dx+o(\theta_\alpha)\label{eq:457}\\
&&=-\ma^2\int_{B_{\da/\ma}(0)} \hat{A}_\alpha^{(k-1)} (0) (\nabla^{k-1}_{\xi}\va,\nabla^{k-1}_{\xi} \va)\, dx+o(\theta_\alpha)\hbox{ when }n\geq 2k+2
\end{eqnarray}
Using \eqref{ineq:59} and \eqref{lim:tua:c2k:bis} and arguing as in the proof of \eqref{eq:81}, we get that
\begin{eqnarray*}
&&(-1)^{k-1}\int_{B_{\da}(0)}\nabla^{k-1}_{\xi}(\hat{A}_\alpha^{(k-1)}  \nabla^{k-1}_{\xi}\hua) T(\hua)\, dx\\
&&=\left\{\begin{array}{ll}
-\ma^2\int_{\rn} \hat{A}_0^{(k-1)} (0) (\nabla^{k-1}_{\xi}U,\nabla^{k-1}_{\xi}U)\, dx+o(\ma^2) &\hbox{if }n>2k+2\\
O\left(\ma^2\ln\frac{1}{\ma}\right)&\hbox{if }n=2k+2
\end{array}\right.
\end{eqnarray*}
where we have used that $|\nabla^{k-1}_\xi U|\in L^2(\rn)$ when $n>2k+2$ due to \eqref{ineq:59} and \eqref{lim:tua:c2k:bis}. With the expression \eqref{eq:76} and using that $\phia(0)=1$ and $\hbox{Ric}_{g_{\alpha}}(\xa)=0$, using formula (2.7) in Mazumdar-V\'etois \cite{mv} (see \eqref{exp:mv}), we get that $(\tilde{\hbox{exp}}_{\xa}^{g_{\xa}})^\star A^{(k-1)}_{\ga}(0)=0$ and then $\hat{A}_0^{(k-1)} (0)=(A_0^{(k-1)}-A_g^{(k-1)})_{x_0}$ via the chart $\tilde{\hbox{exp}}_{x_0}$, and then
\begin{eqnarray*}
&&(-1)^{k-1}\int_{B_{\da}(0)}\nabla^{k-1}_{\xi}(\hat{A}_\alpha^{(k-1)}  \nabla^{k-1}_{\xi}\hua) T(\hua)\, dx\\
&&=\left\{\begin{array}{ll}
-\ma^2\int_{\rn} (A_0^{(k-1)}-A_g^{(k-1)})_{x_0} (\nabla^{k-1}_{\xi}U,\nabla^{k-1}_{\xi}U)\, dx+o(\ma^2) &\hbox{if }n>2k+2\\
O\left(\ma^2\ln\frac{1}{\ma}\right)&\hbox{if }n=2k+2
\end{array}\right.
\end{eqnarray*}
Let us deal with the case $n=2k+2$ and $u_0\equiv 0$. It follows from Proposition \ref{prop:cv:1} that for any sequence $(y_\alpha)_\alpha\in B_\delta(0)$ such that
$$\lim_{\alpha\to 0}y_\alpha=0\, ,\, \lim_{\alpha\to 0}\frac{|y_\alpha|}{\ma}=+\infty,$$
then 
\begin{equation*}
\hua(y_\alpha)=\left(c_0^{\frac{n-2k}{2}}K_0+o(1)\right)\frac{\ma^{\frac{n-2k}{2}}}{|y_\alpha|^{n-2k}}\hbox{ if }u_0\equiv 0.
\end{equation*} 
Therefore, using \eqref{eq:457} and arguing as in Ghoussoub-Mazumdar-Robert (p84 of \cite{gmr:memams} or p537 of \cite{gmr:jde}), we get that
\begin{eqnarray}
&&(-1)^{k-1}\int_{B_{\da}(0)}\nabla^{k-1}_{\xi}(\hat{A}_\alpha^{(k-1)}  \nabla^{k-1}_{\xi}\hua) T(\hua)\, dx\label{est:log}\\
&&=
-\ma^2\ln\frac{1}{\ma}\int_{\mathbb{S}^{n-1}}\phi(\sigma)\, d\sigma+ o\left(\ma^2\ln\frac{1}{\ma}\right)\hbox{ if }u_0\equiv 0\hbox{ and }n=2k+2\nonumber
\end{eqnarray}
where
$$\phi(\sigma):=\left(c_0^{\frac{n-2k}{2}}K_0 \right)^2 (A_0^{(k-1)}-A_g^{(k-1)})_{x_0} (\nabla^{k-1}_{\xi}|x|^{2k-n},\nabla^{k-1}_{\xi}|x|^{2k-n})$$
Putting all these estimates together yields
\begin{equation}\label{estim:IIalpha}
II_\alpha=\left\{\begin{array}{ll}
-\ma^2\int_{\rn} (A_0^{(k-1)}-A_g^{(k-1)})_{x_0} (\nabla^{k-1}_{\xi}U,\nabla^{k-1}_{\xi}U)\, dx+o(\ma^2) &\hbox{if }n>2k+2\\
O\left(\ma^2\ln\frac{1}{\ma}\right)&\hbox{if }n=2k+2\\
-\ma^2\ln\frac{1}{\ma}\int_{\mathbb{S}^{n-1}}\phi(\sigma)\, d\sigma+ o\left(\ma^2\ln\frac{1}{\ma}\right)&\hbox{if }u_0\equiv 0\hbox{ and }n=2k+2
\end{array}\right.
\end{equation}
\subsection{Estimate of $III_\alpha$.} We let $(R_{ijkl})$ (we omit the index $\alpha$ for convenience) be the coordinates of the Riemann tensor $Rm_{\ga}$ at $\xa$ in the chart $\tilde{\hbox{exp}}_{\xa}$. Since $\hbox{Ric}_{\ga}(\xa)=0$, we have that $R_{ijpq}=W_{ijpq}$ where $(W_{ijpq})$ denotes the coordinates of the Weyl tensor at $\xa$ in the chart $\tilde{\hbox{exp}}_{\xa}$. With Proposition \ref{prop:lap}, we  get that
\begin{eqnarray*}
&&III_\alpha=\int_{B_{\da}(0)}\left(\frac{k}{3}W_{p iq j}x^p x^q\partial_{ij}\Delta^{k-1}_\xi \hua \right) T(\hua)\, dx\\
&&+\int_{B_{\da}(0)}\left( -B_{q i j p}x^q\partial_{ijp}\Delta_\xi^{k-2}\hua+ C_{ijpq}\partial_{ijpq}\Delta_\xi^{k-3}\hua\right)T(\hua)\, dx\\
&&\int_{B_{\da}(0)}\left( (x^3)
\star\nabla^{2k}\hua  +(x^2)\star\nabla^{2k-1}\hua +(x)\star\nabla^{2k-2}u+D^{(2k-3)}\hua\right) T(\hua)\, dx
\end{eqnarray*}
where  $(x^l)$  denotes any tensor fields of the form $P(x) T(x)$ where $T$ is a tensor field and $P$ is an homogeneous polynomial of degree $l$ and

$$B_{q i j p}=\frac{k(k-1)}{3}\left(W_{q i p j} +W_{p i q j}\right)\hbox{ and }C_{ijpq}=\frac{2k(k-1)(k-2)}{9}\left(W_{piqj}+W_{qipj}\right)$$
and $D^{(2k-3)}$ is a differential operator of order $\leq\max\{ 2k-3,0\}$. A standard symmetry of the Weyl tensor is that $W_{ijpq}=-W_{jipq}$, so that the terms involving $B_{ijpq}$ and $C_{ijpq}$ vanish. With \eqref{ineq:60}, the remainder terms are neglictible so that
\begin{equation}\label{eq:III:a}
 III_\alpha=\frac{k}{3}W_{p iq j}\int_{B_{\da}(0)} x^p x^q(\partial_{ij}\Delta^{k-1}_\xi \hua)  T(\hua)\, dx+o(\theta_\alpha)\hbox{ for }n\geq 2k+2.
\end{equation}
When $n>2k+2$, with the change of variable $x=\ma y$, arguing as in the proof of \eqref{eq:81}, using \eqref{lim:tua:c2k:bis} and \eqref{ineq:59}, and the Definition \ref{def:tensor:B} of $\hbox{Weyl}_g\otimes B$, we get that
\begin{eqnarray*}
 III_\alpha&=&\ma^2\frac{k}{3}W_{p iq j}\int_{B_{\da/\ma}(0)} y^p y^q(\partial_{ij}\Delta^{k-1}_\xi \va ) T(\va)\, dx+o(\ma^2)\\
&=&\ma^2 \hbox{Weyl}_g\otimes B+o(\ma^2)\hbox{ for }n>2k+2.
\end{eqnarray*}
Note that since $|\nabla^lU(x)|\leq C(1+|x|)^{2k-n-l}$ for $l\in\{0,...,2k\}$ and $x\in\rn$, then for $n>2k+2$, the integral  makes sense in the definition of $\hbox{Weyl}_g\otimes B$. We now deal with the case $n=2k+2$. The pointwise bound \eqref{ineq:59} in \eqref{eq:III:a} yield $III_\alpha=O\left(\ma^2\ln\frac{1}{\ma}\right)$ for $n= 2k+2$. When $u_0\equiv 0$, $n=2k+2$, as in \eqref{est:log}, we get
\begin{eqnarray*}
III_\alpha=  \ma^2\ln\frac{1}{\ma}\frac{k}{3}
\left(K_0c_0^{\frac{n-2k}{2}} \right)^2L+o\left(\ma^2\ln\frac{1}{\ma}\right)
\end{eqnarray*}
with $L:=W_{p iq j}\int_{\mathbb{S}^{n-1}} y^p y^q(\partial_{ij}\Delta^{k-1}_\xi |x|^{2k-n})  T(|x|^{2k-n})\, dx$. Computing explicitly the integral on the sphere and using again the symmetry properties of the Weyl tensor, we get that $L=0$, so that $III_\alpha=  o\left(\ma^2\ln\frac{1}{\ma}\right)$. Finally, we get that
\begin{equation}\label{estim:IIIalpha}
III_\alpha=\left\{\begin{array}{cc}
\ma^2 \hbox{Weyl}_g\otimes B+o(\ma^2)&\hbox{when }n>2k+2,\\
 O\left(\ma^2\ln\frac{1}{\ma}\right)&\hbox{when }n=2k+2,\\
o\left(\ma^2\ln\frac{1}{\ma}\right)&\hbox{when }n=2k+2\hbox{ and }u_0\equiv 0.
\end{array}\right.
\end{equation}

\subsection{Preliminary estimate of $IV_\alpha$.} In this situation, we take $\da=\sqrt{\ma}$. It follows from Proposition \ref{prop:cv:1} that
\begin{equation*}
\lim_{\alpha\to 0} w_\alpha(X)= u_0(x_0)+\Gamma_1(X)\hbox{ for all }X\in\rn\setminus \{0\}
\end{equation*} 
where $\wa(X):=\hua(\sqrt{\ma} X)$ and $\Gamma_1(X):=\frac{K_0 c_0^{\frac{n-2k}{2}}}{|X|^{n-2k}}$ for all $X\in\rn\setminus\{0\}$. Moreover, the convergence holds in $C^{2k}_{loc}(\rn-\{0\})$. Therefore, we get that
\begin{eqnarray*}
IV_\alpha&=&\int_{\partial B_{\sqrt{\ma}}(0)}\left((x,\nu)\left(\frac{(\Delta_\xi^\frac{k}{2}\hua)^2}{2}-\frac{\phia^{-\ea}|\hua|^{\crit-\ea}}{\crit-\ea}\right)+S(\hua)\right)\, d\sigma\\
&=&\ma^{\frac{n-2k}{2}}\int_{\partial B_{1}(0)}\left((x,\nu)\left(\frac{(\Delta_\xi^\frac{k}{2}\wa)^2}{2}-\frac{\phia^{-\ea}(\sqrt{\ma}x)\ma^{k}|\wa|^{\crit-\ea}}{\crit-\ea}\right)+S(\wa)\right)\, d\sigma\\
&=&\ma^{\frac{n-2k}{2}}\int_{\partial B_{1}(0)}\left((x,\nu) \frac{(\Delta_\xi^\frac{k}{2} \Gamma_1)^2}{2}  +S(u_0(x_0)+\Gamma_1)+o(1)\right)\, d\sigma\\
&=&\ma^{\frac{n-2k}{2}}\int_{\partial B_{1}(0)}\left((x,\nu) \frac{(\Delta_\xi^\frac{k}{2} \Gamma_1)^2}{2}  +S( \Gamma_1)\right)\, d\sigma\\
&&-\frac{n-2k}{2}\ma^{\frac{n-2k}{2}}\int_{\partial B_{1}(0)}\partial_\nu\Delta_\xi^{k-1}\Gamma_1 u_0(x_0)\, d\sigma+o(\ma^{\frac{n-2k}{2}})
\end{eqnarray*}
Since $-\partial_\nu \Delta_\xi^{k-1}(C_{n,k}|x|^{2k-n})=\omega_{n-1}^{-1}|x|^{1-n}$ (see \eqref{calc:delta:gamma}), we get that
\begin{eqnarray*}
IV_\alpha&=&\ma^{\frac{n-2k}{2}}\int_{\partial B_{1}(0)}\left((x,\nu) \frac{(\Delta_\xi^\frac{k}{2} \Gamma_1)^2}{2}  +S( \Gamma_1)\right)\, d\sigma\\
&&+\frac{n-2k}{2}c_0^{\frac{n-2k}{2}}\left(\int_{\rn}|U|^{\crit-2}U\, dX\right) u_0(x_0)\ma^{\frac{n-2k}{2}} +o(\ma^{\frac{n-2k}{2}})
\end{eqnarray*}
We now deal with the remaining boundary integral. Taking the identity of Proposition \ref{prop:poho} for $f\equiv 0$, $u\equiv\Gamma_1$ so that $\Delta_\xi^k u=0$ and $\Omega=B_1(0)-B_r(0)$ for $0<r<1$, we get that $D_r=D_1$ for all $0<r<1$ where
$$D_r:=\int_{\partial B_{r}(0)}\left((x,\nu) \frac{(\Delta_\xi^\frac{k}{2} \Gamma_1)^2}{2}  +S( \Gamma_1)\right)\, d\sigma.$$
A quick computation yields  the existence of $D_{k,n}\in\rr$ such that
$$\left(|x| \frac{(\Delta_\xi^\frac{k}{2} \Gamma_1)^2}{2}  +S( \Gamma_1)\right)=D_{k,n}|x|^{2k+1-2n}\hbox{ for all }x\in \rn\setminus\{0\},$$
so that $D_r=D_{k,n}\omega_{n-1}r^{2k-n}$ for all $0<r<1$. Since this quantity is independent of $r$, we get that $D_{k,n}=0$, so that $D_1=0$ and then
\begin{equation}\label{estim:IValpha:1}
IV_\alpha=\frac{n-2k}{2}c_0^{\frac{n-2k}{2}}\left(\int_{\rn}|U|^{\crit-2}U\, dX\right) u_0(x_0)\ma^{\frac{n-2k}{2}} +o(\ma^{\frac{n-2k}{2}})\hbox{ if }u_0\not\equiv 0.
\end{equation}
Let us deal with the case $u_0\equiv 0$, in this case, $\da=\delta>0$. Using Proposition \ref{prop:cv:2} and that $n=2k+1$, we get that
\begin{equation}\label{estim:IValpha:2}
IV_\alpha=O\left(\ma^{n-2k}\right)\hbox{ when }u_0\equiv 0.
\end{equation}

\subsection{Proof of Theorem \ref{th:general} and \ref{th:no:unull} for $n\geq 2k+2$.} We put the expressions \eqref{estim:Ialpha}, \eqref{estim:IIalpha}, \eqref{estim:IIIalpha}, \eqref{estim:IValpha:1} and \eqref{estim:IValpha:2} into \eqref{pohi:IaIV}. On the one hand, this yields $\ea=O(\ma^t\ln\frac{1}{\ma})$ for some $t>0$, and therefore $c_0=1$. On the other hand, separating the different dimensions, we get the asymptotics of Theorems \ref{th:general} and \ref{th:no:unull} for $n\geq 2k+2$.

\subsection{Proof of Theorem \ref{th:general} for $n= 2k+1$.} Using the control \eqref{ineq:59}, we get immediately that $II_\alpha=O(\ma^{n-2k})$ and $III_\alpha=O(\ma^{n-2k})$ when $n=2k+1$. We put these upper bounds together with the expressions \eqref{estim:Ialpha} and \eqref{estim:IValpha:1} into \eqref{pohi:IaIV}. On the one hand, we get $c_0=1$. On the other hand, we get the limit of Theorem \ref{th:general} for $n= 2k+1$. 

\subsection{Proof of Theorem  \ref{th:no:unull} when   $n=2k+1$ and $u_0\equiv 0$.}\label{subsec:mass} It follows from Proposition \ref{prop:cv:2} and the definition of $\hua$ that
\begin{equation}\label{lim:hua:G}
\lim_{\alpha\to 0}\frac{\hua(x)}{\ma^{\frac{n-2k}{2}}} =\lambda \hat{G}(x)\hbox{ for all }x\in B_{2\delta}(0)-\{0\}
\end{equation} 
and the convergence holds in $C^{2k}_{loc}(B_{2\delta}(0)-\{0\})$ where 
\begin{equation*}
\lambda:=c_0^{\frac{n-2k}{2}}\int_{\rn}|U|^{\crit-2}U\, dX\hbox{ and }\hat{G}(x):= \frac{G_0(x_0,\tilde{\hbox{exp}}_{x_0}^{g_{x_0}}(x))}{\varphi_{x_0}(\tilde{\hbox{exp}}_{x_0}^{g_{x_0}}(x))}\hbox{ for all }x\in B_{2\delta}(0)-\{0\}
\end{equation*}
and $\varphi_{x_0}$ is as in \eqref{def:phi} and $\tilde{\hbox{exp}}_{x_0}^{g_{x_0}}$ is an exponential map with respect to $g_{x_0}$ and $G_0$ is the Green function for the operator $P_0:=\Delta_g^k+\sum_{i=0}^{k-1} (-1)^i\nabla^i(A^{(i)}_0\nabla^i)$. With the conformal change of metric as in \eqref{eq:hua}, we get that 
\begin{equation}\label{eq:hG}
\Delta_{g_{x_0}}^k\hat{G}+\sum_{i=0}^{k-1} (-1)^i\nabla^i_{g_{x_0}}(\hat{A}_0^{(i)}  \nabla^i_{g_{x_0}}\hat{G})=0\hbox{ in }B_{2\delta}(0)-\{0\}.
\end{equation}
With \eqref{lim:hua:G}, we get that
\begin{eqnarray}
\frac{IV_\alpha}{\ma^{n-2k}}&=&\ma^{2k-n}\int_{\partial B_{\delta}(0)}\left((x,\nu)\left(\frac{(\Delta_\xi^\frac{k}{2}\hua)^2}{2}-\frac{\phia^{-\ea}|\hua|^{\crit-\ea}}{\crit-\ea}\right)+S(\hua)\right)\, d\sigma\nonumber\\
&=&\lambda^2\int_{\partial B_{\delta}(0)}\left((x,\nu) \frac{(\Delta_\xi^\frac{k}{2}\hat{G})^2}{2}+S(\hat{G})\right)\, d\sigma+o(1)\label{eq:77}
\end{eqnarray}
as $\alpha\to 0$. It follows from \eqref{ineq:59} that $|\nabla^l\hua|(x)\leq C\ma^{\frac{n-2k}{2}}|x|^{2k-n-l}$ for all $x\in B_{2\delta}(0)-\{0\}$ and $0\leq l\leq  2k$. Using also \eqref{lim:hua:G} and $n=2k+1$, we get that
\begin{equation}\label{eq:766}
\lim_{\alpha\to 0}\frac{III_\alpha-II_\alpha}{\ma^{n-2k}}= \lambda^2 \int_{B_{\delta}(0)}\left(\Delta_\xi^k\hat{G}  -\Delta_{g_{x_0}}^k \hat{G} -\sum_{i=0}^{k-1} (-1)^i\nabla^i_{g_{x_0}}\left(\hat{A}_0^{(i)}  \nabla^i_{g_{x_0}}\hat{G}\right)\right) T(\hat{G})\, dx.
\end{equation}
Using \eqref{eq:hG}, Proposition \ref{prop:lap} and \eqref{ineq:59}, we get that
\begin{eqnarray*}
\Delta_\xi^k\hat{G}&=& \Delta_\xi^k\hat{G}  -\Delta_{g_{x_0}}^k \hat{G} -\sum_{i=0}^{k-1} (-1)^i\nabla^i_{g_{x_0}}(\hat{A}_0^{(i)}  \nabla^i_{g_{x_0}}\hat{G})\\
&=& O\left(|x|^2 |\nabla^{2k}\hat{G}| +|x|\cdot |\nabla^{2k-1}\hat{G}|+\sum_{i=0}^{2k-2}|\nabla^{i}\hat{G}|\right)=O\left( |x|^{2-n}\right).
\end{eqnarray*}
As one checks, $T(\hat{G})=O(|x|^{2k-n})$. We put \eqref{estim:Ialpha}, \eqref{eq:77} and \eqref{eq:766} into \eqref{pohi:IaIV}. On the one hand, this yields $\ea=O(\ma)$ and then $c_0=\lim_{\alpha\to 0}\ma^{\ea}=1$. On the other hand, we get that 
\begin{eqnarray*}
\lim_{\alpha\to 0}\frac{\ea}{\ma}&=&\lambda^2\frac{ C_\delta}{\frac{n}{(\crit)^2}\int_{\rn} |U|^{\crit }\, dX}
\end{eqnarray*}
where
\begin{eqnarray*}
C_\delta&:=&\int_{B_{\delta}(0)}\left(\Delta_\xi^k\hat{G}  \right) T(\hat{G})\, dx-\int_{\partial B_{\delta}(0)}\left((x,\nu) \frac{(\Delta_\xi^\frac{k}{2}\hat{G})^2}{2}+S(\hat{G})\right)\, d\sigma
\end{eqnarray*}
Applying Proposition \ref{prop:poho} to $u\equiv\hat{G}$, $f\equiv 0$ and $\Omega=B_\delta(0)\setminus B_r(0)$ for $0<r<\delta$, we get that $C_\delta=C_r$ and then $C_\delta$ is independent of the choice of $\delta\to 0$. Since $\Delta_\xi^k\hat{G}=O\left( |x|^{2-n}\right)$ and $T(\hat{G})=O(|x|^{2k-n})$, using that $n=2k+1$, we get that
$$\int_{B_{\delta}(0)}\left(\Delta_\xi^k\hat{G}  \right) T(\hat{G})\, dx=o(1)\hbox{ as }\delta\to 0.$$ 

\smallskip\noindent We now justify and use the mass defined in \eqref{def:mass:intro}. The mass at $x_0$ is defined as follows. Let us fix $\eta\in C^\infty(M)$ such that $\eta(x)=1$ for $d_g(x,x_0)<i_g(M)/3$ and $\eta(x)=0$ for $d_g(x,x_0)>i_g(M)/2$. It follows from the construction of the Green's function that there exists $\beta_{x_0}\in C^0(M)\cap H_{2k}^p(M)$ for all $p<\frac{n}{n-2}$ such that
$$G_0(x,x_0)=\eta(x)\frac{C_{n,k}}{d_g(x,x_0)^{n-2k}}+\beta_{x_0}(x)\hbox{ for all }x\in M-\{x_0\}.$$
Note that we have that $P_0\beta_{x_0}=-P_0(\eta C_{n,k}d_g(\cdot,x_0)^{2k-n})=O(d_g(\cdot,x_0)^{2-n})$. We define the mass as
$$m_{P_0}(x_0):=\beta_{x_0}(x_0),$$
which is coherent with \eqref{def:mass:intro}. Since $G_0$ is the Green's function for $P_0$ on $(M,g)$, using the conformal change of metric $g_{x_0}=\varphi_{x_0}^{\frac{4}{n-2k}}g$, as one checks, the function 
$$(x,y)\mapsto H_0(x,y):=\frac{G_0(x,y)}{\varphi_{x_0}(x)\varphi_{x_0}(y)}$$
is the Green's function of the operator $u\mapsto P'_0:=\varphi_{x_0}^{1-\crit}P_0(\varphi_{x_0}u)=\Delta_{g_{x_0}}^k+\hbox{lot}$ on $(M, g_{x_0})$. Since $\varphi_{x_0}(x_0)=1$ and $\nabla\varphi_{x_0}(x_0)=0$, we have that $\hat{G}(x)=H_0(x_0,\tilde{\hbox{exp}}_{x_0}^{g_{x_0}}(x))$ for all $x\in B_{2\delta}(0)-\{0\}$, and it follows from elliptic theory that there exists $\hat{\beta}\in C^0(B_{2\delta}(0))$ such that
\begin{equation*}
\hat{G}(x)= \frac{C_{n,k}}{|x|^{n-2k}}+\hat{\beta}(x)\hbox{ for all }x\in B_{2\delta}(0)-\{0\},
\end{equation*}
with $\hat{\beta}(0)=m_{P_0}(x_0)$, $ |\nabla \hat{\beta}(x)|\leq C(1+|\ln |x||)$ and $|\nabla^l \hat{\beta}(x)|\leq C |x|^{1-l}$ for $2\leq l<2k$ and $x\in B_{2\delta}(0)-\{0\}$. Arguing as in the proof of \eqref{estim:IValpha:1} for $IV_\alpha$, we get that
$$\int_{\partial B_{\delta}(0)}\left((x,\nu) \frac{(\Delta_\xi^\frac{k}{2}\hat{G})^2}{2}+S(\hat{G})\right)\, d\sigma=\frac{n-2k}{2}\hat{\beta}(0)+o(1)\hbox{ as }\delta\to 0.$$
Therefore, we get that
\begin{eqnarray*}
\lim_{\alpha\to 0}\frac{\ea}{\ma}&=&-\lambda^2\frac{(n-2k) m_{p_0}(x_0) }{\frac{2n}{(\crit)^2}\int_{\rn} |U|^{\crit }\, dX}
\end{eqnarray*}
which yields the case $n=2k+1$ and $u_0\equiv 0$ of Theorem \ref{th:no:unull}. 

\section{Simplification for radial bubbles}\label{sec:extra}
\begin{proposition} We fix $n,k\in\nn$ be such that $2\leq 2k<n$. Let $U_{1,0}$ be defined as in \eqref{def:sol:pos} and let $A=(A^{i_1...i_{k-1}, j_1...j_{k-1}})$ be a $(2(k-1),0)-$tensor on $\rn$ which is symmetric in the sense that 
$A^{i_1...i_{k-1}, j_1...j_{k-1}}=A^{ j_1...j_{k-1},i_1...i_{k-1}}$ for all $i_1,....,i_{k-1}, j_1,..., j_{k-1}\in \{1,...,n\}$. We assume that $n>2k+2$. Then 
\begin{equation}
\int_{\rn} A(\nabla^{k-1} U_{1,0},\nabla^{k-1} U_{1,0})\, dx=\frac{\int_{\rn}(\Delta_\xi^{\frac{k-1}{2}} U_{1,0})^2\, dx}{n^{k-1}}\hbox{Tr}(A)
\end{equation}
where all the integrals make sense and $\hbox{Tr}(A)=\hbox{Tr}_{x_0}(A)$ is defined in \eqref{def:Tr:A}. 
 \end{proposition}
 \begin{proof} For convenience, we define $p:=k-1$ and $U:=U_{1,0}$. We have that $|\nabla^pU|=|\nabla^{k-1}U_{1,0}|\in L^2(\rn)$ since $n>2k+2$. Since $\nabla^pU$ is symmetric, we get that
$$\int_{\rn} (\sigma.A)(\nabla^{p} U,\nabla^{p} U)\, dx=\int_{\rn} A(\nabla^{p} U,\nabla^{p} U)\, dx$$
for all $\sigma \in {\mathcal S}_{2p}$ where $(\sigma.A)_{l_1...l_{2p}}=A_{l_{\sigma(1)}...l_{\sigma(2p)}}$. Therefore, we have that
$$\int_{\rn}A^S(\nabla^{p} U,\nabla^{p} U)\, dx=\int_{\rn} A(\nabla^{p} U,\nabla^{p} U)\, dx,$$
where $A^S$ is as in \eqref{def:Tr:A} is fully symmetric in the sense that $\sigma.A^S=A^S$ for all $\sigma\in {\mathcal S}_{2p}$. It follows from formula (2.40) in Mazumdar-V\'etois \cite{mv} that for all $i_1,...,i_{2p}\in \{1,...,n\}$, since $U$ is radially symmetric, we have that
$$\partial_{i_1...i_{2p}}U=\sum_{m=0}^p \frac{2^{2p-2m}}{m!(2p-2m)!}\partial_r^{2p-m}U \sum_{\sigma\in{\mathcal S}_{2p}}\delta_{i_{\sigma(1)}i_{\sigma(2)}}...\delta_{i_{\sigma(2m-1)}i_{\sigma(2m)}}x_{i_{\sigma(2m+1)}}...x_{i_{\sigma(2p)}}.$$
Moreover, formula (2.46) in \cite{mv} yields the existence of  $\beta(n,q)\in\rr$ such that
$$\int_{\snmoinsun}x_{i_{1}}...x_{i_{2q}}\, d\hbox{Vol}_{\snmoinsun}=\beta(n,q)\sum_{\sigma\in{\mathcal S}_{2q}}\delta_{i_{\sigma(1)}i_{\sigma(2)}}...\delta_{i_{\sigma(2q-1)}i_{\sigma(2q)}}$$
for $q\leq p$. Using the full symmetry of $A^S$, we then get that 
 there exists $C(p,n)\in\rr$ such that $\int_{\rn} A^S(\nabla^{p} U,\nabla^{p} U)\, dx=C(p,n)\hbox{Tr}(A^S)$ for any fully symmetric tensor $A^S$ where the trace is defined in \eqref{def:Tr:A}. Taking $A_{i_1...i_{2p}}=1$ for all index, we get that the value of $C(p,n)$. An extra integration by parts yields
\begin{eqnarray*}
\int_{\rn} A(\nabla^{p} U,\nabla^{p} U)\, dx
&=&\frac{\int_{\rn}|\nabla^p U|^2\, dx}{n^p}\hbox{Tr}(A^S)=\frac{\int_{\rn}(\Delta_\xi^{\frac{p}{2}} U)^2\, dx}{n^p}\hbox{Tr}(A^S)
\end{eqnarray*}\end{proof}

\medskip\noindent We now simplify expression \eqref{estim:1} and \eqref{estim:2} $U:=U_{1,0}$ is radial. Due to the radial symmetry, the term involving the Weyl tensor vanishes and we get for $n\geq 2k+4$
\begin{eqnarray*}
\lim_{\alpha\to 0} \frac{\ea}{\ma^2}\frac{n^k}{(\crit)^2}\frac{\int_{\rn} U_{1,0}^{\crit}\, dx}{\int_{\rn}(\Delta_\xi^{\frac{k-1}{2}} U_{1,0})^2\, dx}&=& 
Tr_{x_0}\left(A_0^{(k-1)}-A_{g}^{(k-1)}\right)-\kappa_k u_0(x_0){\bf 1}_{n=2k+4}
\end{eqnarray*}
where $\kappa_k$ is defined in \eqref{def:kappa}. This expression also holds when $u_0\equiv 0$ and $n>2k+2$.

\medskip\noindent We now deal with a particular case. We let $(P_\alpha)_{\alpha>0}\to P_0$ be of type (SCC). We assume that there exists a family $(h_\alpha)_\alpha\in C^{k-1,\theta}(M)$ and a family of $(2,0)-$tensors $(T_\alpha)_\alpha$ of class $C^{2(k-2),\theta}$ such that
$$ P_\alpha:=\Delta_g^k+\Delta_g^{\frac{k-1}{2}}\left(h_\alpha\Delta_g^{\frac{k-1}{2}}\right)-\Delta_g^{k-2}((T_\alpha,\nabla^2))+\hbox{lot for }\alpha>0\hbox{ or }=0$$
and $h_0\in C^{k-1,\theta}(M)$ such that $\lim_{\alpha\to 0}h_\alpha=h_0$ in $C^{k-1,\theta}(M)$ and $T_0$ is a $(2,0)-$tensor of class $C^{2(k-2),\theta}$ such that $\lim_{\alpha\to 0}T_\alpha=T_0$ in $C^{2(k-2),\theta}$. When $k=1$, we take $T_\alpha\equiv 0$. Then, taking the formalism of Definition \ref{def:scc}, \eqref{eq:hua} writes
\begin{eqnarray}\label{eq:hua:bis}
\Delta_{\ga}^k\hua+\sum_{i=0}^{k-1} (-1)^i\nabla^i_{\ga}(\hat{A}_\alpha^{(i)}  \nabla^i_{\ga}\hua)=\phia^{-\ea}|\hua|^{\crit-2-\ea}\hua\hbox{ in }B_\delta(0).
\end{eqnarray}
for some symmetric $(0, 2i)-$tensor $\hat{A}^{(i)}_\alpha$ for all $i=0,...,k-1$ and\begin{equation*}
(-1)^{k-1}\nabla^{k-1}_{\ga}(\hat{A}_\alpha^{(k-1)}  \nabla^{k-1}_{\ga}\hua)= \Delta_{\ga}^{\frac{k-1}{2}}(\hat{h}_\alpha \Delta_{\ga}^{\frac{k-1}{2}})-\Delta_{\ga}^{k-2}((\hat{T}_\alpha,\nabla_{\ga}^2))+\hbox{lot}
\end{equation*}
where $\hat{h}_\alpha:=\phia^{-\frac{4}{n-2k}}h_\alpha\circ \tilde{\hbox{exp}}_{\xa}$ and $\hat{T}_\alpha:=\phia^{-\frac{8}{n-2k}}\tilde{\hbox{exp}}_{\xa}^\star(T_\alpha-T_g)+\tilde{\hbox{exp}}_{\xa}^\star T_{\ga}$ and, see formula (2.7) in \cite{mv}, 
\begin{equation}\label{exp:mv}
T_g:=\frac{k(n-2)}{4(n-1)}R_g g-\frac{2k(k-1)(k+1)}{3(n-2)}\left(\hbox{Ric}_g-\frac{R_g}{2(n-1)}g\right).
\end{equation}
Using \eqref{ineq:60}, Lemma \ref{lem:deltap} and integrating by parts, we have that
\begin{eqnarray*}
&&\int_{B_{\da}(0)}\Delta_{\ga}^{\frac{k-1}{2}}(\hat{h}_\alpha \Delta_{\ga}^{\frac{k-1}{2}}\hua) T(\hua)\, dx=\int_{B_{\da}(0) }\hat{h}_\alpha \Delta_{\xi}^{\frac{k-1}{2}}\hua \Delta_{\xi}^{\frac{k-1}{2}}T(\hua)\, dx+o(\ma^2)\\
&&=\int_{B_{\da}(0) }\hat{h}_\alpha \Delta_{\xi}^{\frac{k-1}{2}}\hua \left(\frac{n-2}{2}\Delta_{\xi}^{\frac{k-1}{2}}\hua+x^i\partial_i\Delta_{\xi}^{\frac{k-1}{2}}\hua\right)\, dx+o(\ma^2)\\
&&=-\int_{B_{\da}(0) }\hat{h}_\alpha(\Delta_{\xi}^{\frac{k-1}{2}}\hua)^2 \, dx+o(\ma^2)\hbox{ for }n>2k+2.
\end{eqnarray*}
Using \eqref{ineq:59} and \eqref{lim:tua:c2k:bis} and arguing as in the proof of \eqref{eq:81}, we get that
\begin{equation*}
 \int_{B_{\da}(0)}\Delta_{\ga}^{\frac{k-1}{2}}(\hat{h}_\alpha \Delta_{\ga}^{\frac{k-1}{2}}\hua) T(\hua)\, dx \\
=
-\ma^2h_0(x_0)\int_{\rn} (\Delta_{\xi}^{\frac{k-1}{2}}U_{1,0})^2 \, dx+o(\ma^2)
\end{equation*}
when $n>2k+2$ where we have used that $(\Delta^{(k-1)/2}_\xi U_{1,0})\in L^2(\rn)$ due to \eqref{ineq:59} and \eqref{lim:tua:c2k:bis}. Using \eqref{ineq:60}, Lemma \ref{lem:deltap} and integrating by parts, we have that
\begin{eqnarray*}
&&\int_{B_{\da}(0)} -\Delta_{\tga}^{k-2}((\hat{T}_\alpha,\nabla^2\hua)) T(\hua)\, dx =\int_{B_{\da}(0)} -\Delta_{\xi}^{\frac{k-2}{2}}((\hat{T}_\alpha,\nabla^2\hua)) \Delta_{\xi}^{\frac{k-2}{2}}(T(\hua))\, dx \\
&&=\int_{B_{\da}(0)} -(\hat{T}_\alpha,\nabla^2\Delta_{\xi}^{\frac{k-2}{2}}\hua) \left(\frac{n-4}{2}\Delta_{\xi}^{\frac{k-2}{2}}\hua+x^p\partial_p\Delta_{\xi}^{\frac{k-2}{2}}\hua\right)\, dx
\end{eqnarray*}
We now set $V_\alpha:=\Delta_{\xi}^{\frac{k-2}{2}}\hua$. Using Lemma 2 in \cite{mesmar-robert}, we get that
\begin{eqnarray*}
&&\int_{B_{\da}(0)} -\Delta_{\tga}^{k-2}((\hat{T}_\alpha,\nabla^2\hua)) T(\hua)\, dx \\
&&=\int_{B_{\da}(0)} -(\hat{T}_\alpha,\nabla^2V_\alpha) \left(\frac{n-4}{2}V_\alpha+x^p\partial_pV_\alpha\right)\, dx\\
&&=-\frac{n-4}{2}\int_{B_{\da}(0)} \hat{T}_\alpha^{ij}V_\alpha \partial_{ij}V_\alpha\, dx-\int_{B_{\da}(0)} \hat{T}_\alpha^{ij}x^p\partial_{ij}V_\alpha\partial_pV_\alpha\, dx\\
&&=-\int_{B_{\da}(0)}\hat{T}_\alpha^{ij}\partial_i V_\alpha \partial_{j}V_\alpha\, dx+o(\ma^2)
\end{eqnarray*}
Using \eqref{ineq:59} and \eqref{lim:tua:c2k:bis} and arguing as in the proof of \eqref{eq:81}, we get that
\begin{eqnarray*}
&&\int_{B_{\da}(0)} -\Delta_{\tga}^{k-2}((\hat{T}_\alpha,\nabla^2\hua)) T(\hua)\, dx=-\ma^2\frac{\hbox{Tr}_{x_0}(\hat{T}_0)}{n}\int_{\rn}|\nabla \Delta_\xi^{\frac{k-2}{2}}U_{1,0}|^2\, dx+o(\ma^2)\\
&&=-\ma^2\frac{\hbox{Tr}_{x_0}(T_0-T_g)}{n}\int_{\rn}(\Delta_\xi^{\frac{k-1}{2}}U_{1,0})^2\, dx+o(\ma^2)
\end{eqnarray*}
Coming back to the expression involving the $A_\alpha^{(i)}$'s, we then get that
 \begin{equation*}
II_\alpha=
-\ma^2\left(h_0(x_0)+\frac{\hbox{Tr}_{x_0}(T_0-T_g)}{n}\right)\int_{\rn} (\Delta_{\xi}^{\frac{k-1}{2}}U_{1,0})^2 \, dx+o(\ma^2) \hbox{ when }n>2k+2.\end{equation*}
Note that 
$$\frac{\hbox{Tr}_{x_0}( T_g)}{n}=\frac{k(3n(n-2)-4(k^2-1))}{12n(n-1)}R_g(x_0).$$
We put the expressions \eqref{estim:Ialpha}, $II_\alpha$ above, \eqref{estim:IIIalpha}, \eqref{estim:IValpha:1} and \eqref{estim:IValpha:2} into \eqref{pohi:IaIV} so that for $n\geq 2k+4$, \eqref{estim:1} rewrites
\begin{eqnarray*}
\lim_{\alpha\to 0} \frac{\ea}{\ma^2}\frac{n}{(\crit)^2}\frac{\int_{\rn} U_{1,0}^{\crit}\, dX}{\int_{\rn} (\Delta_{\xi}^{\frac{k-1}{2}}U_{1,0})^2 \, dx}&=&   h_0(x_0)+\frac{\hbox{Tr}_{x_0}(T_0-T_g)}{n} -\kappa_k' u_0(x_0){\bf 1}_{n=2k+4}\end{eqnarray*}
where $\kappa_k'$ is defined in \eqref{def:kappa}. Note that we have assumed that $k\geq 2$. When $T_\alpha\equiv T_0\equiv 0$, we get that for $n\geq 2k+4$, \eqref{estim:1} rewrites
\begin{eqnarray*}
\lim_{\alpha\to 0} \frac{\ea}{\ma^2}\frac{n}{(\crit)^2}\frac{\int_{\rn} U_{1,0}^{\crit}\, dX}{\int_{\rn} (\Delta_{\xi}^{\frac{k-1}{2}}U_{1,0})^2 \, dx}&=&   h_0(x_0)-\frac{k(3n(n-2)-4(k^2-1))}{12n(n-1)}R_g(x_0)  \\
&&-\kappa_k' u_0(x_0){\bf 1}_{n=2k+4} \end{eqnarray*}
which turns to be valid also when $k=1$. When $u_0\equiv 0$, the same result hold for $n>2k+2$. 

\smallskip\noindent Theorem \ref{th:teaser:26} and \ref{th:teaser:25} are consequences of these limits when $n>2k+2$.

 \appendix
\section{Expansion of the Riemannian Laplacian}
Recall that given $T$ and $S$ two tensor fields on $M$, the notation $T\star S$ will denote any linear combination of products of contractions of $T$ and $S$. The notation $(x^l)$ will denote any tensor fieldof the form $P(x) T(x)$ where $T$ is another tensor field and $P$ is  homogeneous  of degree $l$. Fix a point $p\in M$ and consider an exponential chart $\tilde{\hbox{exp}}_p$ with respect to the metric $g$. According to Cartan's expansion, the coordinates of the metric $g$ in the chart $\tilde{\hbox{exp}}_p$ writes
\begin{equation}\label{exp:cartan}
g_{ij}(x)=\delta_{ij}-\frac{1}{3}R_{p iq j}x^p x^q+(x^3)\hbox{ as }x\to 0,
\end{equation}
where $R_{ijkl}$ denotes the coordinates of the Riemann tensor $Rm_g$ at $p$ in the same chart. This expression is locally uniform in $p$ in a neighborhood of a fixed point $p_0\in M$. Therefore, the coordinates of the inverse metric tensor are
\begin{equation}\label{exp:cartan:inverse}
g^{ij}(x)=\delta^{ij}+\frac{1}{3}R_{p iq j}x^p x^q+(x^3)\hbox{ as }x\to 0,
\end{equation}
where one shall remember that the exponential chart is normal at $p$. This expression is locally uniform in $p$. In any chart, the Riemannian Laplacian $\Delta_g$ writes $\Delta_g:=-g^{ij}\left(\partial_{ij}-\Gamma_{ij}^l\partial_l\right)$. We have the following result:
\begin{proposition}\label{prop:lap} Let $(M,g)$ be a Riemannian manifold and let $p\in M$ be a point of $M$. Assume that $\hbox{Ric}_g(p)=0$. Then, in the exponential chart $\tilde{\hbox{exp}}_p$, we have that
\begin{eqnarray*}
\Delta_g^k u&=&\Delta_\xi^k u-\frac{k}{3}R_{p iq j}x^p x^q\partial_{ij}\Delta^{k-1}_\xi u+(x)^3
\star\nabla^{2k}u\\
&&+B_{q i j p}x^q\partial_{ijp}\Delta_\xi^{k-2}u+(x^2)\star\nabla^{2k-1}u\\
&&-C_{ijpq}\partial_{ijpq}\Delta_\xi^{k-3}u+(x)\star\nabla^{2k-2}u+D^{(2k-3)}u
\end{eqnarray*}
where 
$$B_{q i j p}=\frac{k(k-1)}{3}\left(R_{q i p j} +R_{p i q j}\right)\hbox{ and }C_{ijpq}=\frac{2k(k-1)(k-2)}{9}\left(R_{piqj}+R_{qipj}\right)$$
and $D^{(2k-3)}$ is a differential operator of order $\leq\max\{ 2k-3,0\}$. Moreover, this expansion is locally uniform with respect to $p$.
\end{proposition}
The proof is  by iteration on $k$ and uses the expression of $\Delta_g$, \eqref{exp:cartan} and \eqref{exp:cartan:inverse}.

\section{Hardy-type inequality on manifolds}\label{app:hardy}
For the sake of completeness, we prove the well-known Hardy inequality \eqref{hardy:ineq:intro}.
\begin{proposition}\label{prop:hardy}[Hardy-inequalities on manifolds] Let $(M,g)$ be a compact Riemannian manifold of dimension $n\geq 3$. Let $k\in\nn$ be such that $2\leq 2k<n$. Then for any $u\in H_k^2(M)$, we have that $u^2d_g(\cdot,x_0)^{-2k}\in L^1(M)$. Moreover, there exists $C_H(M,g,k)>0$ such that 
\begin{equation}\label{ineq:hardy:app}
\int_M\frac{u^2\, dv_g}{d_g(x,x_0)^{2k}} \leq C_H(M,g,k)\Vert u\Vert_{H_k^2}^2\hbox{ for all }u\in H_k^2(M).
\end{equation}
\end{proposition}

{\bf Proof of Proposition \ref{prop:hardy}.} The starting point is the following Hardy-type inequality on $\rn$ (see Theorem 3.3 in Mitidieri \cite{mitidieri}): there exists $C_H(n,k)>0$ such that
\begin{equation*}
\int_{\rn}\frac{\varphi^2}{|X|^{2k}}\, dX\leq C_H(n,k)\int_{\rn}(\Delta_\xi^{\frac{k}{2}}\varphi)^2\, dX\hbox{ for all }\varphi\in C^\infty_c(\rn).
\end{equation*}
Let us choose $u\in C^\infty(M)$. Let us fix $\delta\in (0, i_g(M)/3)$. Let us define a cutoff function $\eta\in C^\infty(M)$ be such that $\eta(x)=1$ for $x\in B_\delta(x_0)$ and $\eta(x)=0$ for $x\in M-B_{2\delta}(x_0)$. We then have that
\begin{eqnarray*}
\int_M\frac{(\eta(x) u(x))^2}{d_g(x,x_0)^{2k}}\, dv_g&=&\int_{B_{2\delta}(x_0)}\frac{(\eta u)^2\, dv_g}{d_g(x,x_0)^{2k}}=\int_{B_{2\delta}(0)}\frac{((\eta u)\circ \tilde{\hbox{exp}}_{x_0}(X))^2}{|X|^{2k}}\, dv_{\tilde{\hbox{exp}}_{x_0}^\star g}\\
&\leq& C(M,g)\int_{B_{2\delta}(0)}\frac{((\eta u)\circ \tilde{\hbox{exp}}_{x_0}(X))^2}{|X|^{2k}}\, dX\\
&\leq& C(M,g) C_H(n,k)\int_{B_{2\delta}(0)}(\Delta_\xi^{\frac{k}{2}}((\eta u)\circ \tilde{\hbox{exp}}_{x_0}))^2\, dX\\
&\leq & C(M,g,k)\int_{B_{2\delta}(0)}\sum_{i=0}^k|\nabla^i((\eta u)\circ \tilde{\hbox{exp}}_{x_0})|_\xi^2\, dX
\end{eqnarray*}
Coming back to the manifold $(M,g)$, we get that
\begin{eqnarray*}
\int_M\frac{(\eta(x) u(x))^2}{d_g(x,x_0)^{2k}}\, dv_g&\leq & C(M,g,k)\int_{B_{2\delta}(x_0)}\sum_{i=0}^k|\nabla_g^i(\eta u)|_g^2\, dv_g\\
&\leq& C(M,g,k)\int_{M}\sum_{i=0}^k|\nabla_g^iu|_g^2\, dv_g\end{eqnarray*}
It follows from \cite{robert:ccm} (Proposition 2) that there exists $C(k)>0$ such that
$$\int_{M}\sum_{i=0}^k|\nabla_g^iu|_g^2\, dv_g\leq C_1(M,g,k)\Vert u\Vert_{H_k^2}^2\hbox{ for all }u\in H_k^2(M).$$
We then get that
\begin{equation*}
\int_M\frac{(\eta u)^2}{d_g(x,x_0)^{2k}}\, dv_g \leq C_2(M,g,k)\Vert u\Vert_{H_k^2}^2\hbox{ for all }u\in C^\infty(M)
\end{equation*}
Independently, a straightforward computation yields
\begin{eqnarray*}
\int_M\frac{(1-\eta^2) u^2}{d_g(x,x_0)^{2k}}\, dv_g&=&\int_{M-B_{\delta}(x_0)}\frac{(1-\eta^2) u^2}{d_g(x,x_0)^{2k}}\, dv_g \leq \delta^{-2k}\int_{M } u^2\, dv_g \leq \delta^{-2k}\Vert u\Vert_{H_k^2}^2.
\end{eqnarray*}
Combining this two latest inequality yields \eqref{ineq:hardy:app} for all $u\in C^\infty(M)$. Inequality  \eqref{ineq:hardy:app} follows by density.

\section{Green's function for elliptic operators with bounded coefficients}\label{app:green}
\begin{defi}\label{def:green} Let $(M,g)$ be a compact Riemannian manifold without boundary of dimension $n\geq 2$. Fix $k\in\nn$ such that $n>2k\geq 2$. Let $P$ be an elliptic operator of order $2k$. A \emph{Green's function} for $P$ is a function $(x,y)\mapsto G(x,y)=G_x(y)$ defined for all $x\in M$ and a.e. $y\in M$ such that 

\begin{itemize}
\item[(i)] $G_x\in L^1(M)$ for all $x\in M$,
\item[(ii)] for all $x\in M$ and all $\varphi\in C^{2k}(M)$, we have that
$$\int_M G_xP\varphi\, dv_g=\varphi(x).$$
\end{itemize}
\end{defi}

\begin{theorem}\label{th:green:nonsing} Let $(M,g)$ be a compact Riemannian manifold without boundary of dimension $n\geq 2$. Fix $k\in\nn$ such that $n>2k\geq 2$. Let $P$ be an elliptic operator of type $O_{k,L}$. Then there exists a unique Green's function for $P$. Moreover,

\begin{itemize}
\item $G$ extends to $M\times M\setminus\{(x,x)/x\in M\}$ and for any $x\in M$, $G_x\in H_{2k,loc}^p(M-\{x\})$ for all $p>1$;
\item $G$ is symmetric;
\item $P G_x=0$ weakly in $M-\{x\}$ for all $x\in M$
\item For all $f\in L^p(M)$, $p>\frac{n}{2k}$, and $\varphi\in H_{2k}^p(M)$ such that $P\varphi=f$, then 
$$\varphi(x)=\int_M G_xP\varphi\, dv_g\hbox{ for all }x\in M.$$
\item There exists $C(k,L)>0$ such that
$$|G_x(y)|\leq C(k,L)\cdot d_g(x,y)^{2k-n}\hbox{ for all }x,y\in M, \, x\neq y;$$
\item For all $l=1,...,2k-1$, there exists $C_l(k,L)>0$ such that
$$|\nabla^l G_x(y)|\leq C_l(k,L)\cdot d_g(x,y)^{2k-n-l}\hbox{ for all }x,y\in M, \, x\neq y;$$
\end{itemize}
The same conclusion holds if in Definition \ref{def:p} of $P\in O_{k,L}$, the $0^{th}-$order potential $A^{(0)}$ is only in $L^\infty(M)$ with $\Vert A^{(0)}\Vert_{L^\infty(M)}\leq L$.
\end{theorem}

\smallskip\noindent{\bf Remark:} Note that if the coefficients of $P$ are in $C^{p,\theta}(M)$, $p\in\nn$ and $0<\theta<1$, then $G\in C^{p+2k,\theta}$ and similar upper bounds hold for $|\nabla^l G_x|$ with $l\leq p+2k$.

We build the Green's function via the classical Neumann series. We follow the proof in Robert \cite{robert:green}. We write $P=\Delta_g^k+Q$ where $Q$ is a differential operator of order at most $2k-2$. Since $P$ is self-adjoint, then $Q$ is also self-adjoint. Let $\eta\in C^\infty(\rr)$ be such that $\eta(t)=1$ if $t\leq i_g(M)/4$ and $\eta(t)=0$ if $t\geq i_g(M)/2$. We then define
\begin{equation*}
\Gamma_x(y)=\Gamma(x,y):=\eta(d_g(x,y)) C_{n,k}d_g(x,y)^{2k-n}\hbox{ for all }x,y\in M,\; x\neq y.
\end{equation*}
where $C_{n,k}$ is defined in \eqref{def:Cnk}.

\medskip\noindent{\bf Step 1:} We claim that for all $x\in M$, there exists $f_x\in L^1(M)$ such that
\begin{equation}\label{lap:gamma}
\left\{\begin{array}{cc}
P\Gamma_x=\delta_x-f_x&\hbox{ weakly in  }M\\
|f_x(y)|\leq C(k,L) d_g(x,y)^{2-n}&\hbox{ for all }x,y\in M,\, x\neq y,
\end{array}\right\}
\end{equation} 
Where the equality is to be taken in the distribution sense, that is 
$$\int_M \Gamma_x P\varphi\, dv_g=\varphi(x)-\int_M f_x\varphi\, dv_g\hbox{ for all }\varphi\in C^{2k}(M).$$
We prove the claim. We fix $\varphi\in C^{2k}(M)$ and $x\in M$. Since $\Gamma_x\in L^1(M)$, we have that
\begin{eqnarray}
\int_M \Gamma_x P\varphi\, dv_g&=& \lim_{\eps\to 0}\int_{M-B_{\eps}(x)} \Gamma_x P\varphi\, dv_g\nonumber \\
&=&\lim_{\eps\to 0}\int_{M-B_{\eps}(x)} \Gamma_x \Delta_g^k\varphi\, dv_g + \lim_{\eps\to 0}\int_{M-B_{\eps}(x)} \Gamma_x Q\varphi\, dv_g \label{eq:gamma:2}
\end{eqnarray}
Since $Q$ is a self-adjoint differential operator, integrating by parts, we get that 
\begin{eqnarray*}
\int_{M-B_{\eps}(x)} \Gamma_x Q\varphi\, dv_g &=& \int_{M-B_{\eps}(x)} (Q\Gamma_x)\varphi\, dv_g +\int_{\partial (M-B_{\eps}(x))}\sum_{i+j<2k-2}\nabla^i\Gamma_x\star\nabla^j\varphi\, dv_g\\
&=& \int_{M-B_{\eps}(x)} (Q\Gamma_x)\varphi\, dv_g +\int_{ \partial B_{\eps}(x)}\sum_{i+j<2k-2}\nabla^i\Gamma_x\star\nabla^j\varphi\, dv_g
\end{eqnarray*}
As one checks, 
\begin{equation}\label{bnd:der:gamma}
|\nabla^i\Gamma_x(y)|\leq C(k,i) d_g(x,y)^{2k-n-i}\hbox{ for }i\leq 2k\hbox{ and }x,y\in M,\, x\neq y.
\end{equation}
Therefore, for $i+j<2k-2$, we have that
\begin{eqnarray*}
\left|\int_{ \partial B_{\eps}(x)} \nabla^i\Gamma_x\star\nabla^j\varphi\, dv_g\right| &\leq & C \eps^{n-1}\eps^{2k-n-i}=C \eps^{2k-1-i}=o(1)
\end{eqnarray*}
as $\eps\to 0$ since $i<2k-2$. Therefore
\begin{equation}\label{eq:gamma:3}
\lim_{\eps\to 0}\int_{M-B_{\eps}(x)} \Gamma_x Q\varphi\, dv_g = \int_{M} (Q\Gamma_x) \varphi\, dv_g 
\end{equation}
Concerning the first term of \eqref{eq:gamma:2}, we argue as in the proof of Step 1.3, see formula \eqref{ipp:gamma}. Given $\Omega\subset M$ a smooth domain, we have that  
\begin{equation}\label{ipp:gamma:g}
\int_{\Omega }(\Delta_g^kU)V\, dv_g=\int_{\Omega}U(\Delta_g^kV)\, dv_g+\int_{\partial\Omega}\sum_{i=0}^{k-1}\tilde{B}_{i,g}(U,V)\, d\sigma_g
\end{equation}
for all $U,V\in C^{2k}(\overline{\Omega})$, where for any $i=1...,k-1$, we have that
$$\tilde{B}_{i,g}(U,V):=-\partial_\nu\Delta_g^{k-i-1}U\Delta_g^iV+\Delta_g^{k-i-1}U\partial_\nu\Delta_g^iV.$$
Therefore 
\begin{eqnarray*}
\int_{M-B_{\eps}(x)} \Gamma_x \Delta_g^k\varphi\, dv_g &=& \int_{M-B_{\eps}(x)}\varphi (\Delta_g^k\Gamma_x)\, dv_g+\int_{\partial (M-B_{\eps}(x))}\sum_{i=0}^{k-1}\tilde{B}_{i,g}(\varphi,\Gamma_x)\, d\sigma_g\\
&=& \int_{M-B_{\eps}(x)}\varphi (\Delta_g^k\Gamma_x)\, dv_g-\sum_{i=0}^{k-1}\int_{\partial  B_{\eps}(x)}\tilde{B}_{i,g}(\varphi,\Gamma_x)\, d\sigma_g
\end{eqnarray*}
Given $i\in \{0, ...,k-1\}$ and using  \eqref{bnd:der:gamma} we have that
\begin{equation*}
\left|\int_{\partial  B_{\eps}(x)}\tilde{B}_{i,g}(\varphi,\Gamma_x)\, d\sigma_g\right|\leq C \eps^{n-1}\eps^{2k-n-(2i+1)}=C\eps^{2(k-1-i)}=o(1)\hbox{ if }i<k-1,
\end{equation*}
as $\eps \to 0$. Similarly,
\begin{equation*}
\left|\int_{\partial  B_{\eps}(x)}\partial_\nu\varphi \Delta_g^{k-1}\Gamma_x\, d\sigma_g\right|\leq C \eps^{n-1}\eps^{2k-n-2(k-1)}=C\eps =o(1),
\end{equation*}
as $\eps\to 0$. Therefore
\begin{eqnarray}\label{eq:gamma:1}
\int_{M- B_{\eps}(x)} \Gamma_x \Delta_g^k\varphi\, dv_g &=&   \int_{M- B_{\eps}(x)}\varphi (\Delta_g^k\Gamma_x)\, dv_g\\
&&- \int_{\partial  B_{\eps}(x)}\varphi \partial_\nu\Delta_g^{k-1}\Gamma_x\, d\sigma_g+o(1)
\end{eqnarray}
as $\eps\to 0$. For any radial function $W\in C^2(M)$ we have that in radial coordinates
$$\Delta_g W=-\frac{\partial_r(\sqrt{|g|}r^{n-1}\partial_r W)}{\sqrt{|g|}r^{n-1}}$$
with $|g|=1+O(r^2)$ being the determinant of the metric in radial coordinates. Note that with this expression, we get that for any $\alpha>0$, 
\begin{equation}\label{calc:lap}
\Delta_g r^{-\alpha}=\alpha(n-2-\alpha)r^{-\alpha-2}+O(r^{-\alpha})\hbox{ with }r:=d_g(x,\cdot).
\end{equation}
So as for \eqref{calc:delta:gamma}, we get that
\begin{equation}
-\partial_\nu\Delta_g^{k-1}\Gamma_x(y)=\frac{1}{\omega_{n-1}}d_g(x,y)^{1-n}+O(d_g(x,y)^{2-n})\hbox{ for }y\in M-\{x_0\}.
\end{equation}
So that, with a change of variable, we get that
$$\lim_{\eps\to 0}\int_{\partial  B_{\eps}(x)}\left(- \partial_\nu\Delta_g^{k-1}\Gamma_x\right)\varphi\, d\sigma_g=\varphi(x).$$
Plugging this latest identity into \eqref{eq:gamma:1}, \eqref{eq:gamma:2} and \eqref{eq:gamma:3}, we get 
\begin{equation*}
\int_M \Gamma_x P\varphi\, dv_g=\varphi(x)+\lim_{\eps\to 0}\int_{M-B_{\eps}(x)} (\Delta_g^k\Gamma_x +Q\Gamma_x)\varphi\, dv_g
\end{equation*}
Iterating \eqref{calc:lap} and using \eqref{bnd:der:gamma}, we get that 
$$f_x:=-(\Delta_g^k\Gamma_x +Q\Gamma_x)\, ;\, |f_x(y)|\leq C(k,L)d_g(x,y)^{2-n}\hbox{ for }x,y\in M,\, x\neq y.$$
We then get that $f_x\in L^1(M)$ and  \eqref{lap:gamma} with the above definition of $f$. This proves the claim and ends Step 1.

\medskip\noindent{\bf Step 2:} We are now in position to define the Green's function $G$. We define
\begin{equation*}
\left\{\begin{array}{cc}
\Gamma_1(x,y):=f_x(y) &\hbox{ for }x,y\in M,\, x\neq y,\\
\Gamma_{i+1}(x,y):=\int_M \Gamma_i(x,z)f_z(y)\, dv_g(z)&\hbox{ for }x,y\in M,\, x\neq y,\, i\in\nn
\end{array}\right\}
\end{equation*}
Giraud's Lemma \cite{giraud}, as stated in \cite{DHR} for instance, is
\begin{lemma} Let $(M,g)$ be a compact Riemannian manifold of dimension $n\geq 2$. Let $\alpha,\beta\in (0,n)$ be two real number and let $X,Y:M\times M\to\rr$ measurable be such that there exists $C_X, C_Y>0$ such that
$$|X(x,y)|\leq C_X d_g(x,y)^{\alpha-n}\hbox{ and }|Y(x,y)|\leq C_Y d_g(x,y)^{\beta-n}\hbox{ for all }x,y\in M,\, x\neq y.$$
Define $Z(x,y):=\int_M X(x,z)Y(z,y)\, dv_g(z)$ for $x,y\in M$ when this makes sense. Then
\begin{itemize}
\item $Z(x,y)$ is defined for all $x,y\in M$, $x\neq y$ and $Z$ is measurable;
\item There exists $C=C(\alpha,\beta)$ be such that for all $x\neq y$ in $M$,
\end{itemize} 
\begin{equation*}
|Z(x,y)|\leq C\cdot C_X\cdot C_Y \left\{\begin{array}{cc}
d_g(x,y)^{\alpha+\beta-n} &\hbox{ if }\alpha+\beta<n;\\
1+|\ln d_g(x,y)|&\hbox{ if }\alpha+\beta=n;\\
1 &\hbox{ if }\alpha+\beta>n.\end{array}\right.
\end{equation*}
\end{lemma}
Iterating Giraud's Lemma and using the definition of $\Gamma$ and \eqref{lap:gamma}, we get that for all $i\in\nn$, there exists $C_i(k,L)>0$ such that
\begin{equation}\label{ctrl:gamma:i}
|\Gamma_i(x,y)|\leq C_i(k,L) \left\{\begin{array}{cc}
d_g(x,y)^{2i-n} &\hbox{ if }2i<n;\\
1+|\ln d_g(x,y)|&\hbox{ if }2i=n;\\
1 &\hbox{ if }2i>n.\end{array}\right.
\end{equation}
We then get that $\Gamma_i\in L^\infty(M\times M)$ for all $i>\frac{n}{2}$. We fix $p>n$. For $x\in M$, we take $u_x\in H_{2k}^2(M)$ that will be fixed later, and we define
\begin{equation}\label{def:G}
G_x(y):=\Gamma_x(y)+\sum_{i=1}^p\int_M\Gamma_i(x,z)\Gamma(z,y)\, dv_g(z)+ u_x(y)\hbox{ for a.e }y\in M.
\end{equation} We fix $\varphi\in C^{2k}(M)$. Via Fubini's theorem, using  the definition of the $\Gamma_i$'s, the self-adjointness of $P$ and \eqref{lap:gamma}, we get that
\begin{eqnarray*}
\int_M G_x P\varphi\, dv_g &=& \int_M \Gamma_x P\varphi\, dv_g +\sum_{i=1}^p\int_{M\times M} \Gamma_i(x,z)\Gamma(z,y)P\varphi(y)\, dv_g(z)dv_g(y)\\
&& +\int_M Pu_x\varphi\, dv_g\\
&=& \varphi(x)-\int_M f_x\varphi\, dv_g +\sum_{i=1}^p\int_{M } \Gamma_i(x,z)\left(\varphi(z)-\int_M f_z\varphi\, dv_g\right)\,dv_g(z)\\
&& +\int_M Pu_x\varphi\, dv_g\\
&=& \varphi(x)-\int_M \Gamma_1(x,\cdot)\varphi\, dv_g +\sum_{i=1}^p\int_{M } \Gamma_i(x,z) \varphi(z) \,dv_g(z)\\
&&-\sum_{i=1}^p\int_{M }\left(\int_M \Gamma_i(x,z)   f_z(y)\, dv_g(z)\right)\varphi(y)\,  dv_g(y)+\int_M Pu_x\varphi\, dv_g\\
&=& \varphi(x) +\sum_{i=2}^p\int_{M } \Gamma_i(x,z) \varphi(z) \,dv_g(z)\\
&&-\sum_{i=1}^p\int_{M }\Gamma_{i+1}(x,y)\varphi\,  dv_g(y)+\int_M Pu_x\varphi\, dv_g\\
&=& \varphi(x)  +\int_{M }(Pu_x-\Gamma_{p+1}(x,\cdot))\varphi\,  dv_g(y)
\end{eqnarray*}
Since $\Gamma_{p+1}(x,\cdot)\in L^\infty(M)$, we choose $u_x\in \cap_{q>1}H_{2k}^q(M)$  such that $Pu_x=\Gamma_{p+1}(x,\cdot)$ weakly in $M$. The existence follows from Theorem \ref{th:3}. Sobolev's embedding theorem yields $u_x\in C^{2k-1}(M)$ and Theorem \ref{th:3} yields  $ C( k,L,p)>0$ such that
\begin{equation}\label{ctrl:u}
|u_x(y)|\leq C(k,L,p)\hbox{ for all }x,y\in M.
\end{equation}
Finally, we get that 
\begin{equation}\label{id:green}
\int_M G_x P\varphi\, dv_g=\varphi(x)\hbox{ for all }\varphi\in C^{2k}(M).
\end{equation}
It follows from the pointwise controls \eqref{ctrl:gamma:i}, \eqref{ctrl:u}, the definition \eqref{def:G} and Giraud's Lemma that
\begin{equation}\label{ctrl:G}
\left\{\begin{array}{c}
|G_x(y)|\leq C( k, L)d_g(x,y)^{2k-n}\\
\left|G_x(y)- C_{n,k}d_g(x,y)^{2k-n}\right|\leq C( k, L)d_g(x,y)^{2k+1-n}
\end{array}\right\} \hbox{ for all }x,y\in M, \, x\neq y.
\end{equation}
This proves the existence of a Green's function for $P$. Moreover, the construction yields $G_x\in H_{2k,loc}^p(M-\{x\})$ for all $p>1$. 

\medskip\noindent We briefly conclude the proof of Theorem \ref{th:green:nonsing}.  Given $x\in M$, we have that $P G_x=0$ in the distributional sense in $M-\{x\}$. The validity of \eqref{id:green} for $u\in H_{2k}^p(M)$ and $f\in L^p(M)$ such that $Pu=f$ and $p>n/(2k)$ follows by density of $C^\infty(M)$ in $L^p(M)$ and the regularity Theorem \ref{th:3}. The symmetry of $G$ follows from the self-adjointness of the operator $P$. The uniqueness is as the proof of uniqueness of Theorem \ref{th:Green:main}. The pointwise control for $|G_x(y)|$ is exactly \eqref{ctrl:G}. The control of the gradient of $G_x$ is a consequence of elliptic theory. Since the details of these points are exactly the same as in the case of a second-order operator $\Delta_g+h$, we refer to the detailed construction \cite{robert:green}.

\section{Regularity theorems}\label{sec:regul:adn}
The following theorems are reformulations or consequences of Agmon-Douglis-Nirenberg \cite{ADN}.%\theorem 15.1 essentiellement
\begin{theorem}\label{th:1} Let $(M,g)$ be a  Riemannian manifold of dimension $n$ without boundary. Fix $k\in\nn$ be such that $2\leq 2k<n$ , $L>0$ and $\delta>0$ and $x_0\in M$. Let $P$ be a differential operator of type  $O_{k,L}(B_\delta(x_0))$ (see Definition \ref{def:P:loc}). Let $u\in H_{2k}^p(B_\delta(x_0))$ and $f\in L^q(B_\delta(x_0))$, $p,q\in (1,+\infty)$ be such that $Pu=f$. Then $u\in H_{2k}^q(B_{\delta'}(x_0))$ for all $\delta'<\delta$.
\end{theorem}

\begin{theorem}\label{th:2} Let $(M,g)$ be a  Riemannian manifold of dimension $n$ without boundary. Fix $k\in\nn$ be such that $2\leq 2k<n$ , $L>0$ and $\delta>0$ and $x_0\in M$. Let $P$ be a differential operator of type  $O_{k,L}(B_\delta(x_0))$ (see Definition \ref{def:P:loc}). Let $u\in H_{2k}^p(B_\delta(x_0))$ and $f\in L^p(B_\delta(x_0))$, $p\in (1,+\infty)$ be such that $Pu=f$. Then for all $r<\delta$ and for all $q>1$, we have that
$$\Vert u\Vert_{H_{2k}^p(B_r(x_0))}\leq C(M,g,k,L,p,q,\delta,r)\left(\Vert f\Vert_{L^p(B_\delta(x_0))}+\Vert u\Vert_{L^q(B_\delta(x_0))}\right)$$
where $C(M,g,k,L,p,q,\delta,r)$ depends only on $(M,g)$, $k$, $L$, $p$, $q$, $\delta$ and $r$. The same conclusion holds if the $0^{th}-$order potential $A^{(0)}$ is only in $L^\infty(B_\delta(x_0))$ with $\Vert A^{(0)}\Vert_{L^\infty(B_\delta(x_0))}\leq L$.
\end{theorem}
\begin{theorem}\label{th:2bis} Let $(M,g)$ be a  Riemannian manifold of dimension $n$ without boundary. Fix $k\in\nn$ be such that $2\leq 2k<n$ , $L>0$ and $\delta>0$ and $x_0\in M$. Let $P$ be a differential operator of type  $O_{k,L}(B_\delta(x_0))$ (see Definition \ref{def:P:loc}). We fix $\theta\in (0,1)$. We assume that $A^{(i)}$ in Definition \ref{def:P:loc} has $C^{i,\theta}$-regularity and that $\Vert A^{(i)}\Vert_{C^{i,\theta}}\leq L$ for all $i=0,...,k-1$. Let $u\in H_{2k}^p(B_\delta(x_0))$ and $f\in C^{0,\theta}(B_\delta(x_0))$, $p\in (1,+\infty)$ and $\theta\in (0,1)$ be such that $Pu=f$. Then $u\in C^{2k,\theta}(B_\delta(x_0))$ and for all $r<\delta$ and for all $q>1$, we have that
$$\Vert u\Vert_{C^{2k,\theta}(B_r(x_0))}\leq C(M,g,k,L,q,\theta,\delta,r)\left(\Vert f\Vert_{C^{0,\theta}(B_\delta(x_0))}+\Vert u\Vert_{L^q(B_\delta(x_0))}\right)$$
where $C(M,g,k,L,q,\theta,\delta,r)$ depends only on $(M,g)$, $k$, $L$,  $q$, $\theta$, $\delta$ and $r$.
\end{theorem}

\begin{theorem}\label{th:3} Let $(M,g)$ be a compact Riemannian manifold of dimension $n$ without boundary. Fix $k\in\nn$ be such that $2\leq 2k<n$  and $L>0$. Let $P$ be a differential operator of type  $O_{k,L}$  (see Definition \ref{def:p}) and fix $p\in (1,+\infty)$. Then for all $f\in L^p(M)$, there exists $u\in H_{2k}^p(M)$ unique such that $Pu=f$. Moreover, 
$$\Vert u\Vert_{H_{2k}^p(M)}\leq C(M,g,k,L,p ) \Vert f\Vert_{L^p(M)}$$
where $C(M,g,k,L,p)$ depends only on $(M,g)$, $k$, $L$ and $p$. The same conclusion holds if the $0^{th}-$order potential $A^{(0)}$ is only in $L^\infty(M)$ with $\Vert A^{(0)}\Vert_{L^\infty(M)}\leq L$.
\end{theorem}

\begin{theorem}\label{th:4} Let $(M,g)$ be a  Riemannian manifold of dimension $n$ without boundary. Let $x_0\in M$ and $\delta<i_g(M)$ (the injectivity radius) so that $B:=B_\delta(x_0)\subset M$. Fix $k\in\nn$ be such that $2\leq 2k<n$  and $L>0$. Let $P$ be a differential operator such that
\begin{itemize}
\item[(i)] $P:=\Delta_g^k+\sum_{j=0}^{k-1} (-1)^j\nabla^j(B_j\nabla^j)$, where $B_j$ is a symmetric $(0, 2j)-$tensor for all $j=0,...,k-1$, 
\item[(ii)] $B_j\in C^j$ and $\Vert B_j\Vert_{C^{j}(B_\delta(x_0))}\leq L$ for all $j=0,...,k-1$
\item[(iii)] $P$ is symmetric in the $L^2-$sense
\item[(iv)] $P$ is coercive in the sense that  
\end{itemize}
$$\int_M uPu\, dv_g\geq\frac{1}{L}\Vert u\Vert_{H_{2k}^2}^2\hbox{ for all }u\in H_{k,0}^2(B).$$
Fix $p\in (1,+\infty)$. Then for all $f\in L^p(B)$, there exists $u\in H_{2k}^p(B)\cap H_{k,0}^p(B)$ unique such that $Pu=f$. Moreover, 
$$\Vert u\Vert_{H_{2k}^p(B)}\leq C(M,g,x_0,\delta,k,L,p ) \Vert f\Vert_{L^p(B)}$$
where $C(M,g,x_0,\delta,k,L,p )$ depends only on $(M,g)$, $x_0$, $\delta$, $k$, $L$ and $p$.
\end{theorem}

\end{document}